\documentclass[11pt,letterpaper]{amsart}
\usepackage{amsmath}
\usepackage{amsthm}
\usepackage{bbm}
\usepackage{amsfonts,amssymb,bm}
\usepackage{fancyhdr}
\usepackage{mathrsfs}
\usepackage{appendix}
\usepackage{mathtools}
\usepackage{graphicx}
\usepackage[all]{xy}
\usepackage{color}
\usepackage{enumerate}
\usepackage{comment}
\usepackage{stmaryrd} 

\usepackage[pagebackref]{hyperref}
\usepackage{tikz-cd}
\usepackage[alphabetic]{amsrefs}

\usepackage{geometry} 

\usepackage{todonotes}
 
\usepackage{float} 
\usepackage{amsthm}

\newtheorem{thm}{Theorem}[section]
\newtheorem{defi}[thm]{Definition}
\newtheorem{lem}[thm]{Lemma}
\newtheorem{cor}[thm]{Corollary}
\newtheorem{prop}[thm]{Proposition}
\newtheorem{claim}[thm]{Claim}
\newtheorem{ex}[thm]{Example}
\newtheorem{rmk}[thm]{Remark}

\theoremstyle{definition}

\theoremstyle{remark}

\usepackage{hyperref}

\def\Aut{{\rm Aut}}

\def\ord{{\rm ord}}
\def\min{{\rm min}}
\def\max{{\rm max}}
\def\inf{{\rm inf}}
\def\sup{{\rm sup}}
\def\lim{{\rm lim}}
\def\limsup{{\rm lim\,sup}}
\def\liminf{{\rm lim\,inf}}
\def\dif{{\rm d}}
\def\interior{{\rm int}}
\def\Supp{{\rm Supp}}

\def\pr{{\rm pr}}
\def\wt{{\rm wt}}

\def\nv{\mathfrak{v}}
\def\Rees{{\rm Rees}}
\def\Proj{{\rm Proj}}
\def\Spec{{\rm Spec}}
\def\Diff{{\rm Diff}}
\def\supp{{\rm supp}}
\def\span{{\rm span}}
\def\triv{{\rm triv}}
\def\Bl{{\rm Bl}}
\def\log{{\rm log}}
\def\vol{{\rm vol}}

\def\Val{{\rm Val}}

\def\lct{{\rm lct}}
\def\DH{{\rm DH}}
\def\LE{{\rm LE}}

\def\Ex{{\rm Ex}}
\def\mult{{\rm mult}}
\def\Fut{{\rm Fut}}

\def\red{{\rm red}}
\def\div{{\rm div}}
\def\QM{{\rm QM}}
\def\LC{{\rm LC}}

\def\Gr{{\rm Gr}}

\def\Cl{{\rm Cl}}
\def\Conv{{\rm Conv}}
\def\reeb{{\mathbf{t}^+_\mathbb{\IR}}}
\def\rank{{\rm rank}}

\def\BC{\mathbf{C}}
\def\BP{\mathbf{P}}

\def\BO{\mathbf{O}}
\def\BV{\mathbf{V}}

\def\BD{\mathbf{D}}

\def\BH{\mathbf{H}}
\def\BF{\mathbf{F}}
\def\BW{\mathbf{W}}

\def\B0{\mathbf{0}}

\def\By{\mathbf{y}}

\def\Bin{\mathbf{in}}

\def\tS{{\widetilde{S}}}
\def\tC{\tilde{C}}

\def\bg{\bar{g}}

\def\bphi{\bar{\phi}}
\def\bpsi{\bar{\psi}}

\def\opi{{\overline{\pi}}}
\def\oCX{{\overline{\CX}}}

\def\oCL{{\overline{\CL}}}
\def\oCD{{\overline{\CD}}}
\def\oX{{\overline{X}}}
\def\oY{{\overline{Y}}}
\def\oZ{{\overline{Z}}}

\def\oE{{\overline{E}}}

\def\oL{{\overline{L}}}
\def\oPhi{{\overline{\Phi}}}

\def\oLambda{{\overline{\Lambda}}}

\def\toCX{{\mathring{\CX}}}
\def\toCD{{\mathring{\CD}}}
\def\toCZ{{\mathring{\CZ}}}
\def\toIA{{\mathring{\IA}^1}}

\newcommand{\hvol}{{\widehat{\rm vol}}}

\newcommand{\IA}{{\mathbb A}}

\newcommand{\IG}{{\mathbb G}}

\newcommand{\Ik}{{\mathbbm k}}

\newcommand{\IN}{{\mathbb N}}
 
\newcommand{\IP}{{\mathbb P}} 
\newcommand{\IQ}{{\mathbb Q}} 
\newcommand{\IR}{{\mathbb R}}

\newcommand{\IT}{{\mathbb T}}

\newcommand{\IZ}{{\mathbb Z}}

\newcommand{\CC}{{\mathcal C}}
\newcommand{\CD}{{\mathcal D}} 

\newcommand{\CF}{{\mathcal F}}
\newcommand{\CG}{{\mathcal G}}

\newcommand{\CL}{{\mathcal L}}

\newcommand{\CO}{{\mathcal O}}

\newcommand{\CR}{{\mathcal R}}

\newcommand{\CW}{{\mathcal W}}
\newcommand{\CX}{{\mathcal X}}
\newcommand{\CY}{{\mathcal Y}} 
\newcommand{\CZ}{{\mathcal Z}}

\newcommand{\cF}{\mathcal{F}}
\newcommand{\cG}{\mathcal{G}}

\newcommand{\cL}{\mathcal{L}}

\newcommand{\cO}{\mathcal{O}}

\newcommand{\bN}{\mathbb{N}}

\newcommand{\bQ}{\mathbb{Q}}

\newcommand{\bR}{\mathbb{R}}

\newcommand{\bZ}{\mathbb{Z}}

\newcommand{\fa}{\mathfrak{a}}

\newcommand{\fm}{\mathfrak{m}}

\newcommand{\fv}{\mathfrak{v}}

\newcommand{\seq}{\subseteq}
\newcommand{\la}{\langle}
\newcommand{\ra}{\rangle}

\newcommand{\bu}{\bullet}
\newcommand{\lam}{\lambda}
\newcommand{\D}{\Delta}

\newcommand{\vep}{\varepsilon}
\newcommand{\Coeff}{\mathrm{Coeff}}
\newcommand{\Center}{\mathrm{Center}}

\begin{document}
\title{Stable Degenerations of log Fano Fibration Germs}

\author{Jiyuan Han}
\address{Jiyuan Han, Institute for Theoretical Sciences, Westlake University,
No.600 Dunyu Road, Hangzhou, 310030, China}
\curraddr{}
\email{hanjiyuan@westlake.edu.cn}
\thanks{}
\date{}
\dedicatory{}

\author{Minghao Miao}
\address{Minghao Miao: School of Mathematics, Nanjing University, Nanjing 210093, China}
\curraddr{}
\email{minghao.miao@smail.nju.edu.cn}
\thanks{}
\date{}
\dedicatory{}

\author{Lu Qi}
\address{School of Mathematical Sciences, East China Normal University, Shanghai 200241, China}
\email{lqi@math.ecnu.edu.cn}

\author{Linsheng Wang}
\address{Shanghai Center for Mathematical Sciences, Fudan University, Shanghai, 200438, China}
\curraddr{}
\email{linsheng\_wang@fudan.edu.cn}
\thanks{}
\date{}
\dedicatory{}

\author{Tong Zhang}
\address{School of Mathematical Sciences, Ministry of Education Key Laboratory of Mathematics and Engineering Applications \& Shanghai Key Laboratory of PMMP, East China Normal University, Shanghai 200241, China}
\curraddr{}
\email{tzhang@math.ecnu.edu.cn}
\thanks{}
\date{}
\dedicatory{}

\begin{abstract}
We prove the stable degeneration conjecture of log Fano fibration germs formulated in \cite{SZ24}. Precisely, we introduce the $\mathbf{H}$-invariant for filtrations over a log Fano fibration germ, and show that there exists a unique quasi-monomial valuation $v_0$ minimizing the $\mathbf{H}$-invariant. Moreover, we prove that the associated graded ring of $v_0$ is finitely generated and induces a special degeneration to a K-semistable polarized log Fano fibration germ, which further admits a unique K-polystable special degeneration. 
\end{abstract}
\maketitle

\tableofcontents

Throughout, we work over an algebraically closed field $\Ik$ of characteristic $0$. 

\section{Introduction} 
The existence of canonical metrics on K\"ahler manifolds is a fundamental problem in complex geometry. A deep philosophy, as reflected in general Yau-Tian-Donaldson type conjectures, is that such existence is closely related to certain algebraic stability conditions of the manifold.
For Fano manifolds, such a stability condition, known as K-stability, was introduced by Tian \cite{Tia97}, which was later reformulated in algebraic terms in \cite{Don02}. 

In the past decade, after the work \cite{LX14}, 
higher dimensional algebraic geometry, especially the minimal model program, provides a powerful tool for the algebro-geometric study of K-stability for Fano varieties and klt singularities. 
The recent breakthrough \cite{LXZ22} settled an algebraic version of the Yau-Tian-Donaldson conjecture; namely, K-stability is equivalent to uniform K-stability for log Fano pairs. 
Though many log Fano pairs (or their local analogs, log Fano cone singularities) are not K-stable, 
a remarkable phenomenon is that a log Fano pair or a klt singularity always admits a special degeneration to an object with certain K-stability. Such a degeneration 
usually arises from a minimizing process of some canonically defined functional. 

In the global setting, the Hamilton-Tian conjecture \cite[Conjecture 9.1]{Tia97} predicted that a normalized K\"ahler-Ricci flow on a Fano manifold $X$ will converge in the Cheeger-Gromov-Hausdorff sense to a K\"ahler-Ricci soliton space $X_\infty$. This analytic problem is now quite well understood. See \cites{TZ-regularity,Bam-convergence,CW-weak-compactness} for the solution to the conjecture, and \cites{CSW-KKK,DS20} for the $2$-step degeneration process using the archimedean $\BH$-functional.
Based on the work of the first named author and Chi Li \cites{HL23,HL20}, it was proved recently in \cite{BLXZ23} that the unique minimizer of the non-archimedean $\BH$-functional for a log Fano pair gives rise to a weighted K-semistable degeneration, which further admits a unique weighted K-polystable degeneration. 
This can be viewed as an algebraic version of the Hamilton-Tian conjecture. 
We refer to the book \cite{Xu-kstabilitybook} for a comprehensive account of the algebraic theory of K-stability in the global setting.

In the local setting, the 2-step degeneration process was conjectured by \cite{DS15} in the study of the Ricci-flat metric tangent cone, which is the Gromov-Hausdorff limit of pointed K\"ahler-Einstein Fano manifolds. 
Based on the work \cite{MSY08} on Sasaki geometry, Chi Li \cite{Li18} introduced the normalized volume functional $\hvol$ for a klt singularity, which initiated the local study of K-stability. 
The local theory centers around the minimization of normalized volumes; the key conjecture, now called the Stable Degeneration Conjecture \cites{Li18,LX18}, provides a purely algebraic approach toward Donaldson-Sun's algebraicity conjecture.
The Stable Degeneration Conjecture is now known by intensive works including \cites{Blu18,LX18,LWX18,Xu19,XZ20,XZ-sdsing} (see also \cites{LX20,BLQ-convexity}). 
Readers interested in this local theory are referred to the guide \cite{LLX18}, as well as to the excellent and up-to-date survey \cite{Zhu-survey}.


Recently, from a more differential geometric point of view, Sun and Zhang \cite{SZ24} studied the K-stability related to \emph{Fano fibration germs}, which unifies both notions of Fano varieties and klt singularities in a natural way. 
They also predicted that the algebro-geometric study plays a role in the theory of (non-compact) K\"ahler-Ricci shrinking solitons, which incorporates K\"ahler-Einstein metrics on Fano manifolds and Ricci-flat K\"ahler cone metrics on log Fano cone singularities.

Recall that a {\it (log) Fano fibration germ} is a collection of data $f: (X,\D)\to Z\ni o$, where $(X,\D)$ is a klt pair, $f: X\to Z$ is a projective morphism satisfying $f_*\CO_X = \CO_Z$ (i.e., $f$ is a {\it fibration}), $Z$ is affine, and $o\in Z$ is a closed point such that $-(K_X+\D)$ is $f$-ample. 
Log Fano fibration germs provide a natural interpolation between log Fano pairs and klt singularities: if $(X,\D)$ is a (projective) log Fano pair, then $(X,\D)\to\Spec \Ik$ is a Fano fibration germ; if $x\in (X,\D)$ is a klt singularity, then $\mathrm{id}:X\to X\ni x$ is also a Fano fibration germ. 
We remark that Fano fibration germs also appear naturally in other aspects of birational geometry. For example, if $f:(X,\D)\to Z$ is an extremal contraction in the minimal model program, then for any closed point $o\in Z$, $f:(X,\D)\to Z\ni o$ is a Fano fibration germ. This perspective has been taken in \cite{HQZ25}.

\medskip

The primary goal of this paper is to study the minimization of $\BH$-invariants for log Fano fibration germs, and to provide a unified version of the results in both global and local settings. 

More concretely, generalizing the filtration theory for log Fano varieties \cites{BHJ17,BJ20} and that for klt singularities \cite{BLQ-convexity}, we consider {\it filtrations} $\cF$ on the anti-canonical section ring $R=R(X,\D)=\oplus_m H^0(X,-m(K_X+\D))$ of a log Fano fibration germ $(X,\D)\to Z\ni o$. 
We study the minimizer of the (non-Archimedean) {\it $\BH$-functional} on the space of filtrations centered at $o\in Z$, which is defined by 
\begin{eqnarray*}
    \BH(\cF)\coloneqq \mu(\cF)-\tS(\cF),
\end{eqnarray*}
where $\mu(\cF)$ is the \emph{log canonical slope} of $\cF$, and $\tS(\cF)\coloneqq-\log\int_\bR e^{-x}\DH_\cF(dx)$ is a twisted version of the usual $S$-invariant;
see Definitions \ref{Definition: log canonical slope} and \ref{Definition: H invariant} for details.
This unifies the $\BH$-invariant of Fano varieties \cites{TZZZ13,DS20,HL20} and the local volume of klt singularities \cite{Li18}, and is closely related to the weighted volume of Fano fibration germs introduced by \cite{SZ24}. 

Our main result is the following theorem regarding the stable degeneration of log Fano fibration germs, 
which verifies \cite[Conjecture 6.4]{SZ24}.

\begin{thm}\label{Theorem. Intro. stable degeneration of Fano fibration germ}
Let $(X,\D)\to Z\ni o$ be a log Fano fibration germ. Then the following statements hold. 
\begin{enumerate}[{\rm \quad (1)}]
\item {\rm (Existence of quasi-monomial minimizer)}. 
There exists a quasi-monomial valuation $v_0\in\Val_{X,o}^*$ satisfying $\BH(v_0) = \BH(X,\D)$. 
\item {\rm (Uniqueness of minimizer)}. 
For any valuation $v_1\in\Val_{X,o}^*$ satisfying $\BH(v_1) = \BH(X,\D)$, we have $v_1=v_0$. 
\item {\rm (Finite generation)}. The associated graded ring $\Gr_{v_0} R$ is finitely generated. 
\item {\rm (Semistable degeneration)}. The special degeneration 
$$(X_{0},\D_{0},\xi_0)\to Z_{0}  \ni o$$ 
of $(X,\D)\to Z\ni o$ induced by $v_0$ is K-semistable. 
\item {\rm (Polystable degeneration)}. There exists a unique special degeneration $$(X_{p},\D_{p},\xi_0)\to Z_{p}\ni o$$ of $(X_{0},\D_{0},\xi_0)\to Z_{0}\ni o$ which is K-polystable. 
\end{enumerate}
\end{thm}

The theorem specializes to \cite[Theorem 1.2]{BLXZ23} when $Z=o$ is a closed point (global case), and to \cite[Theorem 1.2]{XZ-sdsing} when $X\cong Z$ (local case).


\medskip
\noindent
{\bf Outline of the proof}

{\it Existence and uniqueness of the $\BH$-minimizer}. 
The existence of $\BH$-minimizers can be shown by an approximation process similar to \cite{Xu19}, which is inspired by \cite{LX20}. 
Precisely, we find a sequence of divisorial valuations $\{v_i\}$ such that $\BH(v_i)\to \BH(X,\D)$ as $i\to \infty$, 
and show that a subsequence of $v_i$ has a quasi-monomial limit $v_0$, which minimizes $\BH(X,\D)$. In order to prove the convergence, we first need to establish the boundedness of these valuations $v_i$ in some sense. Indeed, the valuations $v_i$ we found are log canonical places of some $\IQ$-complements of $(X,\D)\to Z$. By the boundedness of complements \cite{Bir19}, there exists $N\in \IN$ depending only on the dimension of $X$ and the coefficients of $\D$, such that each $v_i$ is a log canonical place of some $N$-complement. However, different from the global case, the space of $N$-complements (a subspace of $|-N(K_X+\D)|$) is not of finite type.

To overcome the difficulty, we need a proper estimate as in \cite[Theorem 4.1]{Li18}, which was used in \cite[Proposition 3.5]{Xu19} to establish the boundedness of $v_i$ in the local case. The proof of \cite[Theorem 4.1]{Li18} relies heavily on the rescaling invariance of the normalized volume function. However, the $\BH$-invariant is no longer homogeneous with respect to rescaling. The function 
$$ \BH(- \cdot v): \IR_{>0} \to \IR:  a \mapsto \BH(a\cdot v) $$
is strictly convex and admits a unique minimizer. So in the minimization problem, we may always fix the rescaling of $v$ such that the above function is minimized by $a=1$. With this assumption, we obtain an essential uniform estimate (Proposition \ref{Theorem. bounded log discrepancies}) 
$$A_{X,\D}(v) \le \dim X, $$
by using the theory of relative Okounkov bodies established in Section \ref{Section: Asymptotic invariants on Fano fibration germs}. Such a bound takes the place of rescaling $v$ with log discrepancy $1$ in the proof of \cite[Theorem 4.1]{Li18}, and finally we get the proper estimate Theorem \ref{Theorem. Properness} as desired. 

Now we can find a finite-dimensional $\Ik$-subspace $W\seq H^0(X,-N(K_X+\D))$ such that each $v_i$ is a log canonical place of some $N$-complement in $|W|$. Therefore, using the argument in \cite{Xu19}, 
we see that a subsequence of $v_i$ lies in a single quasi-monomial simplicial cone of the dual complex, thus admitting a limit $v_0$. 
Uniqueness follows from the strong convexity of $\BH$ along geodesics 
as in \cites{HL20,BLXZ23}.

{\it Finite generation}. 
For the finite generation property of $v_0$, we follow the strategy 
of \cite{BLXZ23}. We first introduce the $v_0$-weighted delta invariant $\delta(X,\D,v_0)$ for a Fano fibration germ $(X,\D)\to Z\ni o$ and $v_0\in \Val_{X,o}^*$, and then show that $v_0$ minimizes $\BH(X,\D)$ if and only if $\delta(X,\D,v_0)=1$, which is computed by $v_0$. 
Following the argument of \cite{LXZ22}, we then show that any minimizer of $\delta(X,\D,v_0)=1$ is a log canonical place of some {\it special} $\IQ$-complement. 
Finally, via the relative cone construction, the finite generation property follows from \cite{XZ-sdsing}.

{\it Two-step degeneration}. 
The associated graded ring of the minimizer $v_0$ induces a special degeneration $(X,\D)\to Z\ni o$ to $(X_0,\D_0,\xi_0)\to Z_0 \ni o$, whose K-semistability follows from an argument similar to that in \cite[Theorem 5.3]{HL20}. 
Moreover, since a nontrivial special degeneration increases the rank of the torus acting on the central fiber, 
there exists a K-polystable special degeneration $(X_p,\D_p,\xi_0)\to Z_p \ni o$ of $(X_0,\D_0,\xi_0)\to Z_0 \ni o$. The uniqueness of the K-polystable special degeneration follows from a $\Theta$-reductivity type result as in \cite{LWX18,ABHLX19}. 
This finishes the proof of the main theorem.

\medskip

The paper is organized as follows. 
In Section \ref{Section: Preliminary}, we introduce the space of filtrations on Fano fibrations, extending the space of valuations and test configurations in this setting.
In Section \ref{Section. Special divisors and valuations}, we study (weakly) special valuations over Fano fibrations, which play a central role in the study of special degenerations of Fano fibrations. 
In Section \ref{Section: Asymptotic invariants on Fano fibration germs}, we develop the theory for (unbounded) Okounkov bodies in the relative setting, which is used to define various asymptotic invariants for Fano fibration germs, including $\BH$-invariants and delta invariants. We also prove several properties of these invariants, especially the uniqueness of the $\BH$-minimizer. 
In Section \ref{Section. Existence of H-minimizers}, we prove the existence of the $\BH$-minimizer. Finally, in Section \ref{Section. K-stability and Stable degenerations}, we study the K-stability of polarized Fano fibration germs and show that the special degeneration induced by the $\BH$-minimizer is K-semistable. 

\medskip

\noindent {\bf Acknowledgments}. 
The authors thank Xiaowei Jiang, Chi Li, Chenyang Xu, and Ziquan Zhuang for fruitful discussions. The authors thank Harold Blum and Yuji Odaka for reading a preliminary version of this paper and providing helpful feedback as well as Song Sun for his kind comments. The authors also thank Kim Donghyeon for his comments on the largest weighted volume. 

LW is partially supported by the NKRD Program of China (\#2025YFA1018100). 
JH is partially supported by NSFC-12301059, XHD23A0101.
MM is supported by NSFC Grant 125B2003 and he would like to thank his advisor Gang Tian for his constant guidance and support. 
LQ is partially supported by the NSFC (No. 12501055) and the Shanghai Sailing program 24YF2709800.
TZ is partially supported by Science and Technology Commission of Shanghai Municipality (No. 22JC1400700, No. 22DZ2229014).

\section{Preliminaries}
\label{Section: Preliminary}

In this section, we collect some preliminary results that will be used throughout the article. We follow the standard terminology of \cites{KM98,Kol13}.

\subsection{Notation and conventions}

A \emph{variety} $X$ is a separated integral scheme of finite type over $\Ik$. We use $\Ik(X)$ to denote the function field of $X$. All varieties are assumed to be quasi-projective unless otherwise specified. A \emph{pair} consists of a normal variety $X$ and a $\bQ$-divisor $\Delta\ge 0$ on $X$. A pair $(X,\Delta)$ is a \emph{log pair} if $K_X+\Delta$ is $\bQ$-Cartier. 

A pair $(X,\Delta)$ is called \emph{log smooth} if $X$ is smooth and $\Supp(\Delta)$ is a simple normal crossing (SNC) divisor. A \emph{stratum} of a log smooth pair $(X,\Delta=\sum_{i\in I} b_i B_i)$, where $B_i$ are the irreducible components of $\Delta$, is defined to be a connected component of $B_J:=\cap_{j\in J} B_j$ for some $J\subseteq I$, where, by convention, $B_\varnothing = X$.

A {\it model} $\pi: (Y,E)\to X$ of a pair $(X,\D)$ consists of a proper birational morphism $\pi : Y\to X$ and a reduced divisor $E$ on $Y$. It is {\it log smooth} if $(Y, \Supp(E+\Ex(\pi)+\pi^{-1}_*\D))$ is SNC, and is {\it toroidal} if it is locally a quotient of SNC pairs by abelian groups. 

Let $f\colon X\to Z$ be a morphism and $z\in Z$ a point. We denote by $X_z:=f^{-1}(z)$ the scheme-theoretic fiber over $z$.

The notation of \emph{terminal}, \emph{klt}, \emph{plt}, \emph{dlt}, and \emph{log canonical} pairs is defined as in \cite[Definitions 2.34 and 2.37]{KM98}. 

A {\it totally ordered abelian group} $(\Gamma,\le)$ is an abelian group $\Gamma$ together with a total order $\le$ on $\Gamma$ respecting addition: for any $a,b,c\in \Gamma$, $a\le b$ implies $a+c\le b+c$. For such a $(\Gamma,\le)$, we can add an element $\infty$ to $\Gamma$ such that $a<\infty$ and define $a+\infty\coloneqq \infty$ for any $a\in\Gamma$. Then $(\Gamma\cup\{\infty\},\le)$ is still a totally ordered abelian group. In this paper, we mainly consider (subgroups of) $\bR$ with the usual order and $\bZ^n$ with certain total order.

For a $\Ik$-vector space $V$, write $V^*\coloneqq V\setminus\{0\}$. 

\subsection{Fibrations}

A {\it fibration} is a projective surjective morphism $f: X\to Z$ between normal varieties such that $f_*\CO_X = \CO_Z$. 
A \emph{fibration germ} $f:X\to Z\ni o$ consists of a fibration $f:X\to Z$ where $Z$ is affine, and a closed point $o\in Z$. If $f:X\to Z\ni o$ is a fibration germ and $(X,\Delta)$ is a log pair, slightly abusing the notation, we also say that $f:(X,\Delta)\to Z\ni o$ is a fibration germ. 
A fibration $(X,\D)\to Z$ is \emph{klt} if the total space $(X,\Delta)$ has klt singularities. 
If a fibration $f: X\to Z$ is not a birational morphism (that is, $\dim X>\dim Z$), we say that the filtration is \emph{of fiber type}. 

Let $L$ be an $f$-ample $\bQ$-Cartier Weil divisor on $X$. Fix $l_0\in \IN$ such that $l_0L$ is Cartier and $f$-very ample. Denote the relative section ring by $R =R(X,L)=\oplus_{m\in l_0 \IN} R_m$, where $R_m = f_*\CO_X(mL) \cong H^0(X,mL)$ is the $R_0$-module of relative sections of $mL$ over $Z$. 
We have an isomorphism
\begin{eqnarray*} 
    X\cong \Proj_{R_0}R. 
\end{eqnarray*}
A {\it log Fano fibration germ} $f:(X, \D)\to Z\ni o$ is a klt fibration germ such that $-(K_X+\D)$ is $f$-ample. A $\IQ$-divisor $0\le \Gamma\sim_{\IQ,Z} -(K_X+\D)$ is called a {\it $\IQ$-complement} of a log Fano fibration germ $f:(X,\D)\to Z\ni o$ if $(X,\D+\Gamma)$ is log canonical, and there exists a log canonical center $W$ such that $W\subset f^{-1}(o)$. 


\subsection{Valuations}
We recall basic facts about valuations from \cite{KK12}. Let $(\Gamma,\le)$ be a totally ordered abelian group.

A {\it $\Gamma$-prevaluation} on a $\Ik$-vector space $V$ is a function $v: V^*\to \Gamma$ such that 
\begin{enumerate}
\item $v(s+t)\ge \min(v(s),v(t))$ for any $s,t\in V$ with $s+t\ne 0$; 

\item $v(c s) = v(s)$ for any $s\in V^*, c\in\Ik^*$. 
\end{enumerate}
Denote by 
\begin{eqnarray*}
\CF_v^\lam V = \{s\in V\mid v(s)\ge \lam\}, \quad \CF_v^{>\lam} V = \{s\in V\mid v(s)> \lam\}
\end{eqnarray*}
for any $\lam \in \Gamma$, 
which are subspaces of $V$ by the definition of a prevaluation. We see that $\CF_v^{\lam}V\supseteq \CF_v^{\mu}V$ for any $\lam\le \mu$.  Hence we get a descending sequence of $\Ik$-vector spaces $\CF_v V=\{\CF_v^{\lam}V\}_{\lam \in \Gamma}$ (satisfying $\cap_{\lam \in\Gamma} \CF_v^{\lam}V = \{0\}$ when $\Gamma$ is infinite), which is called a {\it filtration} on $V$. The quotient vector space 
\begin{eqnarray*}
\Gr_v^\lam V = \CF_v^\lam V/ \CF_v^{>\lam} V
\end{eqnarray*} 
is called the {\it leaf} of $v$ at $\lam \in \Gamma$. A prevaluation $v$ is said to have {\it one-dimensional leaves} if for every $\lam\in \Gamma$, we have $\dim \Gr_v^\lam V \le 1$. 

Assume that $\Gamma$ is discrete, then any filtration $\CF\, V= \{\CF^{\lam}V\}_{\lam \in \Gamma}$ on $V$ has an induced prevaluation: 
\begin{eqnarray*}
\ord_\CF(s) := \max\{\lam\in\Gamma \mid s\in \CF^\lam V \}. 
\end{eqnarray*}
Let $W\seq V$ be a subspace. A prevaluation $v$ on $V$ induces filtrations on $W$ and $V/W$ respectively: 
\begin{eqnarray}
\label{Eqnarray. induced filtration on sub & quotient spaces}
\CF_v^\lam W := \CF_v^\lam V \cap W, \quad 
\CF_{\bar{v}}^\lam (V/W) := \CF_v^\lam V / \CF_v^\lam W. 
\end{eqnarray} 
The induced filtration $\CF_v^\lam W$ corresponds to the prevaluation $v|_W$, and the induced filtration $\CF_{\bar{v}}^\lam (V/W)$ corresponds to the following prevaluation $\bar{v}$ on $V/W$: 
\begin{eqnarray*}
\bar{v}(\bar{s}) = \max\{v(s+t)\mid t\in W\}, 
\end{eqnarray*} 
where $\bar{s}=s+W\seq V$ and $s\in V\setminus W$. 

\begin{lem}
\label{Lemma. v(V quot W) = v(V) minus v(W)}
Let $W\seq V$ be $\Ik$-vector spaces, and $(\Gamma,<)$ be a discrete totally ordered set. Then for any $\Gamma$-prevaluation $v$ with one-dimensional leaves on $V$, we have 
\begin{eqnarray*}
\bar{v}((V/W)^*) = v(V^*) \setminus v(W^*). 
\end{eqnarray*} 
\end{lem}

\begin{proof}
By definition, we have $v(W^*)\seq v(V^*)$, $\bar{v}((V/W)^*) \seq v(V^*)$ and their union is the whole $v(V^*)$. We only need to show that $v(W^*)\cap\bar{v}((V/W)^*) =\varnothing$. For any $s\in V\setminus W$ and $t\in W^*$ with $v(s)=v(t)$, we have $v(s+ct)>v(s)$ for some $c\in \Ik^*$ by the one-dimensional leaf condition. Repeating this argument, we see that $\bar{v}(\bar{s}) \notin v(W^*)$. The proof is finished. 
\end{proof}


Let $A$ be a $\Ik$-algebra. A prevaluation $v: A\setminus \{0\} \to \Gamma$ is a {\it valuation} if $v(fg)=v(f)+v(g)$ for any $f,g\in A\setminus \{0\}$. By definition, the image $v(A\setminus \{0\}) \seq \Gamma$ is a sub-semigroup. The valuation $v$ is called {\it faithful} if $v(A\setminus \{0\})=\Gamma$.

For any $\Gamma$-valuation $v$ on $\Ik(X)$, we have the {\it valuation ring} $(\CO_v,\fm_v)$ of $v$ defined by 
\begin{eqnarray*}
\CO_v = \{f\in \Ik(X)\mid v(f)\ge 0\}, \quad
\fm_v = \{f\in \Ik(X)\mid v(f)> 0\}, 
\end{eqnarray*}
which is a local ring with maximal ideal $\fm_v$. A $\Gamma$-valuation $v$ is said to be {\it centered} on $X$ if there exists a scheme-theoretic point $\xi\in X$ such that $\CO_{X,\xi}\subset\cO_v$ and $\fm_\xi=\fm_v\cap \cO_{X,\xi}$. 
Since $X$ is separated, the point $\xi$ is unique, denoted by $c_X(v):=\xi$, and is called the {\it center} of $v$ on $X$. We also denote by $C_X(v)=\overline{c_X(v)} \seq X$ the Zariski closure of $c_X(v)$ in $X$. If $X$ is proper, then any $\IR$-valuation on $\Ik(X)$ is centered on $X$. 

\begin{rmk}\rm We are interested in the case of log Fano fibration germs $f: (X,\D)\to Z\ni o$, where $X$ is not proper. Hence, a valuation on $\Ik(X)$ may not be centered on $X$. We will always consider those valuations centered on $f^{-1}(o)$. \end{rmk}

We are mainly interested in $\IR$-valuations and $\IZ^n$-valuations on $\Ik(X)$, where $\IZ^n$ is equipped with some total order respecting addition. We will denote $\IR$-valuations on $\Ik(X)$ by $v,w,\cdots$ and a $\IZ^n$-valuation by $\fv$. 

We say that $v$ is a {\it valuation} on $X$ if it is an $\IR$-valuation on $\Ik(X)$ centered on $X$, and denote by $\Val_X$ the set of valuations on $X$. Let $X$ be a normal variety and $\D$ be a $\IQ$-divisor such that $K_X+\D$ is $\IQ$-Cartier. For any valuation $v$ on $X$, we can define the {\it log discrepancy} $A_{X,\D}(v)$ by \cite[Section 5]{JM12}. Denote by  $\Val_{X}^*$ the subset of valuations $v\in\Val_{X}$ with $A_{X,\D}(v)<+\infty$ for some $\IQ$-divisor $\D$ such that $K_X+\D$ is $\IQ$-Cartier, which does not depend on the choice of $\D$ by \emph{loc. cit.}. 

Let $f: (X,\D)\to Z\ni o$ be a fibration germ. 
Denote by 
\[
    \Val_{X,o}\coloneqq \{v\in\Val_X\mid c_X(v)\in f^{-1}(o)\}, \quad\text{and} \quad
    \Val_{X,o}^*\coloneqq \Val_X^*\cap \Val_{X,o}. 
\]
Since $\Ik(Z)\seq \Ik(X)$, any valuation $v$ on $\Ik(X)$ naturally restricts to a valuation $v|_{\Ik(Z)}$ on $\Ik(Z)$. Moreover, if $v$ is centered on $f^{-1}(o)\seq X$, we see that $v|_{\Ik(Z)}$ is centered on $o\in Z$. Hence restriction gives a natural map $\Val_{X,o} \to \Val_{Z,o}$, and we sometimes denote $v|_{\Ik(Z)}$ by $v$ for simplicity. 

\begin{defi}[Quasi-monomial valuations]\rm
\label{Definition. Quasi-monomial valuations}
Let $(Y,E=E_1+\cdots+E_r)\to X$ be a log smooth model and $\alpha=(\alpha_1,\cdots, \alpha_r) \in \IR_{\ge 0}^r$. We define the valuation $v_\alpha$ by 
\begin{eqnarray*}
\label{Eqnarray. quasi-monomial valuations}
v_\alpha(s) = \min\{ \sum_{1\le i\le r} \alpha_i \beta_i\mid s_\beta\ne 0 \}
\end{eqnarray*}
for any $s = \sum_\beta s_\beta z^\beta \in \hat{\CO}_{X,\eta}$, where $\eta$ is the generic point of some irreducible component of $\cap_{1\le i\le r} E_i$ and $z_1,\cdots, z_r \in \hat{\CO}_{X,\eta}$ are local functions such that $E_i = (z_i=0)$. The valuation $v_\alpha$ is called a {\it quasi-monomial valuation} on $X$. Denote by 
\begin{eqnarray*}
\QM_\eta(Y,E) = \{v_\alpha\in \Val_{X}\mid \alpha \in \IR_{\ge 0}^r \}, 
\end{eqnarray*}
which is called a {\it quasi-monomial simplicial cone} over $X$. We also denote by $\QM(Y,E) = \cup_\eta \QM_\eta(Y,E)$, where $\eta$ runs over all the generic points of strata of $E$. 
\end{defi}

Let $(X,\D^+)$ be a log canonical pair. A valuation $v\in\Val_{X}$ on $X$ is called a {\it log canonical place} of $(X,\D^+)$ if $A_{X,\D^+}(v)=0$. Denote by 
\begin{eqnarray*}
\LC(X,\D^+) := \{v \in \Val_{X}^* \mid A_{X,\D^+}(v) = 0 \}.  
\end{eqnarray*}
Let $(Y,E) \to (X,\D^+)$ be a log resolution and $E^+ \seq E$ be the sum of components $F$ satisfying $A_{X,\D^+}(F)=0$. By \cite[Lemma 1.39]{Xu-kstabilitybook}, we have 
\begin{eqnarray*} \LC(X,\D^+) = \QM(Y,E^+). \end{eqnarray*}
In particular, any log canonical place of a log canonical pair is quasi-monomial.

\begin{defi}[Good valuations]\rm 
Let $f:(X,\D)\to Z\ni o$ be a fibration germ. 
We say that a $\IZ^n$-valuation $\fv$ on $\Ik(X)$ is a {\it good valuation} on $f:(X,\D)\to Z\ni o$ if 
\begin{enumerate}
\item its center $\xi$ is contained in $f^{-1}(o)$; 
\item it is faithful with one-dimensional leaves; 
\item it is {\it (linearly) bounded}, that is, there is a constant $C>0$ and a linear function $\ell: \IR^n \to \IR$ such that 
$$\ell(\fv(s)) \le C\cdot \ord_{\fm_\xi}(s) $$
for any $s\in \CO_{X,\xi}$. 
\end{enumerate}
\end{defi}

\begin{lem}
\label{Lemma. good valuation associated to a quasi-monomial valuation}
Let $f:(X,\D)\to Z\ni o$ be a klt fibration germ and $v\in \Val_{X,o}^*$ be a quasi-monomial valuation. Then there exist $\alpha\in \IR_{\ge 0}^n$ and a good valuation $\nv$ on $f:(X,\D)\to Z\ni o$ such that $v(s) = \la\alpha,\nv(s)\ra$ for any $s\in \Ik(X)^*$. 
\end{lem}

\begin{proof}
The proof is essentially the same as that of \cite[Lemma A.2]{BLQ-convexity}. We include a sketch here for the reader's convenience. By definition, we may assume that there exists a log smooth model $\mu:(Y,E=E_1+\cdots+E_n)\to X$ and a closed point $p\in \cap E_i$ such that $x\coloneqq\mu(p)\in f^{-1}(o)$ and $v = v_\alpha \in \QM_p(Y,E)$ for some $\alpha = (\alpha_1,\cdots,\alpha_r,0,\cdots, 0)\in\bR_{>0}^n$, where $\alpha_1,\cdots,\alpha_r$ are $\bQ$-linearly independent. Choose $\alpha_{r+1}',\ldots,\alpha_n'\in\bR_{>0}$ that are $\bQ$-linearly independent and let $\alpha'\coloneqq(0,\ldots,0,\alpha_{r+1}',\cdots,\alpha_n')$. Then we can define a total order $\le$ on $\IZ^n$ by
\[
    \beta \le \beta' \quad\text{if and only if}\quad (\la \alpha, \beta \ra, \la \alpha',\beta \ra) \le_{\mathrm{lex}} (\la \alpha, \beta'\ra, \la\alpha',\beta'\ra)
\]
for any $\beta,\beta'\in \IZ^n$, where $\le_{\mathrm{lex}}$ is the lexicographic order on $\bR^2$, that is, $(a,b)\le_{\mathrm{lex}} (a',b')$ if and only if $a< a'$ or $a=a'$ and $b\le b'$. Let $z_i$ be a local equation of $E_i$ on $Y$ near $p$. Define 
$$\fv:\widehat{\cO}^*_{Y,p}\to \bZ^n_{\ge 0},\quad
\fv(s):=\min\{\beta \mid s_\beta \ne 0\},$$ 
where $s=\sum_{\beta} s_\beta z^\beta$. Then $\fv$ is a $\bZ^n$-valuation on $\Ik(Y)=\Ik(X)$ with $c_X(\fv)=x$, and it is easy to see that $v(s) = \la\alpha,\fv(s)\ra$ by the definition of the total order on $\IZ^n$. 

The valuation $\fv$ has one-dimensional leaves by the $\IQ$-linear independence of $\alpha_1,\cdots, \alpha_r$ and $\alpha_{r+1},\cdots, \alpha_{n}$, and is faithful since $\fv(z_i)=e_i$, which is the $i$-th standard basis vector of $\bR^n$.  
Define $\ell:\IR^n\to\bR$, $\beta\mapsto \langle\alpha,\beta\rangle$.
By the Izumi inequality \cite[Theorem 3.1]{Li18}, there exists a constant $C>0$ such that $\ell(\fv(s)) = v(s) \le C \cdot \ord_{\fm_\xi}(s)$ for any $s\in \CO_{X,\xi}$. Hence $\fv$ is a good valuation. 
\end{proof}


\subsection{Filtrations}

In this subsection, we collect some definitions and basic properties of filtrations on a fibration germ.

Let $f:(X, \D)\to Z\ni o$ be a klt fibration germ  and $L$ an $f$-ample $\bQ$-Cartier $\bQ$-divisor on $X$, where $Z=\Spec R_0$. Write $\fm_o\seq R_0$ for the maximal ideal defining the closed point $o\in Z$. Fix $l_0\in\bN$ such that $l_0 L$ is Cartier. Recall that $R=R(X,L) = \oplus_{m\in l_0\IN} R_m$ is the section ring of $L$, where 
\[
    R_m= f_*\CO_X(mL)\cong H^0(X,mL),
\]
which is a finite $R_0$-module.

We provide a short summary of the dictionary between \emph{filtrations} and \emph{norms}, which has been studied in the global setting by \cites{Reb-geodesic,BLXZ23,BJ-global-metric} and in the local setting by \cites{BLQ-convexity,Qi-local}. 

\begin{defi}\label{Definition: Filtrations}\rm
   A (graded) \emph{filtration} $\CF$ on $R=\oplus_m R_m$ is a double-indexed collection of $R_0$-submodules $\{\cF^\lambda R_m \subset R_m\}_{\lambda\in\bR,m\in\bN}$, such that for any $\lambda,\lambda'\in \bR$ and $m,m'\in \bN$,
    \begin{enumerate}
        \item (decreasing) $\cF^\lambda R_m \subset \cF^{\lambda'} R_m $ if $\lambda > \lambda'$,
        
        \item (left continuous) $\cF^\lambda R_m=\cF^{\lambda-\epsilon}R_m$ for any $0<\epsilon \ll 1$, and
        
        \item (multiplicative) $\cF^\lambda R_m \cdot \cF^{\lambda'}R_{m'} \subset \cF^{\lambda+\lambda'}R_{m+m'}$.
    \end{enumerate}
    We set $\fm^\lam R_m = R_m$ for $\lam<0$ by convention, and define $\cF^{>\lam}R_m\coloneqq \cup_{\mu>\lambda}\cF^{\mu} R_m\subset\cF^\lam R_m$. 
    
    A filtration $\CF$ on $R$ is called an \emph{$\fm$-filtration} if 
    \begin{enumerate}
        \item[(4)] (supported over $\fm$) $\CF^\lam R_0\seq R_0$ is an $\fm$-primary ideal for any $\lambda>0$.
    \end{enumerate}
    We will consider the $\fm$-filtration $\CF$ satisfying the following boundedness conditions: 
    \begin{enumerate}
        \item[(5)] $\cF$ is \emph{left linearly bounded} by $\fm$ if there exist $c\in\IR_{>0}$ and $e_-\in\IR$ such that $$\fm^{\lceil\frac{\lam - me_-}{c} \rceil} R_m \seq \CF^\lam R_m, \quad \forall \lam\in\IR, m\in l_0\IN; $$
        \item[(6)] $\cF$ is \emph{right linearly bounded} by $\fm$ if there exist $C\in\IR_{>0}$ and $e_+\in\IR$ such that $$\CF^\lam R_m \seq \fm^{\lceil\frac{\lam - me_+}{C} \rceil} R_m, \quad \forall \lam\in\IR, m\in l_0\IN. $$
    \end{enumerate}

    For any $s\in R_m\setminus\{0\}$, define $\ord_\CF(s) \coloneqq \sup\{\lam\in \IR\mid s\in \CF^\lam R_m\}$, where the supremum is a maximum if $\cF$ is right linearly bounded. Then $\ord_\CF$ is a \emph{semi-valuation}, meaning that for any $s_1,s_2\in R_m$,
    $\ord_\CF(s_1s_2)\ge \ord_\CF(s_1) + \ord_\CF(s_2)$, and 
    $\ord_\CF(s_1+s_2)\ge \min\{\ord_\CF(s_1), \ord_\CF(s_2)\}$. 
    We set $\ord_\CF(0)\coloneqq+\infty$ by convention. 
    
    An $\fm$-filtration $\CF$ is \emph{linearly bounded} if it is both left and right linearly bounded by $\fm$. 
    In other words, $\cF$ is linearly bounded if and only if  
    \begin{eqnarray*}
    c\cdot \ord_\fm(s) +me_- \le \ord_\CF(s) \le C\cdot \ord_\fm(s) +me_+
    \end{eqnarray*}
    for any $m\in l_0 \IN$ and $s\in R_m$. 
\end{defi}

\begin{rmk}\rm 
A linearly bounded $\fm$-filtration $\cF$ on $R$ induces a linearly bounded $\fm$-filtration on the localization of $R_0$. 
In the global case where $Z=\Spec \Ik$, it is linearly bounded in the sense of \cite{BJ20}. 
\end{rmk}


For any $\lam\in\bR$, conditions (3) and (4) imply that $\Gr_\CF^\lam R_m \coloneqq \CF^\lam R_m/\CF^{>\lam} R_m$ is a finite module over $R_0/\fm_o\cong \Ik$, that is, a finite-dimensional $\Ik$-vector space. 
We have the following discrete subsets of $\IR$: 
\begin{eqnarray*}
    \Gamma_m(\CF)=\Gamma(R_m,\CF)\coloneqq \{\lam\in\IR\mid \Gr_\CF^\lam R_m \ne 0\},~ m\in l_0\bN. 
\end{eqnarray*} 
It consists of a sequence of real numbers $\lam^{(m)}_1 < \lam^{(m)}_2 < \cdots$, called the {\it successive minima} of $\CF$ on $R_m$. 
For any $a\in\IR_{>0}, b\in \IR$, we define the $a$-{\it rescaling} and $b$-{\it shift} of $\CF$, respectively, by
    $$(a\CF)^\lam R_m := \CF^{\lam/a}R_m, \,\text{ and } \,
    \CF(b)^\lam R_m := \CF^{\lam-bm} R_m. $$


Let $v\in \Val_{X,o}$. Then for any $s\in R_m$, we may define $v(s)\coloneqq v(f_s)$, where $f_s$ is the image of $s$ under a trivialization of $\cL|_U\xrightarrow{\sim} \cO_U$ for some neighborhood $U\supset c_X(v)$.  
Hence $v$ induces an $\fm$-filtration $\cF_v$ on $R$ by
\[
    \CF_v^\lam R_m \coloneqq \{s\in R_m\mid v(s)\ge \lam\}. 
\]
We remark that if $f$ is birational, then $v$ is also a valuation on $R_0$ since $\Ik(X)=\Ik(Z)$. If, moreover, $L\sim_\bQ E$ for some $\bQ$-divisor $E\ge 0$ with $\supp E\subset \Ex(f)$, then we know that $R_m\subset R_0$. However, in general, the induced filtration $\cF_v$ on $R$ is different from the one given by the restriction of the valuation ideals $\fa_\lambda(v)=\{f\in R_0\mid v(f)\ge \lambda\}$ on $R_0$. 

\subsection{Birational invariants}


\subsubsection{Log canonical threshold}
Given a klt pair $(X, \D)$ and a non-zero ideal $\fa$ on $X$, the {\it log canonical threshold} $\lct(X, \D; \fa)$ of $\fa$ with respect to $(X, \D)$ is defined to be 
\begin{eqnarray*}
\lct(X,\D;\fa) = \max\{t\ge 0\mid (X,\D+\fa^t) \text{ is log canonical}\} = \mathop{\inf}_{v\in\Val_X^*}\frac{A_{X,\D}(v)}{v(\fa)}. 
\end{eqnarray*}
Let $\fa_\bu = \{\fa_m\}_{m\in\IN}$ be a graded sequence of non-zero ideals on $X$. By \cite[Lemma 1.47]{Xu-kstabilitybook}, we have  
$$v(\fa_\bu):=  \mathop{\inf}_{m\in\IN} \frac{v(\fa_m)}{m}= \mathop{\lim}_{m\to \infty} \frac{v(\fa_m)}{m}. $$
The log canonical threshold $\lct(X, \D; \fa_\bu)$ of $\fa_\bu$ with respect to $(X, \D)$ is defined by 
\begin{eqnarray*}
\lct(X,\D;\fa_\bu) := \mathop{\limsup}_{m\to \infty} m\cdot \lct(X,\D;\fa_m) = \mathop{\inf}_{v\in\Val_X^*}\frac{A_{X,\D}(v)}{v(\fa_\bu)}, 
\end{eqnarray*}
where the limsup is indeed a limit by \cite[Lemma 1.49]{Xu-kstabilitybook}, and the last equality follows from \cite[Lemma 1.60]{Xu-kstabilitybook}. 

\subsubsection{Log canonical slope}

In this subsection, we generalize the notion of log canonical slopes introduced in \cite{XZ19} for filtrations on Fano varieties to the log Fano fibration germ setting, which coincides with log canonical thresholds in the local setting. 

\begin{defi}\rm 
\label{Definition: log canonical slope}
Let $(X, \D)\to Z\ni o$ be a log Fano fibration germ with $L=-(K_X+\D)$, and $\CF$ be a filtration on $R$. The {\it base ideal sequence} $I^{(t)}_\bu = \{I_{m,mt}\}_{m\in l_0\IN}$ of $\CF$ is defined by 
\begin{eqnarray*}
\label{Eqnarray. base ideal sequence of filtrations}
I_{m,mt}
\,\,\,=\,\,\,
I_{m,mt}(\CF)
\,\,\,:=\,\,\,{\rm im}
\Big(\CF^{mt}H^0(X,mL)\otimes \CO_X(-mL)\to \CO_X\Big)
\end{eqnarray*}
for any $m\in l_0\IN$ and $t\in\IR$, whose cosupport is contained in $f^{-1}(o)$. The {\it log canonical slope} of $\CF$ is defined by
\begin{eqnarray*}
\mu(\CF)
\,\,\,=\,\,\, 
\mu_{X,\D}(\CF)
\,\,\,:=\,\,\,
\sup\Big\{
t\mid \lct(X,\D;I^{(t)}_\bu)\ge 1
\Big\}. 
\end{eqnarray*}
\end{defi}

\begin{lem}
\label{Lemma. log canonical slope is linear w.r.t rescaling and shifting}
For any $a\in\IR_{>0}, b\in \IR$, we have $\mu(a\CF)=a\mu(\CF)$ and $\mu(\CF(b))=\mu(\CF)+b$. 
\end{lem}

\begin{rmk} \rm
If $X\to Z$ is the identity morphism, then any filtration $\CF$ on $R=R_0[T]$ is the same as an $\fm$-filtration $\fa_\bu$ of $R_0$, that is, $\CF^\lam R_0 = \fa_\lam$ and $I_{m,mt}=\fa_{mt}$. Hence 
\[
    \lct(X,\D; I_\bu^{(t)}) = \lct(X,\D;\fa_{t\bu}) = t^{-1}\lct(X,\D;\fa_\bu).
\]
In particular, $\mu(\CF)=\lct(X,\D;\fa_\bu)$. 
\end{rmk}

When $\dim Z = 0$, the log canonical slope is well understood; see  \cite[Section 3.3]{Xu-kstabilitybook}. 
In the rest of this subsection, we assume that $\dim Z \ge 1$. In this case, $\CF^\lam R_m$ is non-zero for any $\lam\in\IR$. Hence $$f(t)=\lct(X,\D; I^{(t)}_\bu(\CF)) >0$$ for any $t\in \IR$. 
\begin{lem}\cite[Proposition 3.46]{Xu-kstabilitybook}
\label{Lemma. f(t)=lct(I^t)}
The function $f(t)=\lct(X,\D; I^{(t)}_\bu(\CF))$ is continuous and non-increasing on $\IR$ for any filtration $\CF$ on $R$. Moreover, it is strictly decreasing for $t \ge \mu_{+\infty}(\CF)$, where 
$$\mu_{+\infty}(\CF) = \sup\{t\mid f(t) = +\infty \}. $$
\end{lem}

\begin{lem}\label{Lemma: mu<A}
For any $v\in\Val_{X,o}^*$, we have $\mu_{X,\D}(\CF_v) \le A_{X,\D}(v). $
\end{lem}
\begin{proof}
Let $I_\bu^{(t)}$ be the base ideal sequence of $\CF_v$. 
By definition, we have $v(I^{(t)}_\bu) \ge t$. Then for any $t\ge  A_{X,\D}(v)$, we have $\lct(X,\D;I^{(t)}_\bu) \le \frac{A_{X,\D}(v)}{v(I^{(t)}_\bu)} \le 1$. Note that $\mu_{+\infty}(\CF_v)=0$. By Lemma \ref{Lemma. f(t)=lct(I^t)}, $\lct(X,\D;I_\bu^{(t)})$ is strictly decreasing on $[0, \lam_\max)$. We conclude that $\mu(\CF_v)\le A_{X,\D}(v)$. 
\end{proof}

\begin{lem}
\label{Lemma: mu=A}
If there exists a $\IQ$-complement $\Gamma$ of $(X,\D)$, such that $v$ is a log canonical place of $(X,\D+\Gamma)$. Then $\mu_{X,\D}(\CF_v) = A_{X,\D}(v)$. 
\end{lem}

\begin{proof}
Assume that $\Gamma\in \frac{1}{m}|mL|$. Since $v(\Gamma)=A_{X,\D}(v)$, we have $\Gamma\in \frac{1}{m}|\CF_v^{mA_{X,\D}(v)}R_m|$ and
\begin{eqnarray*}
\lct(X,\D;I^{(A_{X,\D}(v))}_\bu) 
\ge \lct(X,\D;\Gamma) 
\ge 1. 
\end{eqnarray*}
Hence $\mu(\CF_v)\ge A_{X,\D}(v)$. We conclude by Lemma \ref{Lemma: mu<A}. 
\end{proof}

\subsection{Test configurations}
Let $X$ be a normal variety and $K=\Ik(X)$ be the field of rational functions on $X$. Let $\IA^1 = \IA^1_a$ be the affine line with the standard $\IG_m$-action. We will establish the test configuration theory of $X$ in parallel with \cite[Section 2]{BHJ17}, though we need more assumptions in the definition of test configurations since $X$ is not proper. 

\begin{defi} \rm 
A {\it test configuration} $(\CX,\eta)$ of $X$ consists of the following data: 
\begin{enumerate}
\item a normal variety $\CX$ with a $\IG_m$-action $\eta: \IG_m\times \CX \to \CX, (t,x)\mapsto \eta_t(x)$; 
\item a flat $\IG_m$-equivariant morphism $\pi: \CX \to \IA^1$, that is, $\pi(\eta_t(x)) = t\cdot \pi(x)$; 
\item an isomorphism $i: \CX_1 \cong X$. 
\end{enumerate}
In particular, the $\IG_m$-action gives us an isomorphism 
\begin{eqnarray*} 
\label{Eqnarray. t.c. CX minus CX_0  cong  X times G_m}
i_\eta: \CX\setminus\CX_0 
&\cong& X\times (\IA^1\setminus\{0\}),\\
x &\mapsto& (i(\eta_{t^{-1}}(x)),t:=\pi(x)), 
\\ \eta_t (i^{-1}(y)) &\mapsfrom& (y,t). 
\end{eqnarray*}

We have the induced isomorphism $i_\eta^*: \Ik(X)(t) = \Ik(X\times \IA^1) \cong \Ik(\CX)$ of the field of rational functions. 
A valuation $w$ on $\Ik(\CX)$ is called {\it $\IG_m$-invariant} if $w(\eta_t^*s) = w(s)$ for any $s\in \Ik(\CX)$ and $t\in \IG_m$. 
For any $\IG_m$-invariant valuation $w$ on $\Ik(\CX)$, denote by $i_{\eta,*}w$ the valuation $(i_{\eta,*}w)(s) = w(i_\eta^*s)$ for any $s\in \Ik(X)(t)$. We denote its restriction on $\Ik(X)$ by $r(i_{\eta,*}w)$. Conversely, for any valuation $v$ on $\Ik(X)$, its {\it Gauss extension} $G(v)$ is the $\IG_m$-invariant valuation on $\Ik(X)(t)$ defined by 
$$G(v)(\sum_k f_k t^k) = \min_k \{v(f_k) + k\}. $$
By \cite[Lemma 4.2]{BHJ17}, the maps
\begin{eqnarray*}  
w \mapsto r(i_{\eta,*}w), \quad
 i_{\eta,*}^{-1} G(v) \mapsfrom v
\end{eqnarray*}
give a one-to-one correspondence between the set of $\IG_m$-invariant valuations on $\Ik(\CX)$ and the set of valuations on $\Ik(X)$.

The central fiber $\CX_0:=\pi^{-1}(0)$ is an effective Cartier divisor on $\CX$ by the flatness of $\pi$. Write $$\CX_0 = \sum b_E E, $$ where $b_E\in \IZ_{\ge 1}$ and the sum is taken over all the irreducible components of $\CX_0$. Then each component $E$ determines a divisorial valuation (or trivial) valuation $r(i_{\eta,*}\ord_E)$ (\cite[Lemma 4.1]{BHJ17}) on $\Ik(X)$. Denote by 
\begin{eqnarray} 
\label{Eqnarray. valuation of t.c.}
v_E = \frac{r(i_{\eta,*}\ord_E)}{b_E} \in \Val_{K}. 
\end{eqnarray}
However, the valuation $v_E$ may not be centered on $X$ since $X$ is not proper. 
\end{defi}



Let $(X,\D)$ be a quasi-projective normal pair and $L$ be a $\IQ$-line bundle on $X$. We define: 

\begin{defi}\rm 
\label{Definition. t.c. of polarized pair}
A {\it test configuration} $(\CX,\D_\CX;\CL,\eta)$ of $(X,\D;L)$ consists of the following data: 
\begin{enumerate}
\item a test configuration $(\CX,\eta)$ of $X$; 
\item a $\IG_m$-linearized $\IQ$-line bundle $\CL$ on $\CX$ such that $i_\eta$ extends to a $\IG_m$-equivariant isomorphism
\begin{eqnarray*} 
\label{Eqnarray. i_eta}
i_\eta: (\CX\setminus\CX_0; \CL|_{\CX\setminus\CX_0}) \cong (X; L)\times (\IA^1\setminus\{0\}); 
\end{eqnarray*}
\item a $\IG_m$-invariant $\IQ$-divisor $\D_\CX = \overline{i_\eta^{-1}(\D\times (\IA^1\setminus\{0\}))}$. 
\end{enumerate}
\end{defi}


Next, we define test configurations for fibration germs. Let $f: (X,\D; L)\to Z\ni o$ be a normal fibration germ with a $\IQ$-line bundle $L$ (hence $Z$ is normal). 

\begin{defi}\rm
A {\it test configuration} $F: (\CX, \D_\CX;\CL,\eta) \to \CZ \supseteq \IA^1$ of $f: (X,\D;L)\to Z\ni o$ consists the following data: 
\begin{enumerate}[{\rm \quad (1)}]
\item a test configuration $(\CX,\D_\CX;\CL,\eta)$ of $(X,\D;L)$; 
\item a test configuration $(\CZ\supseteq \IA^1, \zeta)$ of $Z\ni o$, where $\IA^1$ is the closure of $i_\zeta^{-1}(\{o\}\times (\IA^1\setminus\{0\}))$; 
\item a $\IG_m$-equivariant morphism $F$ compatible with $\CX\to \IA^1$ and $\CZ\to \IA^1$; 
\end{enumerate}
such that {\bf the restricted valuations $r(i_{\eta,*}\ord_E) \in \Val_{\Ik(X)}$ are all centered in $f^{-1}(o)$}, where $E$ are the irreducible components of $\CX_0$. 

If $L$ is (semi)ample over $Z$, then we assume moreover that $\CL$ is (semi)ample over $\CZ$. 

We have the following commutative diagram: 
\begin{equation*}
\xymatrix@R=3ex
{
(\CX, \D_\CX; \CL)|_{\CX\setminus\CX_0}\ar[rr]^{i_\eta}_{\cong}
 \ar[dd]_{F}
&&
(X,\D;L)\times (\IA^1\setminus\{0\}) \ar[dd]^{f\times id}
\\
&&\\
\CZ\setminus \CZ_0 \ar[rr]^{i_\zeta}_{\cong}
&&
Z\times (\IA^1\setminus\{0\}). 
}    
\end{equation*}
\end{defi}

\begin{defi} \rm 
Let $f: (X,\D)\to Z\ni o$ be a log Fano fibration germ, and let $L=-(K_X+\D)$. A {\it test configuration} $F: (\CX, \D_\CX;\CL,\eta) \to \CZ \supseteq \IA^1$ of the log Fano fibration germ is defined by the test configurations of $f: (X,\D;L)\to Z\ni o$. 

If $\CL\sim_{\CZ,\IQ}-(K_\CX+\D_\CX)$, then the test configuration is called {\it $\IQ$-Gorenstein}. It is moreover called {\it weakly special} (resp. {\it special}) if $(\CX,\D_\CX+\CX_0)$ is log canonical (resp. plt). If the test configuration is special, we see that $(\CX_0, \D_{\CX,0})$ is klt by adjunction and $F|_{\CX_0}: (\CX_0, \D_{\CX,0}; \CL_0)\to \CZ_0 \ni o$ is a log Fano fibration germ. The test configuration is called of {\it product type} if there exists a $\IG_m$-equivariant isomorphism $(\CX,\D_\CX;\CL) \cong (X,\D;L)\times \IA^1$ extending $i_\eta$. 
\end{defi}

For any test configuration $(\CX, \D_\CX;\CL,\eta)$ of $f:(X,\D)\to Z\ni o$, we have the following
$\IZ$-filtration $\CF=\CF_{(X,\D_\CX;\CL,\eta)}$ on the relative section ring $R=R(X;L)$, 
\begin{eqnarray} \label{Eqnarray: test configuration to filtration}
\CF^\lam R_m
&:=&\{s\in R_m \mid t^{-\lam} \bar{s} \in \CR_m\}, 
\end{eqnarray}
where $t$ is the parameter on $\IA^1$, and $\bar{s} = i_\eta^* \pr_1^*s$ is the $\IG_m$-extension of $s$ on $X$ to $\CX\setminus\CX_0$, which is a rational relative section of $m\CL$.

\subsubsection{Valuations and filtrations induced by test configurations}
Since $L$ is ample over $Z$ and $Z$ is affine, for sufficiently divisible $m\in\IN$, the line bundle $mL$ is very ample and embeds $X$ into a projective space $\IP^N$. Let $X_c$ be the closure of $X$ in $\IP^N$ and $L_c=\frac{1}{m} \CO_{\IP^N}(1)|_{X_c}$. Let $\nu: \oX\to X_c$ be the normalization and $\oL=\nu^*L_c$. We get a normal projective variety $(\oX,\oL)$ polarized by a semi-ample $\IQ$-line bundle compactifying $(X,L)$. With the same procedure, we get a projective $\IG_m$-equivariant morphism $(\oCX, \oCL) \to \IA^1$ compactifying $(\CX;\CL)\to \IA^1$, where $\oCX$ is normal and $\oCL$ is semi-ample. Note that $\oCX_0\setminus \CX_0$ has no divisorial component in $\oCX$. 

The $\IG_m$-equivariant isomorphism $i_\eta: \CX\setminus\CX_0 \cong X\times (\IA^1\setminus\{0\})$ induces a $\IG_m$-equivariant birational morphism $\oCX\dashrightarrow \oX_{\IA^1}$. Let $\CY$ be the normalization of the graph of this birational map over $\IA^1$, that is, 
\begin{eqnarray}
\label{Eqnarray. common resolution. diagram}
\xymatrix{
  & \CY \ar^-{\chi}[dr]\ar_-{\tau}[dl]&  \\
  \oCX \ar@{-->}^-{i_\eta}[rr] &      &\oX_{\IA^1}. } 
\end{eqnarray}
By abuse of notation, for any irreducible component $E\seq \oCX_0$, we still denote its strict transform to $\CY$ by $E$. 
Consider the following $\IQ$-Cartier $\IQ$-divisor 
$$D= \tau^*\oCL - \chi^* \oL_{\IA^1},$$
which is supported in $\CY_0$. 

\begin{lem}\cite[Lemma 5.17]{BHJ17}
\label{Lemma. BHJ lemma 5.17}
Let $(\CX, \D_\CX;\CL,\eta)\to \CZ\supseteq \IA^1$ be a test configuration of a normal fibration germ $f:(X,\D;L)\to Z\ni o$, where $L$ is $Z$-ample. Then the filtration $\CF=\CF_{(\CX, \D_\CX;\CL,\eta)}$ satisfies: 
\begin{eqnarray*} \label{Eqnarray. filtration of t.c. for polarized fibrations}
\CF^\lam R_m 
= \bigcap_E \{s\in R_m\mid v_E(s) + m\frac{\ord_E(D)}{b_E} \ge \lam \}. 
\end{eqnarray*}
In other words,
\begin{eqnarray*} \label{Eqnarray. filtration of t.c. for polarized fibrations_simple}
\CF=\bigcap_E\CF_{v_E}(\frac{\ord_E(D)}{b_E}), 
\end{eqnarray*} 
where $E$ runs over all the irreducible components of $\CX_0$. 
\end{lem}

\begin{rmk}\rm
This lemma relies on our assumption in the definition for test configurations that $r(i_{\eta,*}\ord_E)$ is centered in $f^{-1}(o)$ for any component $E$ of $\CX_0$. 
\end{rmk}

Let $(\CX, \D_\CX;\CL,\eta)\to \CZ\supseteq \IA^1$ be a test configuration of a log Fano fibration germ $f:(X,\D)\to Z\ni o$. Denote the difference of $\CL$ and $-(K_\CX+\D_\CX)$ by 
$$-\CD= \CL + K_\CX+\D_\CX = \sum_E e_E E, $$
which is a $\IQ$-divisor supported in $\CX_0$. Its closure in $\oCX_0$ is just 
$-\oCD= \oCL + K_\oCX+\D_\oCX. $

\begin{lem}
\label{Lemma. filtration of t.c.}
Let $(\CX, \D_\CX;\CL,\eta)\to \CZ\supseteq \IA^1$ be a test configuration of a log Fano fibration germ $f:(X,\D)\to Z\ni o$. Then the filtration $\CF=\CF_{(\CX, \D_\CX;\CL,\eta)}$ satisfies: 
\begin{eqnarray*} \label{Eqnarray. filtration of t.c. for Fano fibrations}
\CF_{(\CX, \D_\CX;\CL,\eta)} 
=\bigcap_E\CF_{v_E}\Big(\frac{e_E+1}{b_E} - 1 - A_{X,\D}(v_E)\Big), 
\end{eqnarray*}
where $E$ runs over all the irreducible components of $\CX_0$. In particular, if $(\CX, \D_\CX;\CL,\eta)$ is weakly special, then 
\begin{eqnarray*} \label{Eqnarray. filtration of weakly special t.c.}
\CF_{(\CX, \D_\CX;\CL,\eta)} 
=\bigcap_E\CF_{v_E}(-A_{X,\D}(v_E)).  
\end{eqnarray*}
\end{lem}
\begin{proof}
By Lemma \ref{Lemma. BHJ lemma 5.17}, it suffices to show that $\ord_E(D) = e_E+1 - b_E - A_{X,\D}(r(\ord_E))$ for any irreducible component $E$ of $\CX_0$. It follows from 
\begin{eqnarray*}
D &=& \tau^*\oCL - \chi^* \oL_{\IA^1} \\ 
&=&
\tau^*(\oCL + K_\oCX+\D_\oCX) + 
\big(-\tau^*(K_\oCX+\D_\oCX) + \chi^*(K_{X_{\IA^1}}+\D_{\IA^1})\big)\\
&=& \sum_E e_E E + \sum_E (1-b_E-A_{X,\D}(r(\ord_E))) E + F, 
\end{eqnarray*}
where $F$ is a $\IQ$-divisor supported in $\CY_0\setminus \tau_*^{-1}\oCX_0$ and the second equality follows from \cite[Lemma 4.5]{Wang24b} and adjunction. 
\end{proof}

\subsection{Relative cone construction}
\label{ssec: Cone construction}

In this subsection, we collect some properties of the relative cone construction of a fibration, most of which are parallel to those of \cite[Section 3.1]{Kol13} for projective pairs. We also refer the readers to \cite[Section 4.2]{HQZ25} for related results.

Let $f:(X,\Delta)\to Z$ be a fibration, where $Z=\Spec R_0$ is affine.
Let $L$ be an $f$-ample $\bQ$-Cartier Weil divisor on $X$ and 
\[
    R\coloneqq \oplus_{j\in\bN} R_j,~ \text{ where } R_j\coloneqq f_*\cO_X(jL) \cong H^0(X,jL).
\]
Then $C=C(X,L)\coloneqq \Spec R$ is the (orbifold) cone over $Z$. Denote by $R_+ = \oplus_{j>0} R_j$, which is the ideal of the natural section of vertices $\sigma:Z\hookrightarrow C$. Denote by $\Gamma = \sigma(Z) \seq C$. 

Let $\tilde C\coloneqq \Spec_X\oplus_{j\in\bN} \cO_X(jL)$. We have a natural projection $\pi:\tilde C\to X$ with the zero section $\tau: X\to \tC$. Denote by $E=\tau(X)$. Then $\tilde C$ is isomorphic to the blowup of $C$ along $Z$, $\varphi:\tilde C=\Bl_Z C\to C$. The set-theoretical inverse image of the vertices section is $\varphi^{-1}(\sigma(Z)) = E$. If $f$ is of fiber type, we see that $E \seq \tC$ is the exceptional divisor of $\varphi$. However, if $f$ is birational, then the blowup morphism $\varphi: \tC\to C$ is small. In summary, we have the following diagram:
\begin{equation}
\label{Eqnarray. Relative cone diagram}
\begin{tikzcd}
    E\arrow[hook]{r} \arrow{d} &\tilde{C}\arrow{d}{\varphi}\arrow{r}{\pi} &X\arrow{d}{f} \arrow[bend left]{l}{\tau} \\
    F \arrow[hook]{r} &C\arrow{r}{g} & Z\arrow[bend left]{l}{\sigma}. 
\end{tikzcd}
\end{equation}
Denote by $C^* = C\setminus Z \cong \tilde{C}\setminus E$, which is a Seifert $\IG_m$-bundle over $X$ in the sense of \cite{Kol-Seifert-bundle}. 
By \cite[Lemma 2.6]{HQZ25}, the induced map $\pi^*:\Cl(X)\to \Cl(\tC)$ is an isomorphism, and $-E\sim \pi^*L, K_{\tC} +E\sim \pi^*K_X$. For any $\lam \in \IQ_{\ge 0}$, we have 
\begin{eqnarray}
K_{\tC} +\pi^*\D +(1-\lam)E \sim \pi^*(K_X+\D+\lam L). 
\end{eqnarray} 
Denote by $\D_{\tC} \coloneqq \pi^*\D +(1-\lambda)E$ (on $\tilde C$) and 
\begin{eqnarray}
\label{Eqnarray. D_C}
\D_C\coloneqq \varphi_*\pi^*\D =
\left\{ \begin{array}{ll}
        \pi^*\D+(1-\lambda)F, &\text{if $f$ is birational},\\
        \pi^*\D, &\text{if $f$ is of fiber type.}
\end{array} \right. 
\end{eqnarray} 
Then $K_C+\D_{C}
\sim_\IQ \varphi_*(K_{\tC} +\D_{\tC}) 
\sim_\IQ \varphi_*\pi^*(K_X+\D+\lam L)$ is $\IQ$-Cartier if $K_X+\D+\lam L \sim_{\IQ,Z}0$. 

\begin{lem} \label{lem:cone singularity}
\label{Lemma. singularities of cone fibration}
The pair $(C,\D_C)$ is klt (resp. lc, slc) 
if and only if $(X,\Delta)$ is klt (resp. lc, slc) 
and $K_X+\Delta+\lambda L\sim_\bQ 0$ 
 for some 
\begin{enumerate}
\item $0<\lambda\le 1$ (resp. $0\le \lambda\le 1)$ if $f$ is birational;
\item $0<\lambda$ (resp. $0\le \lambda$) if $f$ is of fiber type. 
\end{enumerate}
In both cases, we have $A_{C,\D_C}(E) = \lam$. 
\end{lem}

\begin{proof}
    See \cite[Lemma 4.8 and Lemma 4.9]{HQZ25}.
\end{proof}



In the following assume that $l_0 L \sim_{\IQ} -(K_X+\D)$ for some $l_0 \in \IQ_{>0}$ (assume moreover that $l_0\ge 1$ if $f$ is birational) and let $R=R(L) = \oplus_{j\in \IN } R_j(L)$. Then
\begin{eqnarray}
\label{Eqnarray. D_C l_0}
\D_C\coloneqq \varphi_*\pi^*\D =
\left\{ \begin{array}{ll}
        \pi^*\D+(1-l_0^{-1})F, &\text{if $f$ is birational},\\
        \pi^*\D, &\text{if $f$ is of fiber type.}
\end{array} \right. 
\end{eqnarray} 

Consider the $\IG_m$-action on $C=\Spec R$ given by the $j$-grading of $R=\oplus_{j\in \IN} R_j$; see \cite[Section 2.1.4]{LX18}. Let $\xi$ be the co-weight vector on $C$ reading the $j$-grading, that is, it is determined by the valuation $\wt_\xi(s) = j$ for any $j\in\IN$ and $s\in R_j$. For any $\IG_m$-invariant valuation $v\in \Val_{C,g^{-1}(o)}^{\IG_m}$ with center $C_C(v) \seq g^{-1}(o)$ (where $g:C\to Z$ is the projection, see (\ref{Eqnarray. Relative cone diagram})) and $b>0$, the $b\xi$-twist $v_{b\xi}$ of $v$ is the $\IG_m$-valuation defined by $v_{b\xi}(s) = v(s) + j$ for any $j\in\IN$ and $s\in R_j$. Then the center of $v_{b\xi}$ on $C$ is the closed point $\sigma(o)$ (where $\sigma: Z\to C$ is the section of vertices, see (\ref{Eqnarray. Relative cone diagram})). By the argument of \cite[Lemma 2.6 and 2.7]{LW24}, we have 
\begin{eqnarray}
\label{Eqnarray. A(v_xi)=A(v)+A(xi)}
A_{C,\D_C}(v_{b\xi_0}) = A_{C,\D_C}(v) + bA_{C,\D_C}(\wt_{\xi}). 
\end{eqnarray}
The pull-back of a valuation $v\in \Val_{X,o}^* = \Val_{X,f^{-1}(o)}^*$ via the Seifert $\IG_m$-bundle $\pi: C^* \to X$ is a $\IG_m$-invariant valuation, whose center on $C$ is contained in $g^{-1}(o)$ (still denote the valuation by $v\in \Val_{C,g^{-1}(o)}^{\IG_m, *}$). We see that $C_{C}(v) \nsubseteq F=\sigma(Z)$. Hence 
\begin{eqnarray}
\label{Eqnarray. A=A}
A_{C,\D_C}(v) = A_{C^*,\D_{C}|_{C*}}(v) = A_{X,\D}(v). 
\end{eqnarray}

\begin{rmk}\rm
Note that the $\IG_m$-action is not good in the sense of \cite[Definition 2.12]{LX18} and $\xi$ is not a Reeb vector since the center $\Center_C(\wt_\xi) = F$ is not a closed point on $C$.  
\end{rmk}

\subsubsection{Complements of the relative affine cones}
Using Lemma \ref{lem:cone singularity}, we have a one-to-one correspondence between $\IQ$-complements of $(X,\D)\to Z\ni o$ and $\IG_m$-invariant $\IQ$-complements of the klt singularity $\sigma(o) \in (C,\D_C)$; see \cite[Lemma 4.9]{Wang25} for the case of log Fano varieties. 

Since $L\sim_{Z,\IQ} -l_0^{-1}(K_X+\D)$, we have $R_j=R_j(L) = H^0(X,-jl_0^{-1}(K_X+\D))$. Let $\Gamma$ be a $\IQ$-complement of $(X,\D)\to Z\ni o$. Then $\Gamma = \frac{1}{jl_0^{-1}} \div_X(s)$ for some sufficiently divisible $j$ and some $s\in R_{j}$, and $(X,\D+\Gamma)$ is log canonical. Let $\D^+ = \D+\Gamma$ and
\begin{eqnarray*}
\D_C^+ \coloneqq \varphi_*\pi^*\D^+ =
\left\{ \begin{array}{ll}
        \pi^*\D^++ F, &\text{if $f$ is birational},\\
        \pi^*\D^+, &\text{if $f$ is of fiber type.}
\end{array} \right. 
\end{eqnarray*} 
By Lemma \ref{lem:cone singularity} (for $\lam=0$), we see that $(C,\D_C^+)$ is log canonical. By (\ref{Eqnarray. A(v_xi)=A(v)+A(xi)}) and (\ref{Eqnarray. A=A}), any log canonical place of $(X,\D^+)$ centered in $f^{-1}(o)$ corresponds to a log canonical place of $(C,\D_C^+)$ centered in $\sigma(o)$.
Denote by 
\begin{eqnarray*}
\label{Eqnarray. cone construction. complements Gamma to Gamma_C}
\Gamma_C \coloneqq \D_C^+ - \D_C =
\left\{ \begin{array}{ll}
        \pi^*\Gamma + l_0^{-1} F, &\text{if $f$ is birational},\\
        \pi^*\Gamma, &\text{if $f$ is of fiber type.}
\end{array} \right. 
\end{eqnarray*} 
Hence $\Gamma_C$ is a $\IG_m$-invariant $\IQ$-complement of the klt singularity $\sigma(o)\in (C,\D_C)$. The converse is also true by the same argument as above. 

\begin{prop}
\label{Proposition. cone construction. complement one-to-one correspondence}
The map $\Gamma \mapsto \Gamma_{C}$ between the set of $\IQ$-complements of $(X,\D)\to Z\ni o$ and the set of $\IG_m$-invariant $\IQ$-complements of the klt singularity $\sigma(o) \in (C,\D_C)$ is bijective. Moreover, we have an isomorphism: 
\begin{eqnarray*}
\LC(X,\D+\Gamma) \times \IR_{\ge 0} &\to& \LC(C,\D_C+\Gamma_C)^{\IG_m} \\
(v,b) &\mapsto& v_{b\xi}, 
\end{eqnarray*} 
where $\LC(C,\D_C+\Gamma_C)^{\IG_m}$ is the set of $\IG_m$-invariant log canonical places of $(C,\D_C+\Gamma_C)$. 
\end{prop}

Another application of the relative cone construction is the linear boundedness of $\CF_v$ for valuations $v\in \Val_{X,o}^*$ with finite log discrepancies. 

\begin{prop}
\label{Proposition. F_v is linearly bounded if A(v)<infty}
For any $v\in\Val_{X,o}$, the induced filtration $\cF_v$ is left linearly bounded. Moreover, $\cF_v$ is linearly bounded if $v\in\Val^*_{X,o}$. 
\end{prop}
\begin{proof}
The filtration $\CF_v$ is left linearly bounded by $\fm_o$ since $v(\fm_o)\cdot \ord_o(s) \le v(s)$ for any $s\in R_m$. 
Assume that $A_{X,\D}(v)<+\infty$. 
We need to show that $\CF_v$ is right linearly bounded. We use cone construction. Choose a sufficiently divisible $l_0$ such that $L=-l_0(K_X+\D)$ is very ample over $Z$ and let $(C,\D_C)$ be the relative affine cone as (\ref{Eqnarray. D_C l_0}). Then $C\cong \Spec(R(L))= \Spec(\oplus_{j\in \IN}R_j(L))$ and let $\xi$ be the coweight vector on $C$ reading the $j$-grading. By the Izumi-type inequality \cite{Li18}, there exists $c_1$ depending only on $o\in (C,\D_C)$ such that $w(s) \le c_1A_{C,\D_C}(w) \ord_{C,\sigma(o)}(s)$ for any $w\in\Val_{C,\sigma(o)}^*$, $s\in R_j(L)$ and $j\in\IN$. Let $w=v_\xi\in\Val_{C,\sigma(o)}^*$, then $w(s) = v(s) + j$ and $A_{C,\D_C}(w) = A_{X,\D}(v) + A_{C,\D_C}(\wt_\xi)$. On the other hand, note that $\ord_{C,\sigma(o)}$ is given by the order of the ideal $\fm_o R(L) + R_+(L) \seq R(L)$ and $\ord_{o}$ is given by the order of the ideal $\fm_oR(L) \seq R(L)$, where $R_+(L) = \oplus_{j>0} R_j(L)$. Then for any $s\in R_j(L)$, we have $\ord_{C,\sigma(o)}(s) \le \ord_o(s) + j$. Hence
\begin{eqnarray*}
v(s) &\le& c_1 A_{C,\D_C}(v_\xi)\ord_{C,\sigma(o)}(s) \\
&\le& c_1 A_{C,\D_C}(v_\xi)\ord_o(s) + c_1 A_{C,\D_C}(v_\xi) j. 
\end{eqnarray*} 
We conclude by letting $C=c_1 A_{C,\D_C}(v_\xi)$ and $e_+ = l_0^{-1}c_1 A_{X,\D}(v_\xi)$ in Definition \ref{Definition: Filtrations} (6). 
\end{proof}

\subsubsection{Test configurations of relative affine cones}
\label{Subsubsection. Test configurations of the relative affine cones}

By Lemma \ref{lem:cone singularity}, we also have a one-to-one correspondence between the set of weakly special test configurations of $(X,\D)\to Z\ni o$ and the set of $\IG_m$-equivariant weakly special test configurations of the klt singularity $\sigma(o) \in (C,\D_C,\xi)$. 

Let $(\CX,\D_\CX; l_0^{-1}\CL,\eta) \to \CZ \supseteq \IA^1$ be a weakly special test configuration of $(X,\D)\to Z\ni o$, where $\CL|_{\CX_1} \cong -l_0(K_X+\D) = L$ and $\CL$ is Cartier by enlarging $l_0$. Then $\CZ=\Spec \CR_0$ for some normal $\Ik[t]$-algebra $\CR_0$ and $\CX= \Proj_{\CR_0} \CR(\CL)$, where $\CR(\CL)=\oplus_{j\in\IN} H^0(\CX,j\CL)$. Let $\CC = \Spec \CR$ and let $\D_\CC$ be as in (\ref{Eqnarray. D_C}). Then $(\CC,\D_\CC +\CC_0)$ is log canonical by Lemma \ref{lem:cone singularity} since $(\CX,\D_\CX+\CX_0)$ is log canonical. Hence $(\CC,\D_\CC,\xi,\eta+b\xi)$ is a $\IG_m$-equivariant weakly special test configuration of $(C,\D_C,\xi)$ in the sense of \cite[Lemma 2.21]{LWX15}. 

Conversely, let $(\CC,\D_\CC,\xi,\eta')$ be a $\IG_m$-equivariant weakly special test configuration of the cone $(C,\D_C,\xi)$. Then $\CC = \Spec \CR$ for some $\Ik[t]$-algebra $\CR$ admitting a $\IG_m$-action compatible with the standard $\IG_m$-action on $\Ik[t]$. Hence $\CR$ admits a weight decomposition $\CR=\oplus_{j\in\IN} \CR_j$. Let $\CZ =\Spec \CR_0$ and $\CX = \Proj_{\CR_0} \CR$. Then there is a natural embedding $\sigma:\CZ \to \CC$, and $(\CC\setminus \sigma(\CZ), \D_\CC|_{\CC\setminus \sigma(\CZ)}) \to (\CX,\D_\CX)$ is a Seifert $\IG_m$-bundle in the sense of \cite{Kol-Seifert-bundle}. Let $\CL = \CO_{\CX/\CZ}(1)$. Then $K_\CX+\D_\CX+l_0^{-1} \CL \sim_\IQ 0$ for some $l_0\in \IQ_{>0}$ by (\ref{Eqnarray. D_C}) since $K_\CC+\D_\CC$ is $\IQ$-Cartier. Hence $(\CX,\D_\CX,\CL;\eta) \to \CZ \supseteq \IA^1$ is a weakly special test configuration of $(X,\D)\to Z\ni o$, where $\eta'=\eta+b\xi$ for some $b>0$.

\subsection{Torus actions 
and \texorpdfstring{$\xi$}{}-twists} 
\label{Subsection. torus action and xi-twist}

Let $f: (X,\D)\to Z\ni o$ be a log Fano fibration germ. Assume that $(X,\D)$ admits a {\it good} $\IT=\IG_m^r$-action, that is, the closure of any $\IT$-orbit intersects $f^{-1}(o)$. Then it induces a good $\IT_0=\IG_m^{r_0}$-action on $(o\in Z)$ such that $f$ is equivariant with respect to the torus actions. In this case, $-l_0(K_X+\D)$ admits a canonical $\IT$-linearization, and hence the anti-canonical ring $R$ admits a {\it canonical weight decomposition }
\begin{eqnarray*}
R_0 = \bigoplus_{\alpha \in M_0} R_{0,\alpha}, \quad 
R_m = \bigoplus_{\alpha \in M} R_{m,\alpha}, 
\end{eqnarray*}
where $M_0=M(\IT_0)\cong \IZ^{r_0} \seq M=M(\IT)\cong \IZ^r$ are the weight lattices of the $\IT_0$-action and $\IT$-action, respectively. Since the $\IT_0$-action on $(o\in Z)$ is good, we have $R_{0,0}=\Ik$. Denote by 
\begin{eqnarray*}
\Gamma_0 = \{ \alpha \in M_0 \mid R_{0,\alpha} \ne 0 \}, \quad
\Gamma_m = \{ \alpha \in M \mid R_{m,\alpha} \ne 0 \}.
\end{eqnarray*}
Then $\Gamma_0\seq M_0$ is a semigroup. Since $R_m$ is an $R_0$-module, the subset $\Gamma_m\seq M$ admits a $\Gamma_0$-action. Hence
\begin{eqnarray}
\label{Eqnarray. Gamma_m+Gamma_0 = Gamma_m}
\Gamma_m+\Gamma_0 = \Gamma_m. 
\end{eqnarray}
Let $\sigma_0 = \Conv(\Gamma_0) \seq M_{0,\IR}$ be the cone spanned by $\Gamma_0$, and
\begin{eqnarray*}
\sigma &=& \Conv\{(\alpha,m)\in M\times \IN \mid \alpha \in \Gamma_m\} \seq M_\IR \times \IR, \\
\BP &=& \BP^{-(K_X+\D)}_\IT \,\,\,=\,\,\, \sigma \cap (M_\IR \times \{1\}). 
\end{eqnarray*}
The convex polyhedron $\BP\seq M_\IR$ is called the {\it moment polyhedron} of the $\IT$-action. By (\ref{Eqnarray. Gamma_m+Gamma_0 = Gamma_m}), we have 
\begin{eqnarray}
\label{Eqnarray. BP+sigma_0=BP}
\BP+ \sigma_0 = \BP.  
\end{eqnarray}

Let $N_0=M_0^\vee, N=M^\vee$ be the co-weight lattices of the $\IT_0$-action and $\IT$-action respectively, and let $\la-,-\ra : M\times N \to \IZ$ be the natural pairing. Any $\xi \in N$ determines a one-parameter group: $\IG_m \to \Aut(X,\D), t\mapsto \xi_t$, whose action on $R_m$ is given by 
\begin{eqnarray} 
\label{Eqnarray. 1-PS action} 
(\xi_t^*s)(x) = s(\xi_{t^{-1}}(x))
\end{eqnarray}
for any $s\in R_m$ and $x\in X$, see \cite[(2.21)]{Xu-kstabilitybook}. It satisfies
\begin{eqnarray}
\label{Eqnarray. weight-coweight pairing}
\xi_t^*s = t^{\la\alpha,\xi\ra} s
\end{eqnarray}
for any $s\in R_{m,\alpha}$. 
We have the following short exact sequence: 
\begin{eqnarray}
\label{Eqnarray. M_0^perp to N to N_0}
0\to M_0^\perp \to N \to N_0 \to 0, 
\end{eqnarray}
where $M_0^\perp = \{ \xi \in N \mid \la\alpha, \xi \ra=0, \,\, \forall \alpha \in M_0 \}$. 
Denote by 
\begin{eqnarray*}
\sigma_0^\vee &=& \{ \xi \in N_{0,\IR} \mid \la \alpha, \xi \ra \ge 0, \,\, \forall \alpha \in \sigma_0 \}, \\
\sigma^\vee &=& \{ (\xi, \lam) \in N_\IR \times \IR \mid \la \alpha, \xi \ra + m \lam \ge 0, \,\, \forall( \alpha, m) \in \sigma \}. 
\end{eqnarray*}
Then the {\it Reeb cones} are defined by 
\begin{eqnarray*}
\reeb(Z,\IT_0) &:=& \{ \xi \in N_{0,\IR} \mid \la \alpha, \xi \ra > 0, \,\, \forall \alpha \in \Gamma_0 \} = \interior(\sigma_0^\vee), \\
\reeb(C,\IT) &:=& \{ (\xi, \lam) \in N_\IR \times \IR \mid \la \alpha, \xi \ra + m \lam > 0, \,\, \forall m\in l_0\IN, \alpha \in \Gamma_m \}, \\
\reeb = \reeb(X,\IT) &:=& \pr_1(\reeb(C)) = \{\xi\in N_\IR \mid \la \BP,\xi\ra \,\text{ has a lower bound} \}.
\end{eqnarray*}
We have the following short exact sequence of semigroups: 
\begin{eqnarray}
\label{Eqnarray. M_0,R^perp to t to t_0}
0\to M_{0,\IR}^\perp \to \reeb(X,\IT) \to \reeb(Z,\IT_0) \to 0. 
\end{eqnarray}
Let $\xi \in \reeb(X,\IT)$. We say that $f:(X,\D,\xi) \to Z\ni o$ is a {\it polarized log Fano fibration germ} admitting a good $\IT$-action.

\begin{defi}\rm
\label{Definition. twist of filtrations}
Let $\xi\in \reeb$ and $\CF$ be a $\IT$-invariant filtration on $R$. The $\xi$-{\it twist} of $\CF$ is defined by 
\begin{eqnarray*}  
\CF_\xi^\lam R_{m} = \oplus_{\alpha\in M} (\CF_\xi^{\lam}R_m)_\alpha, \quad 
(\CF_\xi^{\lam}R_m)_\alpha := \CF^{\lam-\la\alpha,\xi\ra}R_{m,\alpha}. 
\end{eqnarray*}
One may verify that $\CF_\xi$ is still a filtration on $R$ (Definition \ref{Definition: Filtrations}). 
\end{defi}

Since $X$ admits a $\IT$-action, there exists a $\IT$-equivariant dominant birational map $\pi: X \dashrightarrow W$, where $W$ is the Chow quotient of $X$ and $\IT$ acts on $W$ trivially. Then the  function field $\Ik(X)$ of $X$ is the fractional field of $\Ik(W)[M] = \oplus_{\alpha \in M} \Ik(W) \cdot 1^\alpha$. 

\begin{defi}\rm
\label{Definition. twist of valuations}
For any valuation $\mu$ on $\Ik(W)$ and $\xi \in N_\IR$, we define the $\IT$-invariant valuation $v_{\mu, \xi}$ on $X$ by 
$$v_{\mu, \xi}(f)
= \min_\alpha\{\mu(f_\alpha)+\la \alpha, \xi \ra\},$$
for any $f=\sum_\alpha f_\alpha \cdot 1^\alpha \in K(W)[M]$. 
For any $v = v_{\mu, \xi_0}\in \Val^\IT_{\Ik(X)}$ and $\xi \in N_\IR$, we define the $\xi$-twist of $v$ by $v_\xi\coloneqq v_{\mu,\xi_0+\xi}$. One can check that the definition is not dependent on the choice of the birational map $X\dashrightarrow W$. If $\mu$ is the trivial valuation on $\Ik(W)$, then we denote $\wt_\xi := v_{\mu, \xi}$. 
\end{defi}

If $v\in \Val_{X,o}^{\IT,*}$ and $\xi\in \reeb$, then $v_\xi \in \Val_{X,o}^{\IT,*}$. We set 
$$\theta_\xi(v) = A_{X,\D}(v_\xi) - A_{X,\D}(v). $$
For any $\xi \in \reeb\cap N$, it determines a product test configuration $(\CX_\xi, \D_{\CX_\xi}; \CL_\xi, \eta=\eta_\xi = (\xi,1))$: 
\begin{eqnarray*} 
(\CX_\xi, \D_{\CX_\xi}; \CL_\xi)=(X, \D,-(K_X+\D))\times \IA^1
\end{eqnarray*}
where the coweight vector $\eta = (\xi,1)\in N\times \IZ \cong N(\IT\times \IG_m)$ is determined by the isomorphism
\begin{eqnarray*} 
i_\eta: \CX_\xi \setminus \CX_{\xi,0} &\cong& X\times (\IA^1\setminus\{0\}), \\
(x,t) &\mapsto& (\xi_{t^{-1}}(x), t), 
\end{eqnarray*}
which follows the notation in Definition \ref{Definition. t.c. of polarized pair}; see also \cite[Example 3.5]{Fuj17_hyperplane}. 
As in (\ref{Eqnarray. common resolution. diagram}), compactifying and taking the graph, we get the commutative diagram: 
\begin{eqnarray*}
\label{Eqnarray. product t.c. common resolution. diagram}
\xymatrix{
  & \CY \ar^-{\chi}[dr]\ar_-{\tau}[dl]&  \\
  \oCX_\xi = \oX_{\IA^1} \ar@{-->}^-{i_\eta}[rr] &      &\oX_{\IA^1}. }
\end{eqnarray*}
By \cite[Proposition 3.3]{Li19} or \cite[Lemma 6.21]{Xu-kstabilitybook}, for any $v\in \Val_{X,o}^*$ 
, we have  
\begin{eqnarray}
\label{Eqnarray. i_eta,* G(v)  =  G(v_xi)}
i_{\eta,*}G(v) &=& G(v_\xi), \\
\label{Eqnarray. theta_xi(v) = G(v)( (-K2) - (-K1) )}
\theta_{\xi}(v) &=& G(v) \Big(-\chi^*(K_{X_{\IA^1}} + \D_{\IA^1}) + \tau^*(K_{X_{\IA^1}} + \D_{\IA^1}) \Big). 
\end{eqnarray}
Using (\ref{Eqnarray. theta_xi(v) = G(v)( (-K2) - (-K1) )}) and \cite[Proposition 3.8]{Li19} or \cite[Lemma 6.22]{Xu-kstabilitybook}, we have the following description of the twist of valuations. 
\begin{prop}
For any $v\in \Val_{X,o}^{\IT,*}$, $\xi\in \reeb$ and $s\in R_{m,\alpha}$, we have 
\begin{eqnarray}
\label{Eqnarray. v_xi = v + <> + theta_xi}
v_\xi(s) = v(s) + \la\alpha, \xi\ra + \theta_\xi(v). 
\end{eqnarray}
\end{prop}
In other words, the twist of valuation $v$ and the twist of filtration $\CF_v$ are related by 
\begin{eqnarray}
\label{Eqnarray. v_xi and F_v,xi}
\CF_{v_\xi} = \CF_{v,\xi}(\theta_\xi(v)). 
\end{eqnarray}

\begin{prop}\cite[Lemma A.6]{XZ19}
\label{Lemma. log canonical slope is invariant under xi-twist} 
For any $\IT$-invariant filtration $\CF$ on $R$, let $\mu=\mu(\CF)$ and $v$ be a valuation computing the log canonical threshold of $I^{(\mu)}_\bu(\CF)$. Then for any $\xi\in \reeb$, we have $\mu(\CF_\xi) = \mu$ and $v_\xi$ computes the log canonical threshold of $I^{(\mu)}_\bu(\CF_\xi)$. 
\end{prop}

If we choose $v=v_\triv$, then by (\ref{Eqnarray. i_eta,* G(v)  =  G(v_xi)}), we have
(using the notation as in (\ref{Eqnarray. valuation of t.c.})): 
\begin{eqnarray*} 
\label{Eqnarray: toric divisor}
\wt_\xi = v_{\CX_{\xi,0}} = r(i_{\eta,*}\ord_{\CX_{\xi,0}}). 
\end{eqnarray*}
By (\ref{Eqnarray. v_xi and F_v,xi}) and (\ref{Eqnarray. v_xi = v + <> + theta_xi}), we have 
\begin{eqnarray} 
\label{Eqnarray: valuation of product TC} 
\CF_{\wt_\xi}
&=& 
\CF_{\triv,\xi}(A_{X,\D}(\wt_\xi)),  \\
\wt_\xi(s)
&=& \la\alpha,\xi\ra + mA_{X,\D}(\wt_\xi), \quad
s\in R_{m,\alpha}. 
\end{eqnarray}
Hence, we have the following characterization of the canonical moment polyhedral $\BP$. Recall that a {\it facet} $F$ of the polyhedral $\BP$ is a codimension one face; the {\it facet normal} of $F$ is the (unique) primitive vector $\xi \in N$ satisfying $\la\alpha-\alpha', \xi\ra \ge 0$ for any $\alpha\in \BP,\alpha'\in F$ and the equality holds for any $\alpha\in F$.  

\begin{cor}
\label{Corollary. 0 in int(P)}
For any facet $F$ of $\BP$, its facet normal $\xi\in N_\IR$ satisfies 
$$\la\alpha,\xi\ra + A_{X,\D}(\wt_\xi) \ge 0 $$
for any $\alpha\in \BP$. The equality holds if and only if $\alpha \in F$. 
\end{cor}

\section{Special divisors and valuations}
\label{Section. Special divisors and valuations}

In this section, we generalize the theory of (weakly) special valuations developed in \cites{LXZ22,XZ-sdsing,Che25} to the setting of log Fano fibration germs, and study the corresponding degenerations. Throughout this section, we work with a log Fano fibration germ $f:(X,\D=\sum_i a_i \D_i)\to Z\ni o$, where $\D_i$ are the irreducible components of $\D$.

\subsection{Special degenerations induced by valuations}

Let $v\in \Val_{X,o}^*$ be a quasi-monomial valuation. Throughout this subsection, we assume that the associated graded ring 
\begin{eqnarray}
\label{Eqnarray. Gr_v R}  
\Gr_v R = \bigoplus_{m\in l_0\IN} \bigoplus_{\lam \in \IR_{\ge 0}} \Gr_v^\lam R_m, \quad
\Gr_v^\lam R_m = \CF_v^{\lam} R_m /\CF_v^{>\lam} R_m 
\end{eqnarray}
is a {\it finitely generated} $\Ik$-algebra. Note that $\Gr_v R$ is always an integral domain. 
Denote by 
\begin{eqnarray}  
\label{Eqnarray. Xv Zv Deltav}
X_v = \Proj_{\Gr_v R_0} \Gr_v R, \quad Z_v = \Spec \Gr_v R_0, \quad \D_v = \sum_i a_i \D_{i,v}, 
\end{eqnarray}
where $\D_{i,v} \seq X_v$ is the corresponding degeneration divisor of $\D_i$ on $X_v$. 
More precisely, let 
$$I_{\D_i} =  \bigoplus_{m\in l_0\IN} I_{\D_i,m} \seq \bigoplus_{m\in l_0\IN} R_{m} = R$$ 
be the graded ideal of $\D_i$. Recall that 
$$\Gr_v R_m = \bigoplus_{\lam\in \Gamma_m(\CF_v) }\Gr_v^\lam R_m. $$ 
We define the {\it initial term degeneration} $\Bin(I_{\D_i}) \seq \Gr_v R$ of $I_{\D_i} \seq R$ by 
\begin{eqnarray*}  
\Bin(I_{\D_i}) &=& \bigoplus_{m\in l_0\IN} \bigoplus_{\lam\in \Gamma_m(\CF_v) } I_{\D_i,m,\lam}, \\
I_{\D_i,m,\lam} &=& \{ \bar{s} \in \Gr_v R_m \mid s\in I_{\D_i,m}, v(s)\ge \lam \}.
\end{eqnarray*}
Then $\D_{i,v}$ is the divisorial part of the sub-scheme $V(\Bin(I_{\D_i})) \seq X_v$ defined by $\Bin(I_{\D_i})$. In other words, $\D_{i,v}$ and $V(\Bin(I_{\D_i}))$ coincide away from a codimension $2$ subset of $X_v$. 

More generally, for any filtration $\CF$ on $R$, we define its {\it initial term degeneration} $\CF'$ on $\Gr_v R$ by
\begin{eqnarray}
\label{Eqnarray. Initial term degeneration of filtration}
\CF'^{\lam} \Gr_v R_m
:= \{ \bar{s} \in \Gr_v R_m \mid s \in \CF^\lam R_m \}. 
\end{eqnarray}
By the lower semi-continuity of log canonical thresholds, we have
$\mu_{X,\D}(\CF) \ge \mu_{X_v,\D_v}(\CF'). $

\begin{defi}\rm
\label{Definition. special valuations}
A quasi-monomial valuation $v\in\Val_{X,o}^*$ is called a {\it special valuation} over $(X,\D)$ if $\Gr_v R$ is finitely generated and the induced degeneration $(X_v,\D_v)$ is klt. 
\end{defi}

Assume that $v$ is of rational rank $r$. As in the proof of Lemma \ref{Lemma. good valuation associated to a quasi-monomial valuation}, there exists a log smooth model $(Y,E=E_1+\cdots+E_n)\to X$ with a closed point $p\in E_1\cap\cdots\cap E_n$ such that $v \in \QM_p(Y,E) \cong \IR^{n}_{\ge 0}$ is the quasi-monomial valuation determined by some $\eta = (\eta_1,\cdots,\eta_r,0,\cdots, 0)$, where $\eta_1,\cdots,\eta_r \in \IR_{>0}$ are $\IQ$-linearly independent. Then the sum $\oplus_\lam$ in (\ref{Eqnarray. Gr_v R}) is taken over the semigroup: 
\begin{eqnarray*}  
\Gamma = \eta_1\IZ_{\ge 0} + \cdots + \eta_r\IZ_{\ge 0} \seq \IR_{\ge 0}. 
\end{eqnarray*}
Denote by 
\begin{eqnarray}  
\label{Eqnarray. R_v}
R_v = \bigoplus_{m\in l_0\IN} \bigoplus_{\alpha\in\IZ_{\ge 0}^r}  R_{v,m,\alpha}, \quad \text{where}\quad
R_{v,m,\alpha} \coloneqq \Gr_v^{\alpha_1\eta_1+\cdots+\alpha_r\eta_r} R_m.  
\end{eqnarray}
Note that this weight decomposition is not the canonical one, Corollary \ref{Corollary. 0 in int(P)}.
Then the initial term degeneration of $\CF_v$ on $R_v$ is the $\IG_m^r$-invariant filtration $\CF_{\wt_{\eta_v}}$: 
\begin{eqnarray*}  
\label{Eqnarray. F_eta_v}
\CF_{\wt_{\eta_v}}^\lam R_{v,m} = \bigoplus_{\alpha\in\IZ_{\ge 0}^r, \alpha_1\eta_1+\cdots+\alpha_r\eta_r \ge \lam}  R_{v,m,\alpha}, 
\end{eqnarray*}
which is just the filtration of the co-weight valuation $\wt_{\eta_v}$ determined by 
\begin{eqnarray*}  
\label{Eqnarray. eta_v}
\eta_v = (\eta_1,\cdots,\eta_r) \in N_\IR(\IG_m^r) \cong \IR^r. 
\end{eqnarray*}
It's clear that $\eta_v \in \reeb(\IG_m^r)$. Hence from a special valuation $v$ of $(X,\D)\to Z\ni o$, we get a polarized log Fano fibration germ $(X_v,\D_v,\eta_v) \to Z_v\ni o$. 
\begin{defi}\rm
\label{Definition. special degeneration induced by v}
The polarized log Fano fibration germ $(X_v,\D_v,\eta_v) \to Z_v\ni o$ is called the {\it special degeneration of $(X,\D)\to Z\ni o$ induced by} the special valuation $v$. 
\end{defi}

Next, we study the behavior of the valuation $v$ on the relative affine cone $(C = C(X,L),\D_C)$ (where $L = -l_0(K_X+\D)$ is Cartier) using notation as in Section \ref{ssec: Cone construction}. Denote by 
\begin{eqnarray*}  
R_j(L)
&=& H^0(X,jL) = H^0(X,-jl_0(K_X+\D)) = R_{jl_0},\\
R(L)
&=& \oplus_{j\in \IN} R_j(L) = \oplus_{m\in l_0\IN} R_m = R. 
\end{eqnarray*}
Let $\xi$ be the co-weight vector of the $\IG_m$-action on $C$ satisfying $\wt_\xi(s)=j$ for any $s\in R_j(L) = R_{jl_0}$ and $w=v_{b\xi} \in \Val_{C,\sigma(o)}^*$. Then $w(s) = v(s) + j$ for any $s\in R_j(L) = R_{jl_0}$. We have 
\begin{eqnarray}  
\label{Eqnarray. Gr_w cong Gr_v}
\Gr_w R 
&=& \bigoplus_{m\in l_0\IN} \bigoplus_{\lam' \in \IR_{\ge 0}} \Gr_w^{\lam'} R_m \\ 
\nonumber
&=& \bigoplus_{m\in l_0\IN} \bigoplus_{\lam' \in \IR_{\ge 0}} \Gr_v^{\lam' - ml_0^{-1}} R_m
\cong \Gr_v R, 
\end{eqnarray}
which is a finitely generated $\Ik$-algebra. 
Moreover, let $(C_w = \Spec \Gr_w R, \D_{C,w})$ be the special degeneration of $(C,\D_C)$ induced by $w$. Then it is the relative affine cone over $(X_v,\D_v)\to Z_v\ni o$, hence klt by Lemma \ref{lem:cone singularity}. In other words, $w$ is a special valuation over $\sigma(o) \in (C,\D_C)$. 

\begin{prop}
\label{Proposition. cone constr. special valuation correspondence}
Let $w\in \Val_{C,\sigma(o)}^{\IG_m,*}$. Then there exist $v\in \Val_{X,o}^*$ and $b>0$ such that $w=v_{b\xi}$. Moreover, $v$ is a special valuation over $(X,\D)\to Z\ni o$ if and only if $w$ is a special valuation over $\sigma(o) \in (C,\D_C)$. 
\end{prop}

Let $R_w = \Gr_w R \cong R_v$ as in (\ref{Eqnarray. Gr_w cong Gr_v}). Using the weight decomposition (\ref{Eqnarray. R_v}), 
we see that the co-weight vector $\eta_w$ on $C$ induced by $w$ is determined by 
\begin{eqnarray*}  
\label{Eqnarray. wt_eta_w}
\wt_{\eta_w}(\bar{s}) = \alpha_1\eta_1+\cdots+\alpha_r\eta_r + ml_0^{-1}, 
\end{eqnarray*}
for any $\bar{s}\in R_{v,m,\alpha}$. Hence 
\begin{eqnarray*}  
\label{Eqnarray. F_eta_w}
\CF_{\wt_{\eta_w}}^{\lam'} R_{w} \cong \bigoplus_{m\in l_0 \IN, \alpha\in\IZ_{\ge 0}^r, \alpha_1\eta_1+\cdots+\alpha_r\eta_r + ml_0^{-1} \ge \lam'}  R_{v,m,\alpha}. 
\end{eqnarray*}
By \cite[Lemma 2.58]{LX18}, we have
\begin{eqnarray}  
\label{Eqnarray. A_C(w) = A_Cw(wt_eta_w)}
A_{C_w,\D_{C,w}}(\wt_{\eta_w}) = A_{C,\D_C}(w).  
\end{eqnarray}
Since $w=v_{b\xi}$, we have $\eta_w = \eta_v + b\xi$. By (\ref{Eqnarray. A(v_xi)=A(v)+A(xi)}) and (\ref{Eqnarray. A_C(w) = A_Cw(wt_eta_w)}), we see that 
\begin{eqnarray}  
\label{Eqnarray. A_X(v) = A_Xv(wt_eta_v)}
A_{X_v,\D_{v}}(\wt_{\eta_v}) = A_{X,\D}(v).  
\end{eqnarray}

\begin{defi}[Rees construction] \rm
If $v$ is of rational rank $r=1$, that is $v = c\cdot\ord_{E}$ for some prime divisor $E$ over $X$ and $c>0$, then we define the {\it extended Rees algebra} by 
\begin{eqnarray*}
\CR_v = \Rees_{v} R 
:= \bigoplus_{m\in l_0\IN, \lam\in\IZ}  t^{-\lam}\CF_v^{c\lam} R_m, \quad 
\CR_{v,0} = \Rees_{v} R_0 
= \bigoplus_{\lam\in\IZ} t^{-\lam}\CF_v^{c\lam} R_0. 
\end{eqnarray*}
\end{defi}
By assumption, this is a finitely generated $\Ik[t]$-algebra since 
$$\CR_v/ t\CR_v \cong \Gr_v R$$ 
is finitely generated. Let $\CX= \Proj_{\CR_{v,0}} \CR_v, \CZ=\Spec\CR_{v,0}$ and $\D_\CX$ be the closure of $\D\times (\IA^1\setminus \{0\}) \seq \CX\setminus \CX_0$ in $\CX$. The $m$-grading of $\Rees_v R$ gives us a $\IQ$-divisor $\CL= \CO_{\CX/\CZ}(1)$, then $\CL|_{\CX_1} \cong -l_0(K_X+\D) = L$. Let $\eta_v$ be the co-weight vector reading the $c\lam$-grading ($\wt_{\eta_v}(f\cdot t) = v(f) - 1$). 
We get a test configuration $(\CX,\D_\CX;\CL,\eta_v) \to \CZ \supseteq \IA^1$. If we assume that $\CX_0$ is integral, by (\ref{Eqnarray. valuation of t.c.}), the induced divisorial valuation $r(i_{\eta_v,*}\ord_{\CX_0}) \in \Val_{X,o}^*$ is just $v$.  One may note that 
\begin{eqnarray}
\label{Eqnarray. CX_0 = X_v}
(\CX_0,\D_{\CX,0}) \cong (X_v, \D_v), 
\end{eqnarray}
where $\D_{\CX,0} = \Diff_{\CX_0}(\D_\CX)$. By inversion of adjunction, we see that $(\CX, \D_\CX+\CX_0)$ is log canonical (resp. plt) if and only if $(X_v,\D_v)$ is semi-log canonical (resp. klt).

On the other hand, let $\CC= \Spec \CR_v$ and $\D_\CC$ be the closure of $\D_C\times (\IA^1\setminus \{0\}) \seq \CC\setminus \CC_0$ in $\CC$. Choose $\eta_w = \eta_v +b\xi$ for some $b\in \IQ_{>0}$. Then $(\CC,\D_\CC, \xi; \eta_w)$ is the test configuration of the klt singularity $(C,\D_C,\xi)$ induced by the divisorial valuation $w = v_{b\xi}\in \Val_{C,\sigma(o)}^*$ in the sense of \cite[Lemma 2.21]{LWX15}. Moreover, $(\CC,\D_\CC, \xi)$ is the relative affine cone over $(\CX,\D_\CX,\CL) \to \CZ$ by Section \ref{Subsubsection. Test configurations of the relative affine cones}.

\begin{defi}\rm
A divisorial valuation $v\in\Val_{X,\D}^*$ is called {\it weakly special} (resp. {\it special}) over $(X,\D)\to Z\ni o$ if the test configuration $(\CX,\D_\CX)$ induced by $v$ is weakly special (resp. special) with irreducible central fiber, that is, $\CX_0$ is integral and $(\CX,\D_\CX+\CX_0)$ is log canonical (resp. plt). 
\end{defi}

By (\ref{Eqnarray. CX_0 = X_v}), a divisorial valuation $v\in\Val_{X,\D}^*$ is weakly special (resp. special) if and only if $(X_v,\D_v)$ is log canonical (resp. klt). Hence, the definition of special divisorial valuations coincides with the definition of special valuations (Definition \ref{Definition. special valuations}).

Let $v\in \Val_{X,o}^*$ be a special valuation over $(X,\D)\to Z\ni o$. We will use the following theorem about two-step degenerations, which
follows from \cite[Section 6]{LX20} or \cite[Theorem 2.64]{LX18} using the relative cone construction Proposition \ref{Proposition. cone constr. special valuation correspondence}. 

\begin{thm}
\label{Theorem. Replacing two-step dege. to one-step dege.}
Let $w\in \Val_{X_v,o}^*$ be a $\IT=\IG_m^r$-invariant special divisorial valuation and let $(X_{v,w},\D_{v,w},\xi_w)$ be the induced special degeneration of $(X_v,\D_v)$. Then there exists a family of special valuations $\{v_\vep\}_{\vep \ge 0} \seq \Val_{X,o}^*$ such that  
\begin{eqnarray*} 
(X_{v_\vep}, \D_{v_\vep}, \xi_{v_\vep}) = 
(X_{v,w},\D_{v,w}, \xi_{v} + \vep \xi_w). 
\end{eqnarray*}
Moreover, if $v$ is divisorial, then $v_\vep$ is divisorial for any $\vep\in \IQ_{>0}$. 
\end{thm}

\subsection{Complements and (weakly) special valuations}

We will deduce the finite generation property of $\Gr_v R$ assumed in the previous subsection from the theory of complements. 
Recall that an effective $\IQ$-divisor $\Gamma\sim_{\IQ,Z} -(K_X+\D)$ is called a {\it $\IQ$-complement} of $(X,\D)\to Z\ni o$ if $(X,\D+\Gamma)$ is log canonical and there exists a log canonical center of it contained in $f^{-1}(o)$. 

\subsubsection{(Weakly) special divisorial valuations}

\begin{prop}
\label{Proposition. weakly special divisor 2}
A divisorial valuation $v=c\cdot \ord_E \in \Val_{X,o}^*$ is weakly special over $(X,\D)\to Z\ni o$ if and only if there exists a $\IQ$-complement $\Gamma$ of $(X,\D)\to Z\ni o$ such that $v\in\LC(X,\D+\Gamma)$. 
\end{prop}

\begin{proof}
This follows from \cite[Lemma 2.21]{LWX15} and the relative cone construction. 
We use the notation in Proposition \ref{Proposition. cone construction. complement one-to-one correspondence}. If $v$ is a log canonical place of some $\IQ$-complement $\Gamma$, then $w= v_{b\xi} = c'\cdot \ord_F$ is a log canonical place of $(C,\D_C+\Gamma_C)$ for any $b\in\IQ_{>0}$, where $F$ is a prime divisor over $\sigma(o)\in (C,\D_C)$. By \cite[Lemma 1.68]{Xu-kstabilitybook}, there exists a projective birational morphism $\pi: Y\to C$ extracting precisely $F$ and $-F$ is ample over $C$. Hence $(Y, \pi^{-1}_*(\D_C+\Gamma_C) +F)$ is log canonical, and $\CR_w\cong \CR_v$ is a finitely generated $\Ik[t]$-algebra by \cite[Lemma 2.14]{XZ-sdsing}.
Moreover, $(\CC=\Spec \CR_w,\D_\CC,\xi, \eta_w)$ is a weakly special test configuration by \cite[Lemma 2.21]{LWX15}. By the correspondence in Section \ref{Subsubsection. Test configurations of the relative affine cones}, we see that the test configuration $(\CX= \Proj_{\CR_{v,0}}\CR_v, \D_\CX, \CL; \eta_v)$ is weakly special. 

Conversely, let $(\CX,\D_\CX;\CL;\eta) \to \CZ \supseteq \IA^1$ be a weakly special test configuration of $(X,\D)\to Z\ni o$ with integral central fiber and $v=r(i_{\eta,*}\ord_{\CX_0}) \in \Val_{X,o}^*$. Following the same notation as above, let $(\CC,\D_\CC,\xi;\eta)$ be the affine cone over $(\CX,\D_\CX;\CL;\eta)$. Then $(\CC,\D_\CC,\xi;\eta+b\xi)$ is a weakly special test configuration of $(C,\D_C,\xi)$ with integral central fiber for any $b\in \IQ_{>0}$, since $w= r(\ord_{\CC_0}) \in \Val_{C,\sigma(o)}^*$. Checking on $s\in R_j$ we see that $w = v_{b\xi}$. Assume that $w=c'\cdot\ord_F$.  By \cite[Lemma 2.21]{LWX15}, we see that there exists a $\IG_m$-equivariant projective birational morphism $\pi: Y\to C$ extracting $F$ precisely such that $(Y,\pi_*^{-1}\D_C+F)$ is log canonical and $-F$ is ample over $C$. By Bertini's theorem, there exists a $\IG_m$-invariant effective $\IQ$-divisor $G\sim_\IQ -F$ such that $(Y,\pi_*^{-1}\D_C+F+G)$ is still log canonical. Then $\Gamma = \pi_*G$ is a $\IG_m$-invariant $\IQ$-complement such that $w$ is a log canonical place of $(C,\D_C+\Gamma)$. Hence $v$ is a log canonical place of $\IQ$-complement of $(X,\D)\to Z\ni o$ by Proposition \ref{Proposition. cone construction. complement one-to-one correspondence}. 
\end{proof}

By \cite[Lemma 4.26]{Xu-kstabilitybook}, we have the following characterization of special divisorial valuations. 

\begin{prop}
\label{Proposition: special divisor definition 1}
Let $v= \ord_E$ be a weakly special divisorial valuation over $(X,\D)\to Z\ni o$. Then $v$ is special if and only if for any effective $\IQ$-Cartier $\IQ$-divisor $D$, there exist $0<\vep\le 1$ and an effective $\IQ$-divisor $D'\sim_{\IQ} -(K_X+\D+\vep D)$ such that $(X,\D+\vep D+D')$ is log canonical and $v\in\LC(X,\D+\vep D+D')$. 
\end{prop}

\begin{proof}
Let $(\CX,\D_\CX)$ be a weakly special test configuration of $(X,\D)\to Z \ni o$ induced by $v$. If $(\CX,\D_\CX)$ is special, then for any $\IQ$-Cartier $\IQ$-divisor $D\ge 0$, we can find sufficiently small $\vep>0$ such that $-(K_X+\D+\vep D)$ is ample over $Z$ and $(\CX,\D_\CX+\vep\CD_\CX+\CX_0)$ is plt, where $\CD_\CX$ is the closure of $D\times (\IA^1\setminus\{0\})$ in $\CX$. Hence $E$ is a log canonical place of some $\IQ$-complement of $(X,\D+\vep D)$ by Proposition \ref{Proposition. weakly special divisor 2}. 

Conversely, if $(\CX,\D_\CX+\CX_0)$ is not plt. Then there exists a log canonical center $W$ strictly contained in $\CX_0$. Choose $m$ sufficiently divisible such that 
\begin{eqnarray*}
0
&\ne& H^0(\CX_0,\CO(-m(K_\CX+\D_\CX)|_{\CX_0})\otimes I_W) \\
&\seq& H^0(X_0,\CO(-m(K_\CX+\D_\CX)|_{\CX_0})) \cong \Gr_{E} R_m. 
\end{eqnarray*}
Since $I_W$ is $\IG_m$-equivariant, there exists $\lam \in \IZ$ such that 
$$H^0(\CX_0,\CO(-m(K_\CX+\D_\CX)|_{\CX_0})\otimes I_W) \cap \Gr^\lam_{E} R_m \ne 0. $$
Let $\bar{s}\ne 0$ be an element of the above space and $s\in R_m$ be its lifting. Denote by $D=\frac{1}{m}\{s=0\}$ and by $\CD_\CX$ the closure of $D\times (\IA^1\setminus\{0\})$ in $\CX$. Then $\CD_{\CX,0} = \frac{1}{m}\{\bar{s}=0\}$, which implies that $W\seq \CD_{\CX,0}$. Hence for any $\vep>0$, the pair $(\CX,\D_\CX+\vep \CD_\CX+\CX_0)$ is not log canonical. By Proposition \ref{Proposition. weakly special divisor 2}, the divisor $E$ is not a log canonical place of any $\IQ$-complement of $(X,\D+\vep D)$. 
\end{proof}

\subsubsection{Special valuations}

Next, we recall the higher rank finite generation theory developed in \cites{LXZ22,XZ-sdsing} and recently generalized in\cite{Che25}. Roughly speaking, a quasi-monomial valuation $v$ over a klt singularity is special if and only if it is a log canonical place of some {\it special} $\IQ$-complement. 
Throughout this subsection, we follow the notation of \cite[Section 3]{XZ-sdsing}. 

\begin{defi}\rm
\label{Definition. special complements}
Let $x\in(X,\D)$ be a klt singularity. 
A $\IQ$-complement $\Gamma$ of $x\in(X,\D)$ is called {\it special} if there exists a toroidal model $\pi: (Y,E) \to (X,\D)$ and an effective $\pi$-ample $\IQ$-divisor $G$ on $Y$ whose support {\bf does not contain any stratum of} $(Y,E)$ such that $\pi_*^{-1}\Gamma \ge G$. 
\end{defi}
By \cite[Theorem 4.1]{XZ-sdsing}, we have:  

\begin{thm}
\label{Theorem. XZ special valuations}
A quasi-monomial valuation $w\in \Val_{X,x}^*$ is special if and only if there exists a special $\IQ$-complement $\Gamma$ of $x \in (X,\D)$ such that $w\in \LC(X,\D+\Gamma)\cap \QM(Y,E)$. 
\end{thm}

The theorem is proved by constructing a {\it qdlt Fano type model} $(Y^s, E^s)\to (X,\D)$ from the special $\IQ$-complement $\Gamma$ (\cite[Theorem 4.1 (2) $\Rightarrow$ (3)]{XZ-sdsing}), and then deducing the finite generation property of $w$ from the qdlt Fano type model (\cite[Theorem 4.1 (3) $\Rightarrow$ (1)]{XZ-sdsing}). 

Let's recall the construction of the qdlt Fano type model $(Y^s, E^s)$. We first choose a minimal rational simplicial cone $\sigma = \span \{F_1,\cdots,F_r\} \seq \LC(X,\D+\Gamma)\cap \QM(Y,E)$ containing $w$. Then we do some toroidal blowup $(Y',E') \to (Y,E)$ and construct a special $\IQ$-complement $\Gamma'$ with respect to $(Y', E')$ such that $\LC(X,\D_X+\Gamma')=\QM(Y',E') =\sigma$ (tie-breaking step, \cite[Lemma 3.17]{XZ-sdsing}). Finally, we run an MMP on $Y'$ (or, on some dlt modification $Y''\to Y'$) over $X$, extracting the exceptional divisors (over $X$) except the components of $E'$ and get the qdlt Fano type model $(Y^s,E^s)$ (MMP step, \cite[Lemma 3.15]{XZ-sdsing}). 

In the above proof, the assumption that $G$ (on $(Y,E)$) does not contain any strata of $E$ is only used in the proof of \cite[Lemma 3.17]{XZ-sdsing} to show that $F_i$ are log canonical places of $(Y,\D_Y+\pi^*\Gamma - G)$. If we choose toroidal divisors $F_i \in \LC(X,\D+\Gamma)\cap \QM(Y,E)$ satisfying $C_{Y}(F_i) = C_Y(w)$ in \cite[Lemma 3.17]{XZ-sdsing}, then the assumption on $G$ can be weakened to $C_{Y}(w)\nsubseteq G$ and the proof still holds. 

Hence, the above theorem can be slightly generalized by weakening the assumption on the special $\IQ$-complement $\Gamma$.
\begin{defi}\rm
\label{Definition. special complements with respect to w}
Let $w\in \Val_{X,x}^*$ be a quasi-monomial valuation. 
A $\IQ$-complement $\Gamma$ of $o \in (X,\D)$ is called {\it special with respect to $w$} if there exists a toroidal model $\pi: (Y,E) \to (X,\D)$ and an effective $\pi$-ample $\IQ$-divisor $G$ on $Y$ whose support {\bf does not contain} $C_Y(w)$ such that $\pi_*^{-1}\Gamma \ge G$ and $w\in \LC(X,\D+\Gamma)\cap \QM(Y,E)$. 
\end{defi}

\begin{cor}
\label{Corollary. XZ special valuations_slightly generalization}
A quasi-monomial valuation $w\in \Val_{X,x}^*$ is special if and only if there exists a $\IQ$-complement $\Gamma$ of $x \in (X,\D)$ which is special with respect to $w$. 
\end{cor}

This corollary is useful in the $\IT$-equivariant case (e.g., the relative cone of a log Fano fibration germ), in which case Bertini's theorem fails and we could only find $\IT$-invariant $G$ satisfying $w(G)=0$; see Theorem \ref{Theorem. H-minimizer has f.g. graded ring}. This corollary was implicitly used in \cite[Proof of Theorem 5.4]{BLXZ23}.

\subsubsection{Weakly special valuations}

In view of Proposition \ref{Proposition. weakly special divisor 2} and Theorem \ref{Theorem. XZ special valuations}, it is natural to make the following definition. 

\begin{defi}\rm
\label{Definition. weakly special valuations}
A valuation $v\in\Val_{X,o}^*$ is called a {\it weakly special valuation} over $(X,\D)\to Z\ni o$ if there exists a $\IQ$-complement $\Gamma$ of $(X,\D)$ such that $v\in\LC(X,\D+\Gamma)$. 
\end{defi}

We have the following characterization of weakly special valuations generalizing \cite[Lemma 3.3]{XZ-sdsing}; see also \cite[Theorem 2.13]{Wang24b} and \cite[Theorem 3.11]{Wang25}. 

\begin{thm}
\label{Theorem: weakly special valuations}
A quasi-monomial valuation $v\in \Val_{X,o}^*$ is weakly special over $(X,\D)\to Z\ni o$ if and only if $\mu_{X,\D}(\CF_v) = A_{X,\D}(v)$. 
\end{thm}

\begin{proof}
We have the ``only if'' part by Lemma \ref{Lemma: mu=A}. Now we prove the ``if'' part. 
By the ACC of lct \cite{HMX14}, there exists $\vep>0$ depending only on $\dim X$ and the coefficients of $\D$ such that, for any birational morphism $\pi:Y\dashrightarrow X$ and any reduced divisor $E$ on $Y$, the pair $(Y,\pi_*^{-1}\D+(1-\vep)E)$ is log canonical if and only if $(Y,\pi_*^{-1}\D+E)$ is log canonical.

Let $\mu=\mu(\CF_v)$ and $I_\bu^{(\mu)} = I_\bu^{(\mu)}(\CF_v)$. Then
$$v(I_\bu^{(\mu)})\ge\mu, \quad 1=\lct(X,\D; I^{(\mu)}_\bu) \le \frac{A_{X,\D}(v)}{v(I_\bu^{(\mu)})} \le \frac{A_{X,\D}(v)}{\mu}. $$
By the assumption $\mu=A_{X,\D}(v)$, we see that all the equalities hold.  
Since $v\in \Val_{X,o}^*$ is a quasi-monomial valuation, there exists a quasi-monomial simplicial cone $\sigma\seq \Val_{X,o}^*$ containing $v$. The functions 
$$w\mapsto A_{X,\D}(w),\quad w\mapsto w(I^{(\mu)}_\bu)$$ 
are linear and concave on $\sigma$ respectively. Hence the function
\begin{eqnarray*}
\label{Eqnarray: function A(-) on sigma}
A_{X,\D+I^{(\mu)}_\bu}(w) = A_{X,\D}(w) - w(I^{(\mu)}_\bu)
\end{eqnarray*}
is convex for $w\in \sigma$. In particular, it is locally Lipschitz on $\sigma$, with respect to some Euclidean norm $|\cdot|$. Hence there exists $C,K>0$ such that 
\begin{eqnarray*}
|A_{X,\D+I^{(\mu)}_\bu}(w) -
   A_{X,\D+I^{(\mu)}_\bu}(v) |
\le C|w-v|. 
\end{eqnarray*}
On the other hand, $A_{X,\D+I^{(\mu)}_\bu}(w)\ge 0$ for any $w\in\sigma$ since $v$ minimizes $\lct(X,\D; I^{(\mu)}_\bu)=1$. Hence
\begin{eqnarray*}
0 \le 
A_{X,\D+I^{(\mu)}_\bu}(w) = 
|A_{X,\D+I^{(\mu)}_\bu}(w) -
   A_{X,\D+I^{(\mu)}_\bu}(v) |
\le C|w-v|. 
\end{eqnarray*}

By the Diophantine approximation \cite[Lemma 2.7]{LX18}, there exist divisorial valuations $v_1,\cdots, v_r$ and positive integers $q_1,\cdots,q_r,c_1,\cdots,c_r$ such that 
\begin{itemize}
\item $\{v_1,\cdots,v_r\}$ spans a quasi-monomial simplicial cone in $\Val_X$ containing $v$; 
\item for any $1\le i\le r$, there exists a prime divisor $E_i$ over $X$ such that $q_iv_i=c_i\ord_{E_i}$; 
\item $|v_i-v|< \frac{\vep}{2Cq_i}$ for any $1\le i\le r$. 
\end{itemize}
In particular, 
\begin{eqnarray*}
A_{X,\D+I^{(\mu)}_\bu}(E_i) 
= \frac{q_i}{c_i}\cdot A_{X,\D+I^{(\mu)}_\bu}(v_i) 
\le \frac{q_i}{c_i}\cdot C|v_i-v|
< \frac{q_i}{c_i}\cdot C \cdot \frac{\vep}{2 C q_i} 
\le \frac{\vep}{2}. 
\end{eqnarray*}

Choose $0< \vep' < \vep/2\ord_{E_i}(I_\bu^{(\mu)})$ for any $1\le i\le r$. By \cite[Lemma 1.41]{Xu-kstabilitybook}, for $m\gg0$ and general $D_m \in \frac{1}{m}|\CF^{m\mu}R_m|$, we have 
$$\lct(X,\D;(1-\vep')D_m) = \lct(X,\D;I_{m,m\mu}^{(1-\vep')/m}) >1, $$
and $\ord_{E_i}(D_m) = \frac{1}{m}\ord_{E_i}(I_{m,m\mu})$ for any $1\le i\le r$. Hence
\begin{eqnarray*}
a_i &=& A_{X,\D+(1-\vep')D_m}(E_i) \\
&=& A_{X,\D+I_\bu^{(\mu)}}(E_i) 
   + \vep'\cdot\ord_{E_i}(I_\bu^{(\mu)})  \\
&& + (1-\vep') \Big(\ord_{E_i}(I_\bu^{(\mu)})-\frac{1}{m}\ord_{E_i}(I_{m,m\mu})\Big)
   \,\,\,\le\,\,\, \vep, 
\end{eqnarray*}
since $\ord_{E_i}(\fa_\bu)\le \frac{1}{m}\ord_{E_i}(\fa_m)$ for any graded ideal sequence $\fa_\bu$. 

By \cite[Corollary 1.4.3]{BCHM10}, there exists a $\IQ$-factorial birational model $\pi:Y\to X$ extracting precisely $E_1,\cdots, E_r$. Then 
\begin{eqnarray} 
\label{Eqnarray: crepant pullback 1}
&& \pi^*(K_X+\D+(1-\vep')D_m)  \\  
\nonumber
&\sim_{\IQ}&  K_Y+\pi_*^{-1}(\D+(1-\vep')D_m)+\sum_{i=1}^r (1-a_i)E_i. 
\end{eqnarray}
In particular, 
\begin{eqnarray*}
&&\pi^*(K_X+\D+(1-\vep')D_m) -(K_Y+\pi_*^{-1}\D+(1-\vep)E)  \\
&=& (1-\vep')\pi_*^{-1}D_m + \sum_i (\vep - a_i) E_i 
\,\,\,\ge\,\,\, 0. 
\end{eqnarray*}
Since $\lct(X,\D;(1-\vep')D_m) >1$, the pair $(Y,\pi_*^{-1}\D+(1-\vep)E)$ is log canonical. Hence $(Y,\pi_*^{-1}\D+E)$ is also log canonical by our choice of $\vep$. 

Note that $Y$ is of Fano type over $Z$ by (\ref{Eqnarray: crepant pullback 1}). We may run a $-(K_Y+\pi_*^{-1}\D+E)$-MMP over $Z$ and get a good minimal model $\phi: Y\dashrightarrow \oY$. Let $\opi: \oY\dashrightarrow X$ be the induced birational map and $\oE = \phi_*E$. 
Since $(X,\D+(1-\vep')D_m)$ admits a $\IQ$-complement, we see that $(Y,\pi_*^{-1}\D+(1-\vep)E)$ admits a $\IQ$-complement $\Phi$ by (\ref{Eqnarray: crepant pullback 1}). Let $\oPhi=\phi_*\Phi$. Since 
\begin{eqnarray*}
K_Y+\pi_*^{-1}\D+(1-\vep)E +\Phi
\sim_{\IQ} 0, 
\end{eqnarray*}
the MMP $\phi$ is crepant for the log canonical pair $(Y,\pi_*^{-1}\D+(1-\vep)E+\Phi)$, that is, 
\begin{eqnarray*}
K_Y+\pi_*^{-1}\D+(1-\vep)E +\Phi
= \phi^*(K_\oY+\opi_*^{-1}\D+(1-\vep)\oE+\oPhi). 
\end{eqnarray*}
Hence $(\oY,\opi_*^{-1}\D+(1-\vep)\oE+\oPhi)$ is also log canonical. In particular, $(\oY,\opi_*^{-1}\D+(1-\vep)\oE)$ is log canonical. By our choice of $\vep$, we see that $(\oY,\opi_*^{-1}\D+\oE)$ is also log canonical. 

Recall that $\phi: Y\dashrightarrow \oY$ is a $-(K_Y+\pi_*^{-1}\D+E)$-MMP over $Z$. We have
\begin{eqnarray}
\label{Eqnarray. (-K)-MMP A_X(v)-decending}
A_{Y,\pi_*^{-1}\D+E}(F)
\ge
A_{\oY,\opi_*^{-1}\D+\oE}(F) \ge 0
\end{eqnarray}
for any prime divisor $F$ over $Y$. 
Since $(\oY,\opi_*^{-1}\D+\oE)$ is log canonical, we see that $\phi$ is an isomorphism at the generic point of each log canonical center $Z$ of $(Y,\pi_*^{-1}\D+E)$ (otherwise, there exists a prime divisor $F$ over $Y$ with $C_Y(F)=Z$ such that 
\begin{eqnarray*}
0=A_{Y,\pi_*^{-1}\D+E}(F)
>
A_{\oY,\opi_*^{-1}\D+\oE}(F) \ge 0, 
\end{eqnarray*}
which is a contradiction). In particular, $\phi$ is an isomorphism at any stratum of $E$. Hence
\begin{eqnarray}
\label{Eqnarray. crepant pullback of (-K)-MMP}
&&\phi^*(K_\oY+\opi_*^{-1}\D+\oE) - (K_Y+\pi_*^{-1}\D+E)  \\
\nonumber
&=&  \sum_i (1-A_{\oY,\opi_*^{-1}\D+\oE}(F_i)) F_i 
\,\,\,\ge\,\,\, 0
\end{eqnarray}
by (\ref{Eqnarray. (-K)-MMP A_X(v)-decending}), where $F_i$ are prime divisors extracted by $\phi$. 

Since $Y$ is of Fano type over $Z$ and $\phi: Y\dashrightarrow \oY$ is a birational contraction, we have that $\oY$ is also of Fano type over $Z$ by \cite[Lemma 2.12]{Bir19}. On the other hand, $-(K_{\oY}+\opi_*^{-1}\D+\oE)$ is nef over $Z$. Hence $-(K_{\oY}+\opi_*^{-1}\D+\oE)$ is semiample over $Z$ and $(\oY,\opi_*^{-1}\D+\oE)$ admits a $\IQ$-complement by Bertini theorem. By (\ref{Eqnarray. crepant pullback of (-K)-MMP}), $(Y,\pi_*^{-1}\D+E)$ also admits a $\IQ$-complement $\Theta$. Then $\Gamma=\pi_*\Theta$ is a $\IQ$-complement of $(X,\D)\to Z\ni o$ such that $E_i\in \LC(X,\D+\Gamma)$ for any $1\le i\le r$. In particular, $v\in \LC(X,\D+\Gamma)$. 
\end{proof}

\section{Asymptotic invariants of Fano fibration germs}
\label{Section: Asymptotic invariants on Fano fibration germs}

In this section, we develop the theory of Okounkov bodies in the setting of fibration germs and study the associated \emph{Duistermaat-Heckman (DH)} measures. Moreover, we introduce the $\BH$-invariants and delta-invariants of log Fano fibration germs. 





\textbf{Set-up}. Let $f: X \to Z\ni o$ be a fibration germ with $\dim X=n$. Let $L$ be an $f$-ample $\IQ$-Cartier $\IQ$-divisor on $X$. Fix $l_0\in\bZ_{>0}$ such that $l_0L$ is Cartier and $f$-very ample, and let $R=\oplus_{m\in l_0\IN}R_m$, where $R_m=H^0(X,mL)$. 
We always denote by $\fv: \Ik(X)^* \to \IZ^n$ a good valuation centered in $f^{-1}(o)\seq X$.
A \emph{filtration} in this section is always assumed to be left linearly bounded. 

A major motivation is to study the following volume function inspired by \cite{XZ20}.

\begin{defi}\rm \label{Definition. volume functions}
For any filtration $\CF$ on $R$ and $t\in \IR$, we define: 
\begin{eqnarray} 
\label{Eqnarray. Def of volume function}
\vol(\CF;t)= \vol(R /\CF^{(t)}R) := \mathop{\limsup}_{m\to \infty} \frac{\ell(R_m/ \CF^{mt} R_m)}{m^n/n!}. 
\end{eqnarray}
It is easy to see that $\vol(\CF;t)$ is non-decreasing since $\CF$ is non-increasing. 
Since $R$ is left linearly bounded, there exists $e_-(\CF)\in \IR$ such that $\CF^{mt} R_m= R_m$ for any $m\in l_0\IN$ and $t\le e_-(\CF)$. Hence $\vol(\CF;t) = 0$ for any $t< e_-(\CF)$. We define 
\begin{eqnarray*} 
\label{Eqnarray. lam_min}
\lam_\min(\CF) &:=& \inf\{t\in \IR\mid \vol(\CF;t)>0\},~ \text{and} \\
\label{Eqnarray. lam_max}
\lam_\max(\CF) &:=& \sup\{t\in \IR\mid \vol(\CF;t) < \mathop{\sup}_{u\in\IR}\vol(\CF;u)\}. 
\end{eqnarray*}
\end{defi}

For any $a>0$ and $b\in \IR$, we have $\vol(a\CF;t) = \vol(\CF;a^{-1}t)$ and $\vol(\CF(b);t) = \vol(\CF; t-b)$. If $\CF=\CF_v$ for some valuation $v \in \Val_{X,o}^*$, we write $\vol(v;t) \coloneqq \vol(\CF_v; t)$. 

If $Z$ is a closed point, then $\vol(\CF;t) = \vol(L) - \vol(\CF^{(t)}R)$ is the co-volume of a graded linear series on the projective variety $X$, which is well understood; see \cites{BHJ17,BJ20}. 
If $X\cong Z$ and $\CF=\CF_v$ for some $v\in\Val_{X,o}^*$, then it is easy to see that $\vol(v;t) = t^n\cdot \vol(v)$. 

In the following, we focus on the case where $\dim Z \ge 1$. In this case, $\lam_\max(\CF)=+\infty$. 
We will show that the limsup in (\ref{Eqnarray. Def of volume function}) is actually a limit using Okounkov bodies 
in the next subsection.

\subsection{Okounkov bodies}
\label{Section. Okounkov bodies and DH-measures}

A \emph{graded linear series} $V_\bu = \{V_m\}_{m\in l_0\IN}$ of $L$ is a sequence of $R_0$-submodules $\{V_m\seq R_m\}_{m\in l_0\bN}$ satisfying $V_m \cdot V_{m'} \seq V_{m+m'}$. Denote by $V_m^* = V_m\setminus \{0\}$. \footnote{We do not assume that $V_m$ is a finite-dimensional $\Ik$-vector space here, which differs from the usual definition of a linear series.}


Let $V_\bu$ be a graded linear series. For any $m\in l_0\IN$, The subset $\fv(V_m^*) \seq \IN^n$ gives a closed convex subset 
\begin{eqnarray*}
\BO(V_m) \coloneqq \Conv\Big( \frac{\fv(V_m^*)}{m} \Big) \seq \IR_{\ge0}^n, 
\end{eqnarray*}
where $\Conv(A)$ denotes the closure of the convex hull of the subset $A\seq \IR^n$. When $\dim Z \ge 1$, $\BO(V_m)$ is in general unbounded. 
Since $V_\bu$ is multiplicative, we have $\fv(V_m^*) + \fv(V_{m'}^*) \seq \fv(V_{m+m'}^*) $. Hence $\BO_m\seq \BO_{pm}$ for any $p\in \IZ_{>0}$. 
We say that the (unbounded) closed convex subset
\begin{eqnarray}
\label{Eqnarray. Okounkov body}
\BO(V_\bu) \coloneqq \overline{\bigcup_{m\in l_0\IN} \BO_m}\subset \bR_{\ge 0}^n
\end{eqnarray}
is the {\it Okounkov body} of the graded linear series $V_\bu$, although the usual definition of a convex \emph{body} requires compactness.

If $\By \in \interior(\BO(V_\bu))_\IQ \coloneqq \interior(\BO(V_\bu))\cap \IQ^n$, then for $m$ sufficiently divisible, $\By$ is a rational convex combination of points in $m^{-1}\fv(V_m^*)$: 
\begin{eqnarray*}
\By = \frac{a_1}{k} \frac{\fv(s_1)}{m} +\cdots+ \frac{a_r}{k} \frac{\fv(s_r)}{m} 
= \frac{\fv(s_1^{a_1}) + \cdots + \fv(s_r^{a_r})}{mk}
=\frac{\fv(s_1^{a_1}\cdots s_r^{a_r})}{mk}, 
\end{eqnarray*}
where $a_i, k\in\IZ_{>0}$ satisfy $a_1+\cdots+a_r = k$ and $s_i\in V_m^*$. So $\By \in (mk)^{-1}\fv(V_{mk}^*)$ and
\begin{eqnarray}
\label{Eqnarray. Okounkov body, by discrete set}
\interior(\BO(V_\bu))_\IQ \seq \bigcup_{m\in l_0\IN} \frac{\fv(V_m^*)}{m} \seq \BO(V_\bu), \quad \text{and} \quad
 \overline{\bigcup_{m\in l_0\IN} \frac{\fv(V_m^*)}{m}} = \BO(V_\bu).  
\end{eqnarray}


Let $\BO:=\BO(R)$ be the Okounkov body of the trivial linear series $R=R_\bu$. For any filtration $\CF$ on $R$ and any $t\in\IR$, we have a graded linear series $ \CF^{(t)} R \coloneqq\{\CF^{mt}R_m\}_{m\in l_0\IN}$ and the associated Okounkov body  
\begin{eqnarray*}
\BO^{(t)}(\CF) = \BO^{(\ge t)}(\CF) := \BO(\CF^{(t)}R) \seq \BO(R) = \BO. 
\end{eqnarray*}
Since $\CF$ is non-increasing and left linearly bounded, we have 
\begin{eqnarray*}
t\le t' &\Rightarrow&
\BO^{(t)}(\CF) \supseteq \BO^{(t')}(\CF),~ \text{and} \\
t\le e_-(\CF) &\Rightarrow& 
\BO^{(t)}(\CF) = \BO(R). 
\end{eqnarray*}
The {\it concave transform} $G_\cF:\bR^n\to\bR\cup\{-\infty\}$ of $\CF$ is defined by 
\[
    G_\CF(\By)\coloneqq\sup\{t\in \IR\mid \By\in\BO^{(t)}(\CF)\}.
\]
By definition, if $\By\in\IR^n\setminus \BO$, then $G_\CF(\By)=-\infty$; if $t\in\IR$, then 
\[
\{\By\in \IR^n\mid G_\CF(\By) \ge t\} = \BO^{(t)}(\CF), 
\]
which is a closed convex subset of $\IR^n$. Hence $G_\CF:\IR^n \to\IR \cup \{-\infty\}$ is concave and upper semi-continuous. In particular, $G_\CF$ is locally Lipschitz continuous on $\interior(\BO)$.

\begin{lem}
For any $a\in\IR_{>0}, b\in \IR$, we have $G_{a\CF}=aG_\CF$ and $G_{\CF(b)}=G_\CF+b$. \qed
\end{lem}

Note that the Okounkov bodies defined above are unbounded in general. 
To overcome this difficulty, we follow the strategy of \cite{KK12} and construct $\BO$ using the good $\IZ^n$-valuation $\fv$ compatible with a quasi-monomial valuation $v$ in Lemma \ref{Lemma. good valuation associated to a quasi-monomial valuation}. In this case, we get a family of {\it bounded} convex bodies 
\begin{eqnarray}
\label{Eqnarray. BO_<=t(v)}
\BO_{(\le t)}(v) = \BO(R/\CF_v^{(t)} R) := \overline{\BO\setminus \BO^{(t)}(v)}, 
\end{eqnarray}
since the concave transform $G_v$ is linear on $\interior(\BO)$ by the definition of $\fv$; see Lemma \ref{Lemma. G_v is linear}. 
It is not clear whether the convex bodies $\BO_{(\le t)}(v)$ are bounded for arbitrary $v$. But we will see in the following Theorem \ref{Theorem. Convergence of vol and DH} that they are bounded for quasi-monomial valuations $v$ of rational rank $n$, hence bounded for all $v\in\Val_{X,o}^*$ by the linear boundedness of $\CF_v$, Proposition \ref{Proposition. F_v is linearly bounded if A(v)<infty}. 

We denote by 
\begin{eqnarray}
\label{Eqnarray. partial BO_<=t(v)}
\partial \BO_{(\le t)}(v) = \partial \BO^{(t)}(v) := \overline{\BO\cap \{G_v = t\}}
\end{eqnarray}
the linear slicing of $\BO(v)$, which is an $(n-1)$-dimensional closed convex body. 
We denote the ($n$-dimensional) Euclidean volume of $\BO_{(\le t)}(v)$ and the ($(n-1)$-dimensional) Euclidean volume of $\partial \BO_{(\le t)}(v)$ by 
$$\vol(\BO_{(\le t)}(v))\quad \text{and} \quad \vol_{n-1}(\partial \BO_{(\le t)}(v)),$$
respectively. Note that $\vol_{n-1}(\partial \BO_{(\le t)}(v)) > 0$ for $t>0$ since in this case $\interior(\BO)\cap \{G_v = t\} \ne \varnothing$.

\begin{lem}
\label{Lemma. G_v is linear}
Let $v \in \Val_{X,o}^*$ be a quasi-monomial valuation, and let $\fv$ be the good valuation associated to $v$. Let $\BO\seq \IR^n$ be the Okounkov body of $R$ constructed using $\fv$. Then the concave transform $G_v:\IR^n\to \IR$ is linear on $\interior(\BO)$. 
\end{lem}

\begin{proof}
We follow the setup of Lemma \ref{Lemma. good valuation associated to a quasi-monomial valuation}. Since $v$ is quasi-monomial, there exists a log smooth model $\mu:(Y,E=E_1+\cdots+E_n)\to X$ and a closed point $p\in \cap E_i$ such that $x=\mu(p)\in f^{-1}(o)$ and $v = v_\alpha \in \QM_p(Y,E)$ for some $\alpha = (\alpha_1,\cdots,\alpha_r,0,\cdots, 0)\in\bR_{>0}^n$, where $\alpha_1,\cdots,\alpha_r$ are $\bQ$-linearly independent. Then $v(s) = \la\alpha,\fv(s)\ra$ for any $s\in \Ik(X)^*$ by Lemma \ref{Lemma. good valuation associated to a quasi-monomial valuation}. 
Hence $G_v(\beta) =\langle\alpha,\beta\rangle$ for any $\beta\in \interior(\BO)_\IQ $ by \eqref{Eqnarray. Okounkov body, by discrete set}, and the equality holds for any $\beta\in \interior(\BO)$ by the concavity of $G_v$. 
\end{proof}

\begin{thm}
\label{Theorem. Convergence of vol and DH}
For any filtration $\CF$ on $R$ and $t\in \IR$, the limit (\ref{Eqnarray. Def of volume function}) exists. 
Let $\BO=\BO(R)$ and $\BO^{(t)} = \BO(\CF^{(t)}R)$ be Okounkov bodies constructed using $\fv$. Then we have 
\begin{eqnarray}
\label{Eqnarray. vol(F,t) = n! vol(O - O^t)}
\vol(\CF;t) = n! \cdot \vol(\BO\setminus \BO^{(t)}).  
\end{eqnarray}
In particular, the volume function $\vol(\CF;t)$ is continuous on $\IR$. 
\end{thm}

\begin{proof}
Let $v \in \Val_{X,o}^*$ be a quasi-monomial valuation of rational rank $n$, $\fv$ be the associated good valuation constructed in Lemma \ref{Lemma. good valuation associated to a quasi-monomial valuation}, and $\BO\seq \IR^n$ be the Okounkov body of $R$ constructed using $\fv$. 
By Lemma \ref{Lemma. G_v is linear}, there exists $\alpha=(\alpha_1,\cdots,\alpha_n) \in \IR^n$, where $\alpha_1,\cdots,\alpha_n\in \IR_{>0}$ are $\IQ$-linearly independent, such that 
\begin{eqnarray*}
G_{v}(\beta) = \alpha_1\beta_1 +\cdots + \alpha_n\beta_n
\end{eqnarray*}
for any $\beta=(\beta_1,\cdots, \beta_n) \in \interior(\BO)$. 

Since $\CF_{v}$ is right linearly bounded by Theorem \ref{Proposition. F_v is linearly bounded if A(v)<infty} and $\CF$ is left linearly bounded, we see that there exists $A, B >0$ such that $v(s) \le A\ord_\CF(s) + Bm$ for any $m\in l_0\IN$ and $s\in R_m$. In other words, $\CF_{v}^{>m(At+B)} R_m \seq \CF^{>mt}R_m$. Hence 
\begin{eqnarray*}
\fv(\CF_{v}^{>m(At+B)}R_m^*) \seq 
\fv(\CF^{>mt}R_m^*) \seq 
\fv(R_m^*) \seq \IN^n.
\end{eqnarray*}
Hence  
\begin{eqnarray*}
\fv(\CF_{v}^{>m(At+B)}R_m^*) = \{\beta\in \fv(R_m^*) \mid \alpha_1\beta_1 +\cdots + \alpha_n\beta_n > m(At+B) \}. 
\end{eqnarray*}

Next, we apply \cite[Proposition 2.1]{LM09} to prove the convergence of (\ref{Eqnarray. Def of volume function}). Fix $t\in\IR$. Consider the following subsets of $\IN^n\times l_0\IN$: 
\begin{eqnarray*}
\Gamma &:=& \{(\beta, m)\in \IN^n\times l_0\IN \mid \beta \in \fv(R_m^*)\setminus \fv(\CF_{v}^{>m(At+B)}R_m^*)\}, \\
\Gamma(\CF) &:=& \{(\beta, m)\in \IN^n\times l_0\IN \mid \beta \in \fv(\CF^{mt}R_m^*)\setminus \fv(\CF_{v}^{>m(At+B)}R_m^*)\}.
\end{eqnarray*}
Then $\Gamma\cap (\IN^n\times\{0\}) = \{(0,0)\}$ since $R_0/\CF_v^{>0}R_0 = R_0/\fm_o \cong \Ik$. If $(\beta,m), (\beta',m')\in \Gamma$, then $\beta=\fv(s), \beta'=\fv(s')$ for some $s\in R_m, s'\in R_{m'}$ satisfying $v(s)\le mt, v(s')\le m't$. Hence $ss'\in R_{m+m'}$, $\fv(ss') = \beta+\beta'$ and $v(ss')\le (m+m')t$. We see that $(\beta+\beta',m+m')\in \Gamma$. In other words, $\Gamma\seq \IN^n\times l_0\IN$ is a sub-semigroup, and it satisfies \cite[(2.3)]{LM09}. Similarly, $\Gamma(\CF) \seq \Gamma$ is also a sub-semigroup satisfying \cite[(2.3)]{LM09}. 

We denote by  
$\Gamma_m=\Gamma\cap (\IN^n\times \{m\}), 
\Gamma_m(\CF)=\Gamma(\CF) \cap (\IN^n\times \{m\}). $
Then the corresponding Okounkov bodies 
\begin{eqnarray*}
\BO(\Gamma) := \overline{\bigcup_m \BO(\Gamma_m)}, \quad 
\BO(\Gamma(\CF)) := \overline{\bigcup_m \BO(\Gamma_m(\CF))}
\end{eqnarray*}
coincide with $\BO_{(\le At+B)}(v)$ and $\BO_{(\le At+B)}(v) \cap \BO^{(\ge t)}(\CF)$ respectively. 

Note that 
\begin{eqnarray*}
\Gamma \seq \{(\beta, m)\in \IN^n\times l_0\IN \mid \alpha_1\beta_1 +\cdots + \alpha_n\beta_n \le m(At+B) \}. 
\end{eqnarray*}
The right hand side is clearly a sub-semigroup of $\IN^n\times l_0\IN$, and it is generated by finitely many $(v_i, l_0) \in \IN^n\times l_0\IN$ for $B>0$ large enough. Hence by enlarging $B$ in the definition of $\Gamma(\CF)\seq \Gamma$, they both satisfy \cite[(2.4)]{LM09}. 

Finally, they also satisfy \cite[(2.5)]{LM09} when $A, B$ are large enough (since $\fv$ is a good valuation). Indeed, there exists $f_i\in \Ik(X)$ such that $\fv(f_i) = e_i$, where $\{e_1, \cdots, e_n\} \seq \IN^{n}$ is the standard basis. We may pick sufficiently large $k$ and $s_0\in \fm_o^k R_m$ such that $s_0, f_i s_0 \in \CF^{mt}R_m$ for any $1\le i\le n$. We denote by $\beta_0 = \fv(s_0)$, then $\fv(f_is_0) = e_i +\beta_0$. Choose $s_1 \in \CF^{l_0}R_{l_0}$ and let $\beta_1 = \fv(s_1)$. 
Enlarging $A,B$ in the definition of $\Gamma$ and $\Gamma(\CF)$ such that $v(s_0), v(f_is_0) < m(At+B), v(s_1) < l_0(At+B)$, then we have $(\beta_0,m), (e_i+\beta_0, m), (\beta_1, l_0) \in \Gamma(\CF)$. Hence the group generated by $\Gamma(\CF)$ contains $(e_i, 0), 1\le i\le n$ and $(0, l_0)$. We see that $\Gamma(\CF)$ (and $\Gamma\supseteq \Gamma(\CF)$) generates $\IZ^n\times l_0\IZ$ as a group \cite[(2.5)]{LM09}. 

Hence by \cite[Proposition 2.1]{LM09}, the following limits exist
\begin{eqnarray*}
\mathop{\lim}_{m\to \infty} \frac{\#(\Gamma_m)}{m^n} = \vol(\BO(\Gamma)),  \quad \text{and} \quad
\mathop{\lim}_{m\to \infty} \frac{\#(\Gamma_m(\CF))}{m^n} = \vol(\BO(\Gamma(\CF))). 
\end{eqnarray*}
Since $\fv$ has one-dimensional leaves, by Lemma \ref{Lemma. v(V quot W) = v(V) minus v(W)} and \cite[Proposition 2.4]{KK12}, we have 
\begin{eqnarray*}
\label{Eqnarray. dimension = number of points}
\ell(R_m/\CF^{>mt}R_m)
= \#(\bar{\fv}(R_m/\CF^{>mt}R_m)^*)
&=& \#(\fv(R_m^*) \setminus \fv(\CF^{>mt}R_m^*)) \\
&=& \#(\Gamma_m) - \#(\Gamma_m(\CF)). 
\end{eqnarray*}
Hence the limit in (\ref{Eqnarray. Def of volume function}) exists, that is,  
\begin{eqnarray*}
\mathop{\lim}_{m\to \infty} \frac{\ell(R_m/\CF^{>mt} R_m)}{m^n} 
= 
\mathop{\lim}_{m\to \infty} \frac{\#(\Gamma_m)-\#(\Gamma_m(\CF))}{m^n} 
&=& \vol(\BO(\Gamma)\setminus \BO(\Gamma(\CF))) \\
&=& \vol(\BO\setminus \BO(\CF^{(t)})). 
\end{eqnarray*}
We see that $t\mapsto \vol(\BO\setminus \BO(\CF^{(t)}))$ is continuous since $\vol(\BO\setminus \BO(\CF^{(t)})) = \vol\{\beta\in \BO\mid G_\CF(\beta) \le t\}$ and $G_\CF$ is continuous on $\interior(\BO)$. 
The proof is finished. 
\end{proof}

\begin{cor}
\label{Corollary. vol(v;t) differentiable}
Let $v \in \Val_{X,o}^*$ be a quasi-monomial valuation. Then the volume function $\vol(v;t)$ is differentiable for $t>0$. Indeed, let $\BO$ be the Okounkov body of $R$ constructed using the good valuation $\fv$ associated to $v$. Then 
\begin{eqnarray}
\label{Eqnarray. dif vol(v,t) = n! dist(0, partial O_t) vol_(n-1)(partial O_t) }
\frac{\dif}{\dif t}\vol(v;t) = n! \cdot d\cdot \vol_{n-1}(\partial \BO_{(\le t)}(v)), 
\end{eqnarray}
where $d$ is the Euclidean distance from the origin to the hyperplane $\{G_v = 1\}$ (recall that $G_v|_{\interior(\BO)}$ is linear by Lemma \ref{Lemma. G_v is linear}). 
\end{cor}

\begin{proof}
Fix $t\in \IR$. By Theorem \ref{Theorem. Convergence of vol and DH}, $\BO_{(\le t)}(v) \seq \IR^n_{\ge 0}$ is bounded, which is empty if $t<0$. If $t>0$, then $\{G_v = t\} \cap \interior(\BO) \ne \varnothing$. A direct computation shows that 
\begin{eqnarray*}
\frac{\dif}{\dif t}\vol(\BO_{(\le t)}(v)) = d\cdot \vol_{n-1}(\partial \BO_{(\le t)}(v)), 
\end{eqnarray*}
and we get (\ref{Eqnarray. dif vol(v,t) = n! dist(0, partial O_t) vol_(n-1)(partial O_t) }) by Theorem \ref{Theorem. Convergence of vol and DH}. 
\end{proof}

\begin{cor}
Let $v \in \Val_{X,o}^*$ be a quasi-monomial valuation. Then the function 
\begin{eqnarray*}
t\mapsto \Big(\frac{\dif}{\dif t}\vol(v;t)\Big)^{\frac{1}{n-1}}
\end{eqnarray*}
is concave for $t>0$. 
\end{cor}
\begin{proof}
Let $\BO$ be the Okounkov body of $R$ constructed using the good valuation $\fv$ associated to $v$. Let $\BF_{(t)}\coloneqq \partial \BO_{(\le t)}(v)$. For any $0<x<y$ and $0< a<1$, we have 
\begin{eqnarray*}
\vol_{n-1}(\BF_{((1-a)x+ay)})^{\frac{1}{n-1}} \ge 
(1-a)\vol_{n-1}(\BF_{(x)})^{\frac{1}{n-1}}+ 
a\vol_{n-1}(\BF_{(y)})^{\frac{1}{n-1}}
\end{eqnarray*}
by the Brunn-Minkowski inequality. We conclude by (\ref{Eqnarray. dif vol(v,t) = n! dist(0, partial O_t) vol_(n-1)(partial O_t) }). 
\end{proof}

The following result is our key tool for the study of $\BH$-invariants of Fano fibration germs. It will be used to prove the continuity on quasi-monomial cones (Theorem \ref{Theorem. v to vol(v;1) is continuous}) and the properness estimate (Theorem \ref{Theorem. Properness}). 

\begin{thm}
\label{Theorem. monotonicity of vol/t^k}
Let $v\in \Val_{X,o}^*$ be a quasi-monomial valuation. Then the function $$t\mapsto \frac{\vol(v;t)}{t^n}$$ is non-increasing for $t>0$. 
\end{thm}

\begin{proof} 
Let $\BO$ be the Okounkov body of $R$ constructed using the good valuation $\fv$ associated to $v$. Denote by $\BO_{(t)}\coloneqq \BO_{(\le t)}(v)$ and $\BF_{(t)}\coloneqq \partial \BO_{(\le t)}(v)$. Let 
$\BC_{(t)}$
be the cone over $\BF^{(t)}$ with vertex $0\in \IR^n$. Then by the convexity of $\BO_{(t)}$, we know that 
$$\BC_{(t)} \seq \BO_{(t)}. $$
Let $d$ be the Euclidean distance of $0$ to the hyperplane $\{G_v = 1\}$. Then 
\begin{eqnarray*}
\frac{\dif}{\dif t} \vol(\BO_{(t)})
&=& d\cdot  \vol_{n-1}(\BF_{(t)}), \\
\vol(\BC_{(t)})
&=& \frac{1}{n} \cdot d\cdot t\cdot \vol_{n-1}(\BF_{(t)}). 
\end{eqnarray*} 
Hence 
$$ \frac{t}{n}\cdot \frac{\dif}{\dif t} \vol(\BO_{(t)}) = \vol(\BC_{(t)}) \le \vol(\BO_{(t)}), $$
and we conclude that 
\[
\frac{\dif}{\dif t} \frac{\vol(\BO_{(t)})}{t^n} = \frac{n}{t^{n+1}} \big(\frac{t}{n}\cdot \frac{\dif}{\dif t}\vol(\BO_{(t)}) - \vol(\BO_{(t)})\big)\le 0. \qedhere
\]
\end{proof}

We have the following result about the volume of a closed point. Recall that for any closed point $x\in f^{-1}(o)$, the filtration $\CF_x$ is defined by $\CF_x^\lam R_m = f_*(\CO_X(mL)\otimes \fm^{\lceil\lam\rceil}_x)$.

\begin{lem}
\label{Lemma. volume of a closed point}
Let $f: X \to Z\ni o$ be a fibration germ with $\dim X = n$ and let $L$ be an $f$-ample $\IQ$-Cartier $\IQ$-divisor on $X$. Let $l_0\in\bZ_{>0}$ such that $l_0L$ is $f$-very ample. For any closed point $x\in f^{-1}(o)$ and $0<t< l_0^{-1}$, we have
\begin{eqnarray*}
\vol(\CF_x;t) =t^n \cdot \mult(\fm_x) >0. 
\end{eqnarray*}
\end{lem}

\begin{proof}
Since $l_0L$ is very ample, for any $x\in f^{-1}(o)$, $\CO_X(l_0L) \otimes \fm_x$ is generated by global sections. In other words, $l_0\pi^*L - E$ is base point free (hence is nef over $Z$), where $\pi:Y\to X$ is the blowup of $\fm_x$, and $E\seq Y$ be the exceptional divisor satisfying $\fm_x \cdot \CO_Y = \CO_Y(-E)$. Hence $al_0\pi^*L-E$ is ample over $Z$ for any $a\in \IQ_{>1}$. Then for any $j>0$ and sufficiently divisible $k>0$, we have 
$$\pi_*\CO_Y(-k E) = \fm_x^k, \quad
R^j\pi_*\CO_Y(-k E) = 0, \quad
R^j (f\circ \pi)_*\CO_Y(k(al_0\pi^*L-E)) = 0. $$
Hence, using Leray spectral sequence, we have 
$$
R^j f_*(\CO_X(kal_0 L)\otimes \fm_x^k) = 0. $$
Let $j=1$. We get the following short exact sequence: 
\begin{eqnarray*}
0\to f_*(\CO_X(kal_0 L)\otimes \fm_x^k)
\to f_*\CO_X(kal_0 L)
\to f_*(\CO_{X,x}/\fm_x^k) \to 0. 
\end{eqnarray*}
Hence we have 
\begin{eqnarray*}
\ell(R_{kal_0} / \CF_x^k R_{kal_0}) = \ell(\CO_{X,x}/\fm_x^k). 
\end{eqnarray*}
Multiplying $n!/(kal_0)^n$ and letting $k\to \infty$, we get 
\begin{eqnarray}
\label{Eqnarray. vol(F_x,t) = t^n mult(m_x)}
\vol(\CF_x; \frac{1}{al_0}) = \frac{1}{(al_0)^n}\mult(\fm_x). 
\end{eqnarray}
By the continuity of $\vol(\CF_x;t)$, we see that (\ref{Eqnarray. vol(F_x,t) = t^n mult(m_x)}) holds for all $a\in \IR_{>1}$. 
\end{proof}

\begin{cor}
\label{Corollary. lam_min(v)=0}
For any $v\in \Val_{X,o}^*$, we have $\lam_\min(v)=0$. 
\end{cor}
\begin{proof}
For any $t\le 0$, we have $\CF_{v}^{mt}R_m =R_m$. Hence $\vol(\CF_v; t) = 0$ and $\lam_\min(\CF_v) \ge 0$. On the other hand, for any closed point $x\in C_X(v)$, there exists $C>0$ such that $v(s)\le C \cdot \ord_x(s)$ for any $s\in R$. Hence $0<\vol(\CF_x; \frac{t}{C}) \le \vol(v;t)$ for any $0<t\ll 1$ by (\ref{Eqnarray. vol(F_x,t) = t^n mult(m_x)}), that is, $\lam_\min(\CF_v)\le 0$. 
\end{proof}


\subsection{Duistermaat-Heckman measures}

In this subsection, we study Duistermaat-Heckman measures associated to filtrations on a Fano fibration germ. 

\subsubsection{DH-measures of filtrations}
In the following, the differentiation of functions is taken in the sense of distributions \cite[Definition 3.1.1]{Hor03}. We denote by $C^\infty_0(\IR)$ the $\IR$-vector space of smooth functions with compact support on $\IR$, and by $L^1_{\rm loc}(\IR)$ the $\IR$-vector space of locally integrable functions on $\IR$. Let $f: \IR\to \IR$ be a locally integrable function. Then its differential distribution $\frac{\dif}{\dif t} f(t)$ is defined to be the linear functional
\begin{eqnarray} 
\label{Eqnarray. definition of differential distribution}
C^\infty_0(\IR) \to \IR,\quad g\mapsto \int_\IR g(t) \frac{\dif}{\dif t}f(t) \cdot\dif t 
:= 
- \int_\IR g'(t) f(t) \cdot\dif t. 
\end{eqnarray}
If $f$ is non-decreasing, then $\frac{\dif}{\dif t} f(t)$ is a positive distribution (that is, the linear functional (\ref{Eqnarray. definition of differential distribution}) is non-negative for any $g\ge 0$). So it is a measure on $\IR$ by \cite[Theorem 2.1.7]{Hor03}. 

Let $\CF$ be a filtration on $R$. Then the following functions are non-decreasing on $\IR$: 
\begin{eqnarray*} 
f_m(t) \coloneqq \frac{\ell( R_m/\CF^{mt} R_m)}{m^n/n!}, \quad \text{and} \quad
f_\infty(t) \coloneqq \vol(\CF;t). 
\end{eqnarray*}
They are locally integrable since $f_m$ is piecewise constant and $f_\infty$ is continuous by Theorem \ref{Theorem. Convergence of vol and DH}. 

\begin{defi}\rm
We define the following measures on $\IR$: 
\begin{eqnarray} 
\DH_{\CF,m} 
&:=& \frac{\dif}{\dif t} \frac{\ell( R_m/\CF^{mt} R_m)}{m^n/n!}
\,\,\,=\,\,\, \sum_{t\in\IR} \delta_{t} \cdot \frac{\ell( \Gr_\CF^{mt} R_m)}{m^n/n!}, \\ 
\label{Eqnarray. DH_F = n! G_* LE}
\DH_\CF 
&:=& \frac{\dif}{\dif t} \vol(\CF;t)
\,\,\,=\,\,\, n!\cdot G_{\CF, *} \LE, 
\end{eqnarray}
where $\delta_t$ is the Dirac measure at $t\in \IR$, $\LE$ is the Lebesgue measure on $\BO$ and the last equality follows from (\ref{Eqnarray. vol(F,t) = n! vol(O - O^t)}) and $\BO^{(t)} = \{\By \in \BO\mid G_\CF(\By)\ge t\}$. 
The measure $\DH_\CF$ is called the {\it Duistermaat-Heckman (DH) measure} of $\CF$. 
\end{defi}

\begin{thm}\label{Theorem. Weak convergence of DH}
The sequence $\DH_{\CF,m}$ converges weakly to $\DH_\CF$ as $m\to \infty$, 
that is, 
\begin{eqnarray}
\label{Eqnarray. DH_F,m to DH_F weak convergence definition}
\mathop{\lim}_{m\to \infty}\int_\IR g(t) \DH_{\CF,m}(\dif t) = \int_\IR g(t) \DH_{\CF}(\dif t)
\end{eqnarray} 
for any smooth function $g$ with compact support. 
\end{thm}
\begin{proof}
By Theorem \ref{Theorem. Convergence of vol and DH}, we have pointwise convergence $f_m\to f_\infty$ as $m\to \infty$ (in the case of $\dim Z\ge 1$; if $\dim Z=0$, the convergence holds except at $t= \lam_\max(\CF)$). Next we prove
\begin{eqnarray}
\label{Eqnarray. g_m to g_infty  L^1_loc convergence}
\mathop{\lim}_{m\to \infty}\int_{\lam_\min(\CF)}^T |f_m(t)-f_\infty(t)|\dif t =0 
\end{eqnarray}  
for any $T\ge \lam_\min(\CF)$, in other words, $f_m\to f_\infty$ in $L^1_{\rm loc}(\IR)$. Since $f_m(t)$ is non-decreasing in $t$ for any $m\in l_0\IN$, we have $f_m(t) \le f_m(T)$ for any $t\le T$. By the convergence $f_m(T) \to f_\infty(T)$, there exists a constant $m_0\in l_0\IN$ and $C_0>0$ such that $f_m(T) \le C_0$ for any $m\ge m_0$. We see that $f_m(t) \le C_0$ for any $t\le T$ and $m\ge m_0$. The convergence (\ref{Eqnarray. g_m to g_infty  L^1_loc convergence}) follows from the dominated convergence theorem. Hence $f_m\to f_\infty$ in the sense of distributions. Then $\frac{\dif}{\dif t}f_m\to \frac{\dif}{\dif t}f_\infty$ as distributions. We conclude by \cite[Theorem 2.1.9]{Hor03} that $\frac{\dif}{\dif t}f_m\to \frac{\dif}{\dif t}f_\infty$ as measures, that is, (\ref{Eqnarray. DH_F,m to DH_F weak convergence definition}) holds. 
\end{proof}


\begin{lem}
\label{Lemma. DH-measure rescaling and shifting}
For any measurable function $g$ on $\IR$, $a\in R_{>0}$ and $b\in\IR$, we have 
\begin{eqnarray*} 
\int_\IR g(t) \DH_{a\CF}(\dif t) 
&=& \int_\IR g(at) \DH_{\CF}(\dif t), \\
\int_\IR g(t) \DH_{\CF(b)}(\dif t) 
&=& \int_\IR g(t+b) \DH_{\CF}(\dif t). 
\end{eqnarray*}
\end{lem}

\subsubsection{DH-measures compatible with two filtrations}
We also consider the DH-measures compatible with two filtrations on $\IR$ similar to \cite[Section 3.1.3]{BLXZ23}. Here, we only treat the case where one of the two filtrations is induced by a quasi-monomial valuation. 
Let $f:\IR^2\to \IR$ be a locally integrable function. Its differential distribution $\frac{\partial^2}{\partial x\partial y}f$ is defined by the linear functional 
$C^\infty_0(\IR^2) \to \IR$: 
\begin{eqnarray} 
\label{Eqnarray. definition of differential distribution on R^2}
g\mapsto
\int_{\IR^2} g(x,y) \frac{\partial^2}{\partial x\partial y}f(x,y) \cdot \dif x \dif y
:= 
\int_{\IR^2} g_{xy}(x,y) f(x,y) \cdot \dif x \dif y, 
\end{eqnarray}
where $g_{xy}=\frac{\partial^2 g}{\partial x\partial y}$. If $f$ is non-decreasing in both $x$ and $y$, then $\frac{\partial^2}{\partial x\partial y}f$ is a positive distribution, hence a measure on $\IR^2$. If $f$ is non-decreasing in $x$ and non-increasing in $y$, then $-\frac{\partial^2}{\partial x\partial y}f$ is a measure on $\IR^2$. 

Let $v\in \Val_{X,o}^*$ be a quasi-monomial valuation and $\CF$ be a filtration on $R$. Let 
\begin{eqnarray} 
f_m(x,y) &=& \frac{\ell(\CF^{my}(R_m/ \CF_v^{mx} R_m))}{m^n/n!}, \\
\label{Eqnarray. Def of double volume function}
f_\infty(x,y) &=& \mathop{\limsup}_{m\to \infty} f_m(x,y) 
\,\,\,=:\,\,\, \vol(\CF^{(y)}(R/\CF_v^{(x)}R))
\end{eqnarray}
for any $(x,y)\in\IR^2$, Then $f_m$ is non-decreasing in $x$ and non-increasing in $y$, and so is $f_\infty$.

Following the argument of \cite[Proposition 3.5]{BLXZ23}, we set 
\begin{eqnarray*}  P_m := \Supp (f_m) \seq \IQ_{\ge 0}\times \IQ,\quad  P = P(v,\CF) := \overline{\cup_{m\in l_0\IN} P_m} \seq  \IR_{\ge 0}\times \IR. \end{eqnarray*}

\begin{prop}
The set $P$ is convex and $\interior(P) = \cup_{m\in l_0\IN} \interior(P_m)$. 
\end{prop}
\begin{proof}
For any $(x,y)\in P_m, (x',y')\in P_{m'}$ and $k\in \IZ_{>0}$, we have $(x,y)\in P_{km}$ and  
$$((1-t)x+tx', (1-t)y+ty') \in P_{m+m'}, $$
where $t = \frac{m'}{m+m'}$, that is, $\CF^{my+m'y'}(R_{m+m'}/\CF_v^{mx+m'x'}R_{m+m'}) \ne 0$ since 
$$\CF^{my}R_m \cdot \CF^{m'y'}R_{m'}\seq\CF^{my+m'y'}R_{m+m'}, \quad \nsubseteq \CF_v^{mx+m'x'}R_{m+m'}. $$
Otherwise $v(\CF^{my}R_m) \ge mx$ or $v(\CF^{m'y'}R_{m'}) \ge m'x'$, which contradicts 
$$\CF^{my}(R_m/\CF_v^{mx}R_m)) \ne 0 \text{ or }\CF^{m'y'}(R_{m'}/\CF_v^{m'x'}R_{m'})) \ne 0. $$
Hence $\cup_{m\in l_0\IN} P_m$ is convex, and so is its closure $P= \overline{\cup_{m\in l_0\IN} P_m}$. 

Next, we prove $\interior(P) = \cup_{m\in l_0\IN} \interior(P_m)$, where the inclusion $\supseteq$ is clear. Conversely, fix $(x_0, y_0) \in \interior(P)$. Since $\interior(P)$ is open, there exists $\vep > 0$ such that $(x_0-\vep, y_0+\vep) \in \interior(P)\seq P$. 
Since $P$ is the closure of $\cup_{m\in l_0\IN} \interior(P_m)$, there exists $(x,y)\in P_m$ such that $x_0-\frac{\vep}{2} > x$, $y_0+\frac{\vep}{2} < y$. Hence $f_m(x_0,y_0) \ge f_m(x,y) >0$, since $f_m$ is non-decreasing in $x$ and non-increasing in $y$. We conclude that $(x_0,y_0)\in \cup_{m\in l_0\IN} P_m$. 
\end{proof}

\begin{thm}
\label{Theorem. Convergence of double vol and DH}
The limit in (\ref{Eqnarray. Def of double volume function}) exists for any $(x,y)\in \IR^2 \setminus \partial P$, that is, 
$$\vol(\CF^{(y)}(R/\CF_v^{(x)}R)) = \mathop{\lim}_{m\to \infty} \frac{\ell(\CF^{my}(R_m/ \CF_v^{mx} R_m))}{m^n/n!}. $$
Moreover, the function $f_\infty(x,y) = \vol(\CF^{(y)}(R/\CF_v^{(x)}R))$ is continuous on $\IR^2 \setminus \partial P$. 
\end{thm} 

\begin{proof}
If $(x,y) \in \IR^2\setminus P$, then $f_m(x,y)=0$ and the limit clearly exists. We assume that $(x,y) \in \interior(P)$. 
We follow the same argument and notation as Theorem \ref{Theorem. Convergence of vol and DH}. Let $\fv$ be the good valuation associated to the quasi-monomial valuation $v$. 
Then the following subset of $\IN^n\times l_0\IN$
\begin{eqnarray*}
\Gamma 
&:=& \{(\beta, m)\in \IN^n\times l_0\IN \mid \beta \in \fv(\CF^{my}R_m^*)\setminus \fv(\CF_{v}^{mx}R_m^*)\} \\
&=& \{(\beta, m)\in \IN^n\times l_0\IN \mid \beta \in \fv(\CF^{my}R_m^*), \alpha_1\beta_1+\cdots + \alpha_r\beta_r < mx\}
\end{eqnarray*}
is a sub-semigroup, which satisfies
\begin{eqnarray*}
\label{Eqnarray. level dimension = number of points}
\ell(\CF^{my}(R_m/\CF_v^{mx}R_m))
&=& \#(\bar{\fv}(\CF^{my}(R_m/\CF_v^{mx}R_m))^*) \\
&=& \#\fv(\CF^{my}R_m^*)\setminus \fv(\CF_{v}^{mx}R_m^*)
= \#(\Gamma_m)
\end{eqnarray*}
by Lemma \ref{Lemma. v(V quot W) = v(V) minus v(W)} and \cite[Proposition 2.4]{KK12}. Hence, using the same argument as Theorem \ref{Theorem. Convergence of vol and DH}, the following limit exists 
\begin{eqnarray*}
\mathop{\lim}_{m\to\infty}
\frac{\ell(\CF^{my}(R_m/\CF_v^{mx}R_m))}{m^n} 
= 
\mathop{\lim}_{m\to\infty}
\frac{\#(\Gamma_m)}{m^n} 
&=&
\vol(\BO(\Gamma)) \\
&=&
\vol(\BO(\CF^{(y)}R) \setminus \BO(\CF_v^{(x)}R)). 
\end{eqnarray*}
We see that $(x,y)\mapsto \vol(\BO(\CF^{(y)}R) \setminus \BO(\CF_v^{(x)}R))$ is continuous on $ \interior(P)$ since 
$$\BO(\CF^{(y)}R) \setminus \BO(\CF_v^{(x)}R) = \{\beta\in \BO\mid G_v(\beta) < x, G_\CF(\beta) \ge y\}$$ 
has non-empty interior for any $(x,y)\in \interior(P)$ and $G_v, G_\CF$ are continuous on $\interior(\BO)$. 
The proof is finished. 
\end{proof}

For any $x\ge 0$, let  
$$\lam^{(x)}_\max(v;\CF) := \sup\{y\in \IR \mid \CF^{my}(R_m/\CF_{v_0}^{mt}R_m)\ne 0, \exists m\in l_0\IN\}. $$

\begin{thm}
\label{Theorem. vol^(1/n) is concave}
For any $x\ge0$, the function $$y\mapsto \vol(\CF^{(y)}(R/\CF_v^{(x)}R))^{\frac{1}{n}}$$ is concave for $y\in (-\infty, \lam_\max^{(x)}(v;\CF))$. 
\end{thm}
\begin{proof}
For any $m,m'\in l_0\IN$ and $y_0, y_1\in(-\infty, \lam_\max^{(x)}(v;\CF))$, we have
\begin{eqnarray*}
& & \fv(\CF^{my_0}(R_{m}/\CF_v^{mx}R_{m})^*) + \fv(\CF^{my_1}(R_{m'}/\CF_v^{m'x}R_{m'})^*)  \\
&=& \Big(\fv(\CF^{my_0}R_{m}^*) \setminus \fv(\CF_v^{mx}R_{m}^*)\Big) + 
\Big(\fv(\CF^{m'y_1}R_{m'}^*) \setminus \fv(\CF_v^{m'x}R_{m'}^*)\Big) \\
&\seq& 
\fv(\CF^{my_0+m'y_1}R_{m+m'}^*) \setminus \fv(\CF_v^{(m+m')x}R_{m+m'}^*)\\
&=& \fv(\CF^{my_0+m'y_1}(R_{m+m'}/\CF_v^{(m+m')x}R_{m+m'})^*), 
\end{eqnarray*}
where the equalities follow from Lemma \ref{Lemma. v(V quot W) = v(V) minus v(W)} and the inclusion is obtained by noting that 
$v(s) < mx$ and $v(s') < m'x$ implies $v(ss')=v(s)+v(s') <(m+m')x$ for any $s\in R_m$ and $s'\in R_{m'}$ (which does not hold for general filtrations). 

For any rational $0<t<1$, we may choose sufficiently divisible $m,m'$ satisfying $m'=t(m+m')$. Dividing by $m+m'$ and letting $m\to \infty$, we see that  
\begin{eqnarray*}
&&(1-t)\cdot \BO(\CF^{(y_0)}(R/\CF_v^{(x)}R)) + t\cdot \BO(\CF^{(y_1)}(R/\CF_v^{(x)}R)) \\
&\seq& \BO(\CF^{((1-t)y_0+ty_1)}(R/\CF_v^{(x)}R)). 
\end{eqnarray*}
By the Brunn-Minkowski inequality, we conclude that 
\begin{eqnarray*}
&&\vol(\BO(\CF^{((1-t)y_0+ty_1)}(R/\CF_v^{(x)}R)))^{\frac{1}{n}} \\
&\ge& (1-t)\cdot \vol(\BO(\CF^{(y_0)}(R/\CF_v^{(x)}R)))^{\frac{1}{n}} 
+ t\cdot \vol(\BO(\CF^{(y_1)}(R/\CF_v^{(x)}R)))^{\frac{1}{n}}. 
\end{eqnarray*}
\end{proof}

\begin{defi}\rm
We define the following measures on $\IR^2$: 
\begin{eqnarray} 
\label{Eqnarray. discrete DH measure compatible with F and G}
\DH_{v,\CF,m} 
&:=& 
-\frac{\partial^2}{\partial x \partial y} \frac{\ell( \CF^{my}(R_m/\CF_v^{mx}R_m))}{m^n/n!}  \\
\nonumber
&=&
\sum_{(x,y)\in \IR^2} \delta_{(x,y)} \cdot \frac{\ell (\Gr_v^{mx}\Gr_\CF^{my} R_m)}{m^n/n!}, \\ 
\DH_{v,\CF} 
&:=& - \frac{\partial^2}{\partial x \partial y} \vol(\CF^{(y)}(R/\CF_v^{(x)}R)), 
\end{eqnarray}
where $\delta_{(x,y)}$ is the Dirac measure at $(x,y)\in \IR^2$. 
The measure $\DH_{v, \CF}$ is called the {\it Duistermaat-Heckman measure compatible with both} $v$ and $\CF$. 
\end{defi}

\begin{thm}
The sequence $\DH_{v,\CF,m}$ converges weakly to $\DH_{v,\CF}$ as $m\to \infty$. 
\end{thm}
\begin{proof}
By Theorem \ref{Theorem. Convergence of double vol and DH}, $f_m\to f_\infty$ converge pointwise away from a set of measure zero $\partial P \seq \IR^2$. Next we prove
\begin{eqnarray}
\label{Eqnarray. f_m to f_infty  L^1_loc convergence 2}
\mathop{\lim}_{m\to \infty}\int_{[0,S]\times [T,+\infty]} |f_m(x,y)-f_\infty(x,y)|\dif x \dif y =0 
\end{eqnarray}  
for any $S>0$ and $T<\lam_\min(\CF)$, in other words, $f_m\to f_\infty$ in $L^1_{\rm loc}(\IR)$. 
We have $(S,T)\in \interior(P)$ and $f_m(x,y) \le f_m(S,T)$ for any $x\le S, y\ge T$ since $f_m$ is non-decreasing in $x$ and non-increasing in $y$. By Theorem \ref{Theorem. Convergence of double vol and DH}, we have convergence $f_m(S,T)\to f_\infty(S,T)$. 
There exists a constant $m_0\in l_0\IN$ and $C_0>0$ such that $f_m(S,T) \le C_0$ for any $m\ge m_0$. We see that $f_m(x,y) \le C_0$ for any $x\le T, y\ge S$ and $m\ge m_0$. The convergence (\ref{Eqnarray. g_m to g_infty  L^1_loc convergence}) follows from the dominated convergence theorem. Hence $f_m\to f_\infty$ in the sense of distributions. Then $- \frac{\partial^2}{\partial x \partial y}f_m\to - \frac{\partial^2}{\partial x \partial y}f_\infty$ as distributions. We conclude by \cite[Theorem 2.1.9]{Hor03} that $- \frac{\partial^2}{\partial x \partial y}f_m\to - \frac{\partial^2}{\partial x \partial y}f_\infty$ as measures.
\end{proof}

By the construction of $\DH_{v,\CF}$, there is also a formulation using concave transforms (Lemma \ref{Corollary. converges to S(v;CF)}) similar as (\ref{Eqnarray. DH_F = n! G_* LE}), which is a consequence of the following lemma of variable changing. 

\begin{lem}
\label{Lemma. real analysis lemma, change variable integration}
Let $\BO\seq \IR^n$ be a bounded measurable set and $f,G:\BO\to \IR$, $g:\IR\to \IR$ be locally integrable functions. Then
\begin{eqnarray*}
\int_\IR g(t) \Big(-\frac{\dif}{\dif t} \int_{\{G\ge t\}} f(\By) \LE(\dif \By)\Big)\dif t 
=
\int_\BO g\circ G(\By) \cdot f(\By) \LE(\dif \By). 
\end{eqnarray*}
\end{lem}

\begin{defi}[Geodesics]\rm
Let $\CF_0, \CF_1$ be filtrations on $R$. The {\it geodesic} $\{\CF_t\}_{0\le t\le 1}$ connecting $\CF_0$ and $\CF_1$ is a collection of filtrations defined by
\begin{eqnarray}
\label{Equality: Geodesic}
\CF^\lam_t R_m 
= \sum_{(1-t)\mu + t\nu\ge \lam}
\CF_0^\mu R_m \cap \CF_1^\nu R_m.
\end{eqnarray}
\end{defi}

\begin{lem}
\label{Lemma: Concave transforms and DH measures of Geodesic}
Assume that $\CF_0 = \CF_v$ for some quasi-monomial valuation $v\in \Val_{X,o}^*$. 
For any measurable function $g$ on $\IR$, we have 
\begin{eqnarray}
\label{Equality: DH measures of Geodesic}
\int_\IR g(z)\DH_{\CF_t}(\dif z) = \int_{\IR^2} g((1-t)x+ty)\DH_{\CF_0,\CF_1}(\dif x \dif y). 
\end{eqnarray}
\end{lem}
\begin{proof}
Let $\ell: \IR^2\to \IR, (x,y)\mapsto (1-t)x+ty$. Then it follows that 
$\ell_*\DH_{\CF_0,\CF_1,m} = \DH_{\CF_t,m}. $
\end{proof}

\begin{lem}
\label{Lemma. v,w d_1-equiv. implies v=w}
Let $v\in \Val_{X,o}^*$ be a quasi-monomial valuation and $w\in \Val_{X,o}^*$. If $$\int_{\IR^2} |x-y| \DH_{v,w}(\dif x \dif y)=0, $$
then $v=w$. 
\end{lem}

\begin{proof}
We modify the proof of \cite[Lemma 3.16]{BLXZ23}. It suffices to show that $v(f) = w(f)$ for any $f\in R_m$. Assume that there exists $l\in l_0\IN$ and $f\in R_l$ such that $a=v(f)>w(f)=b$. We fix $0<\vep<a-b$, $\lam>0$ and $p\in\IZ_{>0}$ such that $\lam\le (a-b-\vep)p$. Consider the following map 
$$\phi: R_k \to R_{(lp+1)k}/\CF_w^{(\lam+bp)k}R_{(lp+1)k} $$
given by multiplying with $f^{pk}$. Then $\ker(\phi)= \CF_w^{\lam k}R_k$ since 
$$g\cdot f^{pk} \in \CF_w^{(\lam+bp)k}R_{(lp+1)k} 
\Leftrightarrow 
w(g)+w(f)\cdot pk \ge \lam k + b\cdot pk 
\Leftrightarrow   
g\in \CF_w^{\lam k}R_k.  $$
Denote the image of $\phi$ by $V_k$. Then $R_k/\CF_w^{\lam k}R_k \cong V_k$. 

For any $g\in R_k$ satisfying $w(g)<\lam k$, we have 
\begin{eqnarray*}  
v(g\cdot f^{pk}) 
&=& v(g) +apk \ge apk, \\
w(g\cdot f^{pk}) 
&=& w(g) + bpk \le (\lam+bp)k \le (a-\vep)pk. 
\end{eqnarray*}
In particular, the function $(v-w)|_{V_k\setminus \{0\}}$ has lower bound $\vep pk>0$. Let $m=(lp+1)k$. Hence 
\begin{eqnarray*}  
&&\int_{\{(x,y)\in \IR^2\mid y\le\frac{\lam+bp}{lp+1}\}} |x-y| \DH_{\CF_v,\CF_w,m}(\dif x \dif y) \\
&=& 
\sum_{(\mu,\nu)\in\IR^2, \nu\le (\lam+bp)k} \frac{|\mu-\nu|}{m} \cdot \frac{\ell(\Gr_v^\mu\Gr_w^\nu R_m)}{m^n/n!}
\ge  \frac{\vep pk}{m} \cdot \frac{\ell(V_k)}{m^n/n!} \\
&=& \frac{\vep p}{(pl+1)^{n+1}} \cdot \frac{\ell(V_k)}{k^n/n!} 
\longrightarrow \frac{\vep p}{(pl+1)^{n+1}} \cdot \vol(\CF_w; \lam) >0
\end{eqnarray*}
as $k\to \infty$ by Corollary \ref{Corollary. lam_min(v)=0}. Note that $\{(x,y)\in \IR^2\mid y\le\frac{\lam+bp}{lp+1}\}\cap P(v,w)$ is bounded. Hence, the limit of the first line exists as $k\to\infty$ by Theorem \ref{Theorem. Convergence of double vol and DH} and 
$$\int_{\IR^2} |x-y| \DH_{v,w}(\dif x \dif y) \ge
\int_{\{y\le\frac{\lam+bp}{lp+1}\}} |x-y| \DH_{v,w}(\dif x \dif y) >0 . $$
We get a contradiction. 
\end{proof}

\subsubsection{Equivariant DH-measures}
We also introduce the following equivariant version of DH-measures. Let $(X,\D)\to Z\ni o$ be a log Fano fibration germ with a good $\IT=\IG_m^r$-action with moment polyhedral $\BP\seq \IR^r$. For any $\IT$-invariant quasi-monomial valuation $v\in\Val_{X,o}^{\IT,*}$ and $\IT$-invariant filtration $\CF$ on $R$, we define: 
\begin{eqnarray*} 
\label{Eqnarray. discrete equivariant DH measure of F}
    \DH_{\BP,v,\CF,m} 
    &=& \sum_{(\alpha,x,y)\in \BP\times \IR\times \IR} \delta_{(\alpha,x,y)} \cdot \frac{\ell (\Gr_v^{mx} \Gr_\CF^{my} R_{m,m\alpha})}{m^n/n!}, 
\end{eqnarray*}
which converges weakly to the {\it equivariant DH-measure $\DH_{\BP,v,\CF}$ compatible with both $v$ and $\CF$} on $\BP\times \IR\times \IR$. We also denote by, for example, $\DH_{\BP,\CF} = \pr_{1,3,*} \DH_{\BP,v,\CF}$ and $\DH_{\BP} = \pr_{1,*} \DH_{\BP,v,\CF}$. For any $\xi \in \reeb$, we have 
\begin{eqnarray*}
    \DH_{\BP,\CF_\xi,m} 
    &=& \sum_{(\alpha,x)\in \BP\times \IR} \delta_{(\alpha,x)} \cdot \frac{\ell (\Gr_\CF^{m(x-\la\alpha,\xi\ra)} R_{m,m\alpha})}{m^n/n!} \\
    &=& \sum_{(\alpha,x)\in \BP\times \IR} \delta_{(\alpha,x+\la\alpha,\xi\ra)} \cdot \frac{\ell (\Gr_\CF^{mx} R_{m,m\alpha})}{m^n/n!}. 
\end{eqnarray*}
\begin{lem}
\label{Lemma. equivariant DH-measure twisting}
For any measurable function $g$ on $\BP\times\IR$ and $\xi\in\reeb$, we have 
\begin{eqnarray*} 
\int_{\BP\times\IR} g(\alpha,x) \DH_{\BP,\CF_\xi}(\dif \alpha \dif x) 
= \int_{\BP\times\IR} g(\alpha,x+\la\alpha,\xi\ra) \DH_{\BP,\CF}(\dif \alpha \dif x). 
\end{eqnarray*}
\end{lem}

Let $v\in \Val_{X,o}^*$ be a special valuation and $(X_v,\D_v,\xi_v)$ be the special degeneration of $(X,\D)$ induced by $v$ Definition \ref{Definition. special degeneration induced by v}. Then we have: 

\begin{lem}
\label{Lemma. DH_v = D_wt}
$\DH_v = \DH_{\wt_{\xi_v}}. $
\end{lem}
\begin{proof}
By definition, we have 
$\Gr_v^\lam R_m = \Gr_{\wt_{\xi_v}}^\lam R_{v,m}. $
Hence $$\DH_{v,m} = \DH_{\wt_{\xi_v},m}. $$
\end{proof}

\subsection{\texorpdfstring{$\BH$}{}-invariants}
\label{Section: Generalized H-invariants}

In this subsection, we generalize the $\BH$-invariants of log Fano varieties \cites{TZZZ13,DS20,HL20} to log Fano fibration germs. 
Recall that we fix a log Fano fibration germ $(X,\D)\to Z\ni o$ with section ring $R=\oplus_{m\in l_0\IN}R_m$. 

Let $\CF$ be a filtration on $R$. For any $m\in l_0\IN$ and $t\in \IR$, we define
\begin{eqnarray*}
\tS_{m,mt}(\CF) 
= -\log\Big(
\frac{1}{m^n/n!} \sum_{x \le t} e^{-x}\cdot \ell(\Gr_{\CF}^{mx} R_m) 
\Big). 
\end{eqnarray*}
By Theorem \ref{Theorem. Convergence of vol and DH}, we have: 

\begin{cor}
For any $t\in \IR$, the following limit exists:  
\begin{eqnarray*}
\tS^{(t)}(\CF)
:= \mathop{\lim}_{m\to \infty} \tS_{m,mt}(\CF) 
= -\log\Big(
\int_{-\infty}^t e^{-x} \cdot \DH_{\CF}(\dif x)
\Big). 
\end{eqnarray*}
\end{cor}

Since $\vol(\CF; t) = \int_{-\infty}^t \DH_\CF(\dif x)$ is of polynomial growth, we have: 
\begin{eqnarray*}
\tS(\CF)
:= \mathop{\lim}_{t\to +\infty} \tS^{(t)}(\CF) 
= -\log\Big(
\int_{-\infty}^{+\infty} e^{-x} \cdot \DH_{\CF}(\dif x)
\Big). 
\end{eqnarray*}
Using integration by parts, we see that 
\begin{eqnarray*}
\tS(\CF)
= -\log\Big(
\int_{-\infty}^{+\infty} e^{-x} \vol(\CF;x)\dif x
\Big). 
\end{eqnarray*}

\begin{defi}[$\BH$-invariants]\rm 
\label{Definition: H invariant}
The {\it $\BH$-functional} of a filtration $\CF$ is defined by
\begin{eqnarray}
\BH(\CF) 
\,\,\,=\,\,\,  \BH_{X,\D}(\CF) 
\,\,\,:=\,\,\,  \mu_{X,\D}(\CF) - \tS(\CF). 
\end{eqnarray}
The $\BH$-invariant of the log Fano fibration germ $(X,\D)\to Z\ni o$ is defined by 
\begin{eqnarray}
\BH(X,\D) 
\,\,\,:=\,\,\, 
\inf_\CF \,\,
\BH(\CF),
\end{eqnarray}
where the infimum runs over all filtrations $\CF$ on $R$. We also denote by 
\begin{eqnarray*}
\BV(\CF) 
\,\,\,=\,\,\, 
e^{\BH(\CF)}
\,\,\,=\,\,\, 
\int_\IR e^{\mu(\CF) - x}\cdot \DH_\CF(\dif x). 
\end{eqnarray*}
For any $v\in \Val_{X,o}^*$, we simply denote by $\BH(v) = \BH(\CF_v)$ and $\BV(v) = \BV(\CF_v)$. 
\end{defi}

\begin{rmk}\rm 
Since $\mu(\CF)$ and $\DH_\CF$ are affine with respect to shifting (Lemma \ref{Lemma. log canonical slope is linear w.r.t rescaling and shifting} and \ref{Lemma. DH-measure rescaling and shifting}), for any $b\in \IR$, we have 
\begin{eqnarray*}
\BH(\CF(b))=\BH(\CF). 
\end{eqnarray*}
\end{rmk}

Next we recall the weighted volume of a Fano fibration germ introduced by \cite{SZ24}. 

\begin{defi}[Weighted volumes]\rm
The {\it weighted volume} of a valuation $v\in \Val_{X,o}^*$ is defined by  
\begin{eqnarray}  
\BW(v) 
\,\,\,=\,\,\, 
\BW_{X,\D}(v) 
&:=&
\int_0^\infty e^{A_{X,\D}(v) - x} \DH_{v}(\dif x). 
\end{eqnarray}
If $v\in\Val_{X,o}$ with $A_{X,\D}(v)=+\infty$, then we set $\BW(v) = +\infty$ by convention. 
The weighted volume of the log Fano fibration germ $(X,\D)\to Z\ni o$ is defined by 
\begin{eqnarray}
\BW(X,\D) 
\,\,\,:=\,\,\, 
\mathop{\inf}_{v\in\Val_{X,o}} \,\,
\BW(v). 
\end{eqnarray}
\end{defi}

\begin{rmk} \rm 
\label{Remark: H = log W}
Since $\mu_{X,\D}(\CF_v)\le A_{X,\D}(v)$ by Lemma \ref{Lemma: mu<A}, we have naturally $\BV(v) \le \BW(v)$. The equality holds if any only if $v$ is weakly special by Theorem \ref{Theorem: weakly special valuations}. We will see in Section \ref{Section. Existence of H-minimizers} that $\BH(X,\D)$ is approximated by weakly special valuations, hence 
\begin{eqnarray}
\BH(X,\D) 
\,\,\,=\,\,\, 
\log(\BW(X,\D)). 
\end{eqnarray}
\end{rmk}

\begin{rmk}\rm
    Our definition of weighted volume differs Sun-Zhang's definition by $n!$ due to a different convention of DH measures. 
    In particular, in the global case when $Z$ is a point, then $(X,\Delta)$ is a log Fano pair, we have 
    $$\BW(X,\Delta)=(-K_X-\Delta)^n\cdot e^{h(X,\Delta)}, $$
    where $h(X,\D)$ is the H-invariant of a log Fano pair as in \cite{HL20} (note that $h(X,\D) = 0$ if and only if $(X,\D)$ is K-semistable). 
    In the local case when $f$ is the identity map and $f:X\cong Z$, we have 
    $$\BW(X,\Delta;o)=n!\frac{e^n}{n^n}\cdot \hvol(X,\Delta,o). $$
    By Stirling's approximation, we see that the coefficients $n!(e/n)^n \sim \sqrt{2\pi n}$ as $n\to +\infty$. 
\end{rmk}

We remark that another definition of $\BW$ via the restricted volume function was introduced by \cite{Oda25}, where explicit values of $\BW$ were computed for several examples. And we establish the following local-to-global comparison, which is stated in \cite[(6.3)]{SZ24}.

\begin{thm}
Let $f\colon (X,\Delta)\rightarrow Z\ni o$ be log Fano fibration germ of dimension $n$. For any closed point $x\in f^{-1}(o)$, we have the following local-to-global estimate:
\begin{eqnarray*}
   \BW(X,\Delta) \leq \BW(X,\Delta;x)=n!\frac{e^n}{n^n}\cdot \hvol(X,\Delta;x).
\end{eqnarray*}
    In particular, for any log Fano fibration germ $f\colon(X,\D)\rightarrow Z\ni o$ of dimension $n$, we have
\begin{eqnarray} 
\label{Eqnarray. max weighted volume}
   \BW(X,\Delta) \leq n!\cdot e^n. 
\end{eqnarray}
\end{thm}
\begin{proof}
We follow the argument as in the proof of Lemma \ref{Lemma. volume of a closed point}. Let $L=-(K_X+\D)$ and $v\in \Val_{X,x}\subset \Val_{X,o}$ be a valuation centered at $x$. Let $\fa_m(v)=\{f\in\CO_X \mid v(f) \geq m\}$ be the ideal sequence associated to $v$ for $m\in \IN$. Then for any $t\in \IQ$ and $m\in \IN$, we have the following exact sequence:
    \begin{eqnarray*}
        0\rightarrow f_*(\CO_X(mL)\otimes \fa_{mt}(v) )  \rightarrow f_*(\CO_X(mL) ) \rightarrow f_*(\CO_{X,x}/\fa_{mt}(v) ).
    \end{eqnarray*}
    Hence, we have 
    \begin{eqnarray*}
        \ell(R_m/\CF_v^{mt}R_m)\leq \ell(\CO_{X,x}/\fa_{mt}(v)).
    \end{eqnarray*}
    Multiplying $n!/m^n$ and letting $m\rightarrow +\infty$, we get 
    \begin{eqnarray*}
        \vol(v;t)\leq t^n\cdot \vol(v). 
    \end{eqnarray*}
Since $\vol(a\cdot v;t)= \vol(v;t/a)$ for any $a>0$, we have 
    \begin{eqnarray*}
        \BW(X,\Delta) 
        &\leq& e^{a \cdot A}\int_0^{+\infty} e^{-t}\cdot \vol(v;t/a ) \, \dif t\\
        &\leq&  e^{a \cdot A}\int_0^{+\infty}e^{-t}\cdot \frac{t^n}{a ^n}\vol(v)\, \dif t 
        \,\,\,=\,\,\ e^{a \cdot A} \frac{n!}{a^n}\vol(v),
    \end{eqnarray*}
    where $A=A_{X,\D}(v)$.  Note that the function $a \mapsto e^{a  A}/a ^n$ is minimized by $a  = n/A$. Thus 
    $$\BW(X,\D) \leq n!\frac{e^n}{n^n}\cdot A_{X,\D}(v)^n \vol(v).$$ 
We get the first inequality by taking infimum for $v\in \Val_{X,x}$. The second inequality follows since $\hvol(Y,\D_Y;y)\le n^n$ for any klt singularity $y\in(Y,\D_Y)$ of dimension $n$ by \cite[Lemma A.1]{LX19}. 
\end{proof}

\subsection{Delta invariants}
In this subsection, we fix a weakly special valuation $v_0\in\Val_{X,o}^*$ with $\mu_0=\mu(v_0)=A_{X,\D}(v_0)$. We will construct a delta invariant with respect to $(X,\D)\to Z\ni o$ and $v_0$ using basis type divisors as introduced by \cite{FO16}. For simplicity, we set 
\begin{eqnarray*}
\BV_{m,mt} &=& \BV_{m,mt}(v_0)
\,\,\,:=\,\,\ \frac{1}{m^n/n!}
\sum_{x \le t} e^{\mu_0-x}\cdot \ell(\Gr_{v_0}^{mx} R_m),  \\
\BV^{(t)} &=& \BV^{(t)}(v_0)
\,\,\,:=\,\,\
\int_{-\infty}^t e^{\mu_0-x} \cdot \DH_{v_0}(\dif x).  
\end{eqnarray*}
Let $\CF$ be a filtration on $R$. We define: 
\begin{eqnarray*}
S_{m,mt}(v_0;\CF) 
= \frac{1}{\BV_{m,mt}}\cdot \frac{1}{m^n/n!}\cdot
 \sum_{x \le t,\, y\in \IR} e^{\mu_0-x}y \cdot \ell(\Gr_{\CF}^{my}\Gr_{v_0}^{mx} R_m). 
\end{eqnarray*}

The asymptotic invariants $S_{m,mt}(v_0;\CF)$ can be viewed as a generalization of the invariants $S_{c,m}(v)$ defined in \cite[Section 3.3]{BLXZ25}. 

By Theorem \ref{Theorem. Convergence of double vol and DH}, we have: 

\begin{cor}
\label{Corollary. converges to S(v;CF)}
For any $t\in \IR$, the following limit exists:  
\begin{eqnarray*}
S^{(t)}(v_0;\CF)
&:=& \mathop{\lim}_{m\to \infty} S_{m,mt}(v_0;\CF) \\
&=& \frac{1}{\BV^{(t)}} \int_{\{x\le t\}}  e^{\mu_0-x} \cdot y\cdot \DH_{v_0,\CF}(\dif x \dif y) \\
&=& \frac{n!}{\BV^{(t)}} \int_{\{G_{v_0}\le t\}}  e^{\mu_0-G_{v_0}(\By)} \cdot G_\CF(\By)\cdot \LE(\dif \By). 
\end{eqnarray*}
\end{cor}
\begin{proof}
We shortly explain the last equality. 
\begin{eqnarray*}
&& \frac{1}{n!}\int_{\{x\le t\}} e^{\mu_0-x} \cdot y \cdot \DH_{v_0,\CF}(\dif x \dif y) \\
&=&  \int_{\{x\le t\}} e^{\mu_0-x} \cdot y \cdot  \Big(- \frac{\partial^2}{\partial x \partial y} \vol\{G_{v_0} \le x, G_\CF \ge y \}\Big) \dif x \dif y \\
&=&  \int_0^t e^{\mu_0-x} \cdot 
\frac{\dif}{\dif x} 
\Big( \int_\IR y\cdot \big(- \frac{\dif}{\dif y} \vol\{G_{v_0} \le x, G_\CF \ge y \}\big) \dif y\Big)
\dif x \\
&=&  \int_0^t e^{\mu_0-x} \cdot 
\frac{\dif}{\dif x} 
\Big( \int_{\{G_{v_0}\le x\}} G_\CF(\By) \cdot \LE(\dif \By) \Big)
\dif x \\
&=& \int_{\{G_{v_0}\le t\}} G_\CF(\By)\cdot e^{\mu_0-G_{v_0}(\By)} \cdot \LE(\dif \By), 
\end{eqnarray*}
where the second equality follows from Fubini's theorem, and the third and fourth follow from Lemma \ref{Lemma. real analysis lemma, change variable integration} (variable changing). 
\end{proof}
 
Recall that for any $t\ge 0$: 
$$\lam^{(t)}_\max(v_0;\CF) := \sup\{y\in \IR \mid \CF^{my}(R_m/\CF_{v_0}^{mt}R_m)\ne 0, \exists m\in l_0\IN\}. $$
Since $\CF$ is linearly bounded by $\CF_{v_0}$, 
there exists an affine linear function $l(t)$ such that
$$\lam^{(t)}_\max(v_0;\CF) \le l(t), $$
hence $\Supp (\DH_{v_0,\CF}) \seq \{(x,y)\mid y\le l(x)\}$ and $$\int_{\{x\le t\}} y \cdot \DH_{v_0,\CF}(\dif x \dif y) \le \int_{0}^t l(x)\DH_{v_0}(\dif x), $$ 
which is of polynomial growth in $t$. We see that the following limit exists: 
\begin{eqnarray*}
S(v_0;\CF)
:= \mathop{\lim}_{t\to +\infty} S^{(t)}(v_0;\CF) 
= \frac{1}{\BV(v_0)}
\int_{\IR^2} e^{\mu_0-x}y \cdot \DH_{v_0,\CF}(\dif x \dif y).
\end{eqnarray*}

\begin{cor}
\label{Corollary. T <= (n+1)S}
For any $t\ge 0$ and filtration $\CF$, we have 
$$\frac{e^{\mu_0-t}}{n+1} \lam^{(t)}_\max(v_0;\CF) 
\le S^{(t)}(v_0;\CF)
\le \lam^{(t)}_\max(v_0;\CF). $$
\end{cor}

\begin{proof}
The second inequality is clear. For the first one, it follows from the concavity of $y\mapsto\vol(\CF^{(y)}(R/\CF_{v_0}^{(t)}))^{\frac{1}{n}} $ (Theorem \ref{Theorem. vol^(1/n) is concave}). For any $0\le y< \lam <\lam^{(t)}_\max(v_0;\CF)$, we have 
\begin{eqnarray*}
\vol(\CF^{(y)}(R/\CF_{v_0}^{(t)}))^{\frac{1}{n}} 
\ge (1-\frac{y}{\lam}) \vol(\CF^{(0)}(R/\CF_{v_0}^{(t)}))^{\frac{1}{n}} 
+ \frac{y}{\lam} \vol(\CF^{(\lam)}(R/\CF_{v_0}^{(t)}))^{\frac{1}{n}}. 
\end{eqnarray*}
Letting $\lam \to \lam^{(t)}_\max(v_0;\CF)$, we get
\begin{eqnarray*}
\vol(\CF^{(y)}(R/\CF_{v_0}^{(t)})) 
&\ge& \BV^{(t)} \Big(1-\frac{y}{\lam^{(t)}_\max(v_0;\CF)}\Big)^n. \\
\end{eqnarray*}
Hence 
\begin{eqnarray*}
S^{(t)}(v_0;\CF)
&=& \frac{1}{\BV^{(t)}} \int_{\{x\le t\}} e^{\mu_0-x} \cdot y\cdot \DH_{v_0,\CF}(\dif x \dif y) \\ 
&\ge& \frac{e^{\mu_0-t}}{\BV^{(t)}} \int_{\{x\le t\}} y\cdot \DH_{v_0,\CF}(\dif x \dif y) \\ 
&=& \frac{e^{\mu_0-t}}{\BV^{(t)}} \int_0^{\lam^{(t)}_\max(v_0;\CF)} \vol(\CF^{(y)}(R/\CF_{v_0}^{(t)})) \dif y \\ 
&\ge& e^{\mu_0-t}\int_0^{\lam^{(t)}_\max(v_0;\CF)} \Big(1-\frac{y}{\lam^{(t)}_\max(v_0;\CF)}\Big)^n \dif y 
\,\,\,=\,\,\, \frac{e^{\mu_0-t}}{n+1} \lam^{(t)}_\max(v_0;\CF). 
\end{eqnarray*}
\end{proof}

For any $m\in l_0\IN$ and $t\in \IR$, let $\Lambda_{m,mt} \seq R_m$ be the lifting of a basis $\oLambda_{m,mt} \seq R_m/\CF_{v_0}^{mt}R_m$ which is compatible with the filtration induced by $v_0$. Then the divisor 
\begin{eqnarray}
D_{m,mt}
= \frac{1}{\BV_{m,mt}}\cdot \frac{1}{m^n/n!}\cdot
 \sum_{s\in \Lambda_{m,mt}} e^{\mu_0-\frac{v_0(s)}{m}}\frac{\{s=0\}}{m} 
\end{eqnarray}
is called a {\it $v_0$-weighted $(m,mt)$-basis type $\IR$-divisor} on $X$. 
Let $\CF$ be a filtration. We say that $D_{m,mt}$ is {\it compatible} with $\CF$ if the basis $\oLambda_{m,mt} \seq R_m/\CF_{v_0}^{mt}R_m$ is compatible with both the filtrations induced by $v_0$ and $\CF$. We see that 
\begin{eqnarray*}
S_{m,mt}(v_0;\CF) = \mathop{\sup}_{D} \ord_\CF(D), 
\end{eqnarray*}
where the supremum runs over all the $v_0$-weighted $(m,mt)$-basis type $\IR$-divisors $D$ on $X$. 
If $D$ is also compatible with $\CF$, then we have:  
\begin{eqnarray*}
S_{m,mt}(v_0;\CF)= \ord_\CF(D). 
\end{eqnarray*}

For any $m\in l_0\IN$ and $t\in \IR$, we define the following version of asymptotic delta invariants: 
\begin{eqnarray*}
\delta_{m,mt}(X,\D,v_0)
&=& \mathop{\inf}_D \lct(X,\D;D) \\
&=& \mathop{\inf}_D 
\mathop{\inf}_{v\in\Val_{X,o}^*} \frac{A_{X,\D}(v)}{v(D)} 
= \mathop{\inf}_{v\in\Val_{X,o}^*} \frac{A_{X,\D}(v)}{S_{m,mt}(v_0;v)}, 
\end{eqnarray*}
where $\inf_D$ runs over all the $v_0$-weighted $(m,mt)$-basis type $\IR$-divisors $D$ on $X$, and $\inf_v$ runs over all $v\in\Val_{X,o}^*$.

\begin{defi}\rm 
The {\it $t$-truncated $v_0$-weighted delta invariant} and {\it $v_0$-weighted delta invariant} of the log Fano fibration germ $(X,\D)\to Z\ni o$ are defined by 
\begin{eqnarray}
\delta^{(t)}(X,\D,v_0)
&=& \mathop{\inf}_{v\in\Val_{X,o}^*} 
\frac{A_{X,\D}(v)}{S^{(t)}(v_0;v)}, \\ 
\delta(X,\D,v_0)
&=& \mathop{\inf}_{v\in\Val_{X,o}^*} 
\frac{A_{X,\D}(v)}{S(v_0;v)}. 
\end{eqnarray}
\end{defi}

For any $v\in\Val_{X,o}^*$, we have 
\begin{eqnarray*}
\mathop{\limsup}_{m\to \infty} \delta_{m,mt}(X,\D,v_0)
&\le& \mathop{\lim}_{m\to \infty} 
\frac{A_{X,\D}(v)}{S_{m,mt}(v_0;v)}
= \frac{A_{X,\D}(v)}{S^{(t)}(v_0;v)}, \\
\mathop{\limsup}_{t\to +\infty} \delta^{(t)}(X,\D,v_0)
&\le& \mathop{\lim}_{t\to +\infty} 
\frac{A_{X,\D}(v)}{S^{(t)}(v_0;v)}
= \frac{A_{X,\D}(v)}{S(v_0;v)}. 
\end{eqnarray*}
Hence 
\begin{eqnarray}
\label{Eqnarray. limsup_m delta_m,mt le delta^(t)}
\mathop{\limsup}_{m\to \infty} \delta_{m,mt}(X,\D,v_0) &\le& \delta^{(t)}(X,\D,v_0), \\
\label{Eqnarray. limsup_t delta^(t) le delta}
\mathop{\limsup}_{t\to +\infty} \delta^{(t)}(X,\D,v_0) &\le& \delta(X,\D,v_0). 
\end{eqnarray}

Following the argument of \cite[Lemma 2.2 and Corollary 2.10]{BJ20} using Okounkov bodies and concave transforms (see \cite[Lemma 4.2]{BLXZ23} or \cite[Lemma A.1]{MW23} for the weighted version), we have the following uniform control for the convergence of $S_m$.  
\begin{lem}
\label{Lemma. S_m <= (1+vep)S}
For any $t\in\IR$ and $\vep>0$, there exists $m_0=m_0(v_0,t)>0$ such that 
\begin{eqnarray*}
S_{m,mt}(v_0;\CF) \le (1+\vep)S^{(t)}(v_0;\CF)
\end{eqnarray*}
for any $m\ge m_0$ and any filtration $\CF$. 
\end{lem}
\begin{proof}
Consider the linear function $G_{v_0}: \BO\to \IR$. Let $\BO_{(t)}=\{G_{v_0}< t\} \seq \BO$, which is empty when $t\le 0$. We set  
\begin{eqnarray*} 
g(\By) &=& e^{\mu_0 - G_{v_0}(\By)}, \\
\LE_m^g(\dif \By) &=& g(\By) \cdot \LE_m(\dif \By),  \\
\LE^g(\dif \By) &=& g(\By) \cdot \LE(\dif \By), 
\end{eqnarray*}
where $\LE$ is the Lebesgue measure on $\BO$ and $$\LE_m = \frac{1}{m^n}\mathop{\sum}_{\By \in (m^{-1}\IZ)^n \cap \BO} \delta_\By,$$ 
which converges to $\LE$ weakly. 
Choose $\{s_j\}\seq R_m$ compatible with both $v_0$ and $\CF$, whose image in $R_m/\CF_{v_0}^{mt}R_m$ is a basis. Then 
\begin{eqnarray*}
S_{m,mt}(v_0;\CF) 
= \frac{1}{\BV_{m,mt}}\cdot \frac{1}{m^n/n!}\cdot
 \sum_{j} e^{\mu_0-\frac{v_0(s_j)}{m}}\frac{\ord_{\CF}(s_j)}{m}. 
\end{eqnarray*}
By the definition of $G_\CF $, we have $\frac{\ord_{\CF}(s_j)}{m} \le G_\CF (\frac{\fv(s_j)}{m})$. 
By Lemma \ref{Lemma. good valuation associated to a quasi-monomial valuation} and the definition of $\fv$, we have $\frac{v_0(s_j)}{m} = G_{v_0}(\frac{\fv(s_j)}{m})$. 
Hence
\begin{eqnarray*}
S_{m,mt}(v_0;\CF) 
&\le& \frac{1}{\BV_{m,mt}}\cdot \frac{1}{m^n/n!}\cdot
 \sum_{j} g(\frac{\fv(s_j)}{m})G_\CF(\frac{\fv(s_j)}{m}) \\
&\le& \frac{n!}{\BV_{m,mt}}\cdot 
    \int_{\BO_{(t)}} G_\CF (\By) \LE_m^g(\dif\By). 
\end{eqnarray*}
Recall that 
\begin{eqnarray*}
S^{(t)}(v_0;\CF) 
= \frac{n!}{\BV^{(t)}}\cdot 
    \int_{\BO_{(t)}} G_\CF (\By) \LE_m^g(\dif\By). 
\end{eqnarray*}

Since $\BV_{m,mt} \to \BV^{(t)}$ as $m\to \infty$, we have $m_1\in\IN$ such that $$\BV^{(t)} \le (1+\frac{\vep}{2})\BV_{m,mt}$$ for any $m\ge m_1$. 
Choose 
$$\vep' = \frac{e^{\mu_0-t}\cdot \BV^{(t)} \cdot \vep}{(n+1)!\cdot (2+\vep)}. $$
By \cite[Lemma A.1]{MW23}, there exists $m_2 =m(t,\vep')$ such that 
\begin{eqnarray*}
\int_{\BO_{(t)}} G(\By) \LE_m^g(\dif\By)
\le 
\int_{\BO_{(t)}} G(\By) \LE^g(\dif\By) +\vep'
\end{eqnarray*}
for any $m\ge m_2$ and every concave function $G$ on $\BO_{(t)}$ satisfying $0\le G\le 1$. We will choose $G=G_\CF /\lam^{(t)}_\max(v_0;\CF)$ for any filtration $\CF$. 

For any $m\ge m_0 := \max\{m_1,m_2\}$ and filtration $\CF$, we have 
\begin{eqnarray*}
S_{m,mt}(v_0;\CF) 
&\le& \frac{n!\lam^{(t)}_\max(v_0;\CF)}{\BV_{m,mt}}\cdot 
\Big(\int_{\BO_{(t)}} G(\By) \LE^g(\dif\By) +\vep'\Big) \\
&\le& \frac{\BV^{(t)}}{\BV_{m,mt}} S^{(t)}(v_0;\CF) + \frac{n! \lam^{(t)}_\max(v_0;\CF)}{\BV_{m,mt}}\vep' \\
&\le& (1+\frac{\vep}{2}) S^{(t)}(v_0;\CF) + \frac{\vep}{2}\cdot \frac{e^{\mu_0-t}}{n+1} \lam^{(t)}_\max(v_0;\CF) \\
&\le& (1+\vep) S^{(t)}(v_0;\CF), 
\end{eqnarray*}
where the last inequality follows from Corollary \ref{Corollary. T <= (n+1)S}. 
\end{proof}

\begin{cor}
\label{Corollary. convergence of delta_m and delta^(t)}
We have the following convergence: 
\begin{eqnarray}
\mathop{\lim}_{m\to \infty} \delta_{m,mt}(X,\D,v_0) 
&=& \delta^{(t)}(X,\D,v_0),~\text{and}\\
\mathop{\lim}_{t\to +\infty} \delta^{(t)}(X,\D,v_0) 
&=& \delta(X,\D,v_0).  
\end{eqnarray}
\end{cor}
\begin{proof}
For any $t\in \IR$ and $\vep>0$, we have
\begin{eqnarray*}
\mathop{\liminf}_{m\to \infty} \delta_{m,mt}(X,\D,v_0) 
&=& \mathop{\liminf}_{m\to \infty} \mathop{\inf}_{v\in \Val_{X,o}^*} \frac{A_{X,\D}(v)}{S_{m,mt}(v_0;v)} \\
&\ge& (1+\vep)^{-1} \mathop{\inf}_{v\in \Val_{X,o}^*} \frac{A_{X,\D}(v)}{S^{(t)}(v_0;v)} 
\end{eqnarray*}
by Lemma \ref{Lemma. S_m <= (1+vep)S}. Letting $\vep\to 0$, we conclude by (\ref{Eqnarray. limsup_m delta_m,mt le delta^(t)}) that 
\begin{eqnarray*}
\mathop{\lim}_{m\to \infty} \delta_{m,mt}(X,\D,v_0) 
= \delta^{(t)}(X,\D,v_0).  
\end{eqnarray*}

On the other hand, recall that $\lim_{t\to +\infty}\BV^{(t)} =\BV$. For any $\vep>0$, we may choose $t_0\in\IR$ such that $\BV\le (1+\vep)\BV^{(t_0)}$. Then 
\begin{eqnarray*}
S^{(t)}(v_0;v) \le (1+\vep)S(v_0;v) 
\end{eqnarray*}
for any $t>t_0$ and any $v\in\Val_{X,o}^*$. By the same argument as above, we conclude that 
\begin{equation*}
\mathop{\lim}_{t\to +\infty} \delta^{(t)}(X,\D,v_0) 
= \delta(X,\D,v_0).  \qedhere
\end{equation*}
\end{proof}

If $(X,\D)\to Z\ni o$ admits a good $\IT=\IG_m^r$-action and $v_0=\wt_{\xi_0}$ for some $\xi_0\in \reeb$. Then 
\begin{eqnarray*}
\BV &=& \BV(\xi_0)
\,\,\,=\,\,\
\int_\BP e^{-\la\alpha,\xi_0 \ra} \cdot \DH_{\BP}(\dif \alpha), \\
S(\xi_0;\CF) &=& 
\frac{1}{\BV} \int_{\BP\times \IR} e^{-\la\alpha,\xi_0 \ra} x \cdot \DH_{\BP,\CF}(\dif \alpha \dif x). 
\end{eqnarray*}
In this case, we define the $\IT$-equivariant $\xi_0$-weighted delta invariant by
\begin{eqnarray*}
\delta_\IT(X,\D,\xi_0)
= \mathop{\inf}_{v\in\Val_{X,o}^{\IT,*}} 
\frac{A_{X,\D}(v)}{S(\xi_0;v)}. 
\end{eqnarray*}

\subsection{Convexity of \texorpdfstring{$\BH$}{}-invariants and uniqueness of \texorpdfstring{$\BH$}{}-minimizer}
We study the global behavior of $\BH$ in the rest of this section. 
Following the similar strategy of \cite[Theorem 3.7]{BLXZ23}, we prove the convexity of the $\BH$-invariants. As a consequence, we prove the uniqueness of a valuative minimizer of $\BH$.

\begin{lem}
For any weakly special $v_0,v_1\in \Val_{X,o}^*$, we have
$\mu(\CF_t) \le (1-t)\mu(\CF_{v_0}) + t\mu(\CF_{v_1}), $
where $\CF_t$ is the geodesic connecting $\CF_{v_0}$ and $\CF_{v_1}$. 
\end{lem}

\begin{proof}
We prove the lemma by the argument of \cite[Proposition 3.12]{BLXZ23} using the relative cone construction. Let $(C=C(X,L),\D_C) \to Z$ be the relative cone over $X\to Z$, where $L=-l_0(K_X+\D)$ and $\D_C$ is defined by (\ref{Eqnarray. D_C}). Then $C=\Spec R(L)$, where $R(L)=\oplus_{j\in\IN} R_j(L), R_j(L) = f_*\CO_X(jL) = R_{l_0j}$. Let $\xi_0$ be the Reeb vector on $C$ reading the $j$-grading, that is, $\wt_{\xi_0}(s)=j$ for any $s\in R_j(L)\setminus\{0\}$. We choose $w_0=v_{0,l_0\xi_0}$ and $w_1=v_{1,l_0\xi_1}$. Let $\fa_{t,\bu} = \fa_\bu((1-t)w_0) \boxplus \fa_\bu(tw_1)$ be the geodesic connecting $\fa_\bu(w_0)$ and $\fa_\bu(w_1)$ on $C$ (\cite[Section 3.3.1]{XZ20}). Then 
$$\fa_{t,m} = \span\{s\in R(L)\mid \lfloor (1-t)w_0(s) \rfloor + \lfloor tw_1(s) \rfloor \ge m \}.$$
For any $k\in \IR$ and $s\in \CF_t^{k+2}R_j(L)$, we have 
\begin{eqnarray*}
(1-t)w_0(s) + tw_1(s) = 
(1-t)v_0(s) + tv_1(s) + l_0j \ge 
k+2 + l_0j. 
\end{eqnarray*}
Hence $\lfloor (1-t)w_0(s) \rfloor + \lfloor tw_1(s) \rfloor \ge k+l_0j$. We see that $\CF_t^{k+2}R_j(L)\seq \fa_{t,l_0j+k}$. Let $k=c l_0j$, then $\CF_t^{cl_0j+2}R_j(L)\seq \fa_{t,(1+c)l_0j}$. Comparing the singularities of $(X,\D; I_\bu^{(c)})$ and $(C,\D_C; \fa_{t,(1+c)\bu})$ via the cone structure, we see that $\lct(X,\D;I_\bu^{(c)}) \le \frac{1}{1+c}\lct(C,\D_C; \fa_{t,\bu})$ for any $c>0$. Hence $1+\mu(\CF_t) \le \lct(C,\D_C;\fa_{t,\bu})$. 
On the other hand, by \cite[Theorem 3.11]{XZ20}, we have   
\begin{eqnarray*}
\lct(C,\D_C; \fa_{t,\bu})
&\le& (1-t) \lct(C,\D_C; \fa_{\bu}(w_0)) + t \lct(C,\D_C; \fa_{\bu}(w_1)) \\
& \le & (1-t) A_{C,\D_C}(w_0)+t A_{C,\D_C}(w_1) \\
& = & (1-t) A_{C,\D_C}(v_0)+t A_{C,\D_C}(v_1) + 1 \\ 
& = & (1-t) A_{X,\D}(v_0)+t A_{X,\D}(v_1) + 1 \\
& = & (1-t) \mu(\CF_{v_0})+t \mu(\CF_{v_1}) + 1, 
\end{eqnarray*}
where the second equality follows from our choice of $w_i=v_{i,l_0\xi_0}$, the last equality follows from the fact that $v_i$ are weakly special. 
We conclude that
\begin{equation*}
    \mu(\CF_t) \le \lct(C,\D_C;\fa_{t,\bu}) - 1 \le (1-t) \mu(\CF_{v_0})+t \mu(\CF_{v_1}). \qedhere
\end{equation*}
\end{proof}

\begin{thm}
\label{Theorem:  Convexity}
The functional $\BH$ is convex along geodesics. More precisely, for any $0\le t\le 1$, we have
$\BH(\CF_t)\le (1-t)\BH(\CF_0) + t\BH(\CF_1). $
\end{thm}

\begin{proof}
We have 
\begin{eqnarray*}
\BH(\CF_t) 
&=& \log\Big(
\int_\IR 
e^{\mu(\CF_t) - s}
\DH_{\CF_t}(\dif s)
\Big) \\
&=& \log\Big(
\int_{\IR^2} 
e^{\mu(\CF_t) - (1-t)x-ty}
\DH_{\CF_0,\CF_1}(\dif x \dif y)
\Big) \\
&\le& \log\Big(
\int_{\IR^2} 
e^{(1-t)(\mu(\CF_0) - x) + t(\mu(\CF_1) - y)}
\DH_{\CF_0,\CF_1}(\dif x \dif y)
\Big) \\
&=& \log\Big(
\int_{\IR^2} 
(e^{\mu(\CF_0) - x})^{1-t} \cdot (e^{\mu(\CF_1) - y})^t 
\cdot \DH_{\CF_0,\CF_1}(\dif x \dif y)
\Big) \\
&\le& (1-t)\,\log\Big(
\int_\IR 
e^{\mu(\CF_0) - x}
\DH_{\CF_0}(\dif x)
\Big)
+ t\,\log\Big(
\int_\IR 
e^{\mu(\CF_1) - y}
\DH_{\CF_1}(\dif y)
\Big) \\
&=& (1-t)\BH(\CF_0) + t\BH(\CF_1), 
\end{eqnarray*}
where the first inequality follows from (\ref{Equality: Geodesic}), and the second one follows from H\"older's inequality. 
\end{proof}

\begin{cor}
\label{Corollary: uniqueness of minimizer}
Let $v\in \Val_{X,o}^*$ be a quasi-monomial valuation and $w\in \Val_{X,o}^*$. If $\BH(\CF_v)=\BH(\CF_w)=\BH(X,\D)$, then $v=w$. 
\end{cor}

\begin{proof}
The proof is slightly different from \cite[Proposition 3.14]{BLXZ23}. Let $\CF_0\coloneqq\CF_v$ and $\CF_1\coloneqq\CF_w$, and let $\CF_t$ be the geodesic connecting them. Then  
\begin{eqnarray*}  
\BH(\CF_t)\le (1-t)\BH(\CF_0) + t\BH(\CF_1) = \BH(X,\D). 
\end{eqnarray*}
Hence the equality holds, and so do those in the proof of Theorem \ref{Theorem:  Convexity}. By H\"older's inequality, we know that $e^{\mu(\CF_0)-x}=c\cdot e^{\mu(\CF_1)-y}$ almost everywhere on $\IR^2$ with respect to the measure $\DH_{\CF_0,\CF_1}$ for some $c>0$. On the other hand, since $\BH(\CF_0)=\BH(\CF_1)$, we have $c=1$. Hence $\mu(\CF_0)-x=\mu(\CF_1)-y$ almost everywhere on $\IR^2$ with respect to the measure $\DH_{\CF_0,\CF_1}$, that is, 
\begin{eqnarray*}  
0=\int_{\IR^2}|x-y-d|\DH_{\CF_0,\CF_1}(\dif x \dif y) 
=\int_{\IR^2}|x-y|\DH_{\CF_0,\CF_1(d)}(\dif x \dif y) , 
\end{eqnarray*}
where $d = \mu(\CF_0)-\mu(\CF_1)$. Then $\DH_{\CF_0} = \DH_{\CF_1(d)}$, so they have the same $\lam_\min$. Hence $d=0$ by Corollary \ref{Corollary. lam_min(v)=0}. We conclude that $v=w$ by Lemma \ref{Lemma. v,w d_1-equiv. implies v=w}. 
\end{proof}

\section{Existence of \texorpdfstring{$\BH$}{}-minimizers}
\label{Section. Existence of H-minimizers}

We have shown that the $\BH$-invariant of a log Fano fibration germ $(X,\D)\to Z\ni o$ admits at most one valuative minimizer assuming the existence of a quasi-monomial valuation. In this section, we aim to show the existence of the $\BH$-minimizer.

\subsection{Weakly special divisorial approximations} 
We show that there exists a sequence of weakly special divisorial valuations $\{v_i\}$ such that 
$$\mathop{\lim}_{i\to \infty} \BH(v_i) = \BH(X,\D)$$ 
in this subsection.

\begin{lem}
\label{Lemma. canonical shift}
For any filtration $\CF$ on $R$, there exists a quasi-monomial $v\in \Val_{X,o}$ such that 
\begin{eqnarray*}
\CF':=\CF(A_{X,\D}(v)-\mu(\CF)) \seq \CF_v. 
\end{eqnarray*}
In particular, $A_{X,\D}(v)=\mu(\CF')\le \mu(\CF_v)\le A_{X,\D}(v)$. Hence $v$ is weakly special. 
\end{lem}

\begin{proof}
Just assume that $\mu = \mu(\CF) < \lam_\max(\CF)$. Then we have $\lct(X,\D;I^{(\mu)}_\bu) \le 1$. There exists a quasi-monomial valuation $v$ on $X$ computing $\lct(X,\D;I^{(\mu)}_\bu)$ by \cite{Xu19}. Hence $v(I^{(\mu)}_\bu)\ge A_{X,\D}(v)$. 
Denote by $f_v(t)=v(I^{(t)}_\bu)$, which is a convex function on $\IR$. Rescale $v$ such that the first order left-derivative at $\mu\in\IR$ equals to one, that is, $f'_{v,-}(\mu)=1$. Then we have 
\begin{eqnarray}
\label{Inequality: convexity of f_v(t)}
f_v(t)
\ge t + f_v(\mu) - \mu 
\ge t + A_{X,\D}(v)-\mu. 
\end{eqnarray}
We claim that $\CF':=\CF(A_{X,\D}(v)-\mu) \seq \CF_v$. Indeed, for any $\lam \in \IR$ and $s \in \CF^{m(\lam - A_{X,\D}(v)+\mu)}R_m$, 
\begin{eqnarray*}
\frac{1}{m}v(s) 
\ge \frac{1}{m}
v(I_{m,m(\lam - A_{X,\D}(v)+\mu)})
\ge f_v(\lam - A_{X,\D}(v)+\mu)
\ge \lam, 
\end{eqnarray*}
where the third inequality follows from (\ref{Inequality: convexity of f_v(t)}) with $t=
\lam - A_{X,\D}(v)+\mu$. Hence $s\in \CF^{m\lam}_v R_m$. 
\end{proof}

\begin{thm}
\label{Theorem: weakly special approximations}
We have
$\BH(X,\D) = \inf_{v} \, \BH(v), $
where the infimum runs over all the weakly special valuations $v\in \Val_{X,o}^*$. 
\end{thm}

\begin{proof}
We need to show that for any linearly bounded filtration $\CF$ on $R$, there exists a weakly special valuation $v$ over $X$ such that $\BH(\CF) \ge \BH(v)$. 
By Lemma \ref{Lemma. canonical shift}, there exists a weakly special valuation $v$ such that $\CF':=\CF(A_{X,\D}(v)-\mu) \seq \CF_v$ and $\mu(\CF') = A_{X,\D}(v)$, so $\vol(\CF';t) \ge \vol(v;t)$ for any $t\in\IR$. 
Recall that the functional $\mu(\CF)$ and measure $\DH_{\CF}$ are affine with respect to shifts of filtrations, that is, $\mu(\CF(b))=\mu(\CF)+b$ and $\int_\IR f(s) \DH_{\CF(b)}(\dif s) = \int_\IR f(s+b) \DH_{\CF}(\dif s)$ for any $b\in \IR$. Hence $\BH(\CF) = \BH(\CF(b))$. 
We conclude that 
\begin{eqnarray*}
\BH(\CF) 
\,\,\,=\,\,\, \BH(\CF') 
&=& A_{X,\D}(v) + \log\Big(
\int_\IR e^{-t} \vol(\CF';t)\dif t
\Big) \\
&\ge& A_{X,\D}(v) + \log\Big(
\int_\IR e^{-t} \vol(v;t)\dif t
\Big) 
\,=\,\BH(v). 
\end{eqnarray*}
The proof is finished. 
\end{proof}

\begin{cor}
\label{Corollary. H-minimizer is weakly special}
If there exists a filtration $\CF_0$ on $R$ minimizing the $\BH$-invariant, then any minimizer $v_0$ of $\lct(X,\D; I_\bu^{(\mu)}(\CF_0))$ minimizes the $\BH$-invariant, where $\mu=\mu(\CF_0)$. In particular, there exists a weakly special valuation $v_0$ minimizing the $\BH$-invariant. 
\end{cor}

Let $(Y,E=E_1+\cdots+E_r)\to (X,\D)$ be a log smooth model and $\eta$ be the generic point of some strata of $\cap_i E_i$. Let $\IR^r_{\ge 0} \cong \sigma = \QM_\eta(Y,E)$ be the corresponding quasi-monomial simplicial cone. Let $l: \sigma \to \IR_{\ge 0}$ be the linear function $l(x_1,\cdots,x_r) = x_1+\cdots+x_r$. For any $b>a>0$, denote by $\sigma_{[a,b]} = l^{-1}([a,b])$. We have the following modified version of \cite[Theorem A]{BFJ14}. 

\begin{lem}
\label{Lemma. v to v(f) is Lipschitz continuous}
There exists constant $A = A(\sigma,a,b)$ such that for any $g\in \CO_{Y,\eta}$, the function $\sigma_{[a,b]} \to \IR, v\mapsto v(g)$ is Lipschitz continuous (with respect to the Euclidean metric on $\sigma \cong \IR_{\ge 0}^r$) with Lipschitz constant at most $A\cdot \ord_\eta(g)$. 
\end{lem}

\begin{thm}
\label{Theorem. v to vol(v;1) is continuous}
The function $\sigma_{[a,b]} \to \IR, v\mapsto \vol(v;1)$ is continuous. 
\end{thm}

\begin{proof}
We follow the argument of \cite[Theorem D]{BFJ14}. 
Since $a>0$, we have 
\[
a_0:= \inf_{v\in \sigma_{[a,b]}} v(\fm_\eta) >0.
\]
Let $v_0 = v_{(1,\cdots,1)} \in \sigma$. Then by Izumi's inequality, there exists $a_0'>0$ such that $\ord_\eta \ge a_0'v_0$. Hence we have $v\ge a_0 \ord_\eta \ge a_1 v_0$ for any $v\in \sigma_{[a,b]}$, where $a_1=a_0a_0'$. We see that $\CF_{v_0}^{m/a_1} R_m \seq \CF_v^{m}R_m$. Hence 
\begin{eqnarray*}
\vol(v;1) \le \frac{1}{a_1^n} \vol(v_0;\frac{1}{a_1}) =: C
\end{eqnarray*}
for any $v\in \sigma_{[a,b]}$. 

Recall that 
$I_{m,m}=I_{m,m}(v)=
{\rm im}
(\CF_v^{m}H^0(X,mL)\otimes \CO_X(-mL)\to \CO_X)$ is the base ideal sequence of $v$. By Lemma \ref{Lemma. v to v(f) is Lipschitz continuous}, the function $w\mapsto w(I_{m,m})$ is Lipschitz continuous with Lipschitz constant at most $A\ord_{\eta}(I_{m,m}) \le A a_0^{-1} v(I_{m,m})$. Denote by $|v-w|$ the Euclidean distance of $v,w\in \sigma \cong \IR^r_{\ge 0}$ and assume that $1-Aa_0^{-1} |v-w| > 0$. Then 
$$|w(I_{m,m})- v(I_{m,m})| \le Aa_0^{-1} v(I_{m,m}) |v-w|.$$
Hence 
\begin{eqnarray*}
w(I_{m,m})\ge v(I_{m,m}) - Aa_0^{-1} v(I_{m,m}) |v-w| \ge (1-Aa_0^{-1} |v-w|) m. 
\end{eqnarray*}
In other words, $w(\CF_v^m R_m) \ge (1-\vep)m$, where $\vep = Aa_0^{-1} |v-w|$. Hence $\CF_v^m R_m \seq \CF_w^{(1-\vep)m}R_m$ and 
\begin{eqnarray*}
\vol(v;1)\ge (1-\vep)^n \vol(w;1-\vep).  
\end{eqnarray*}
By Theorem \ref{Theorem. monotonicity of vol/t^k}, we have $\vol(w;1-\vep) \ge (1-\vep)^n \vol(w;1)$. Hence 
\begin{eqnarray*}
\vol(w;1)-\vol(v;1)
&\le& \vol(w;1)- (1-\vep)^n \vol(w;1-\vep) \\
&\le& (1-(1-\vep)^{2n}) \vol(w;1) \\
&\le& 2n CAa_0^{-1} |v-w|, 
\end{eqnarray*}
where the last inequality follows from the inequality $1-(1-\vep)^{2n} \le 2n\vep$. Exchanging $v$ and $w$, we conclude that 
\begin{eqnarray*}
|\vol(v;1)-\vol(w;1)| \le 2n CAa_0^{-1} |v-w|
\end{eqnarray*}
for any $v,w \in \sigma_{[a,b]}$ satisfying $|v-w|<A^{-1}a_0$. 
\end{proof}

\begin{thm}
\label{Theorem. H is continuous on quasi-monomial cone}
Let $\Gamma$ be a $\IQ$-complement of $(X,\D)\to Z\ni o$ and $\sigma\seq \LC(X,\D+\Gamma)$ be
a quasi-monomial simplicial cone. Then the restriction of $\BH$ on $\sigma$ is continuous. Precisely, the function 
$$\BH: \sigma\to \IR_{>0},\quad v\mapsto A_{X,\D}(v)+ \log\Big(\int_0^{+\infty} e^{ - t} \vol(v;t)\,\dif t\Big)$$ 
is continuous. 
\end{thm}

\begin{proof}
By \cite[Proposition 5.1]{JM12}, the function $\sigma \to \IR, v\mapsto A_{X,\D}(v)$ is linear. The continuity follows from Theorem \ref{Theorem. v to vol(v;1) is continuous} as $\vol(v;t) = \vol(v/t;1)$ by definition. 
\end{proof}

\begin{thm}
\label{Theorem. weakly special divisor approximation}
We have
\[
\BH(X,\D) = \inf_{v} \, \BH(v), 
\]
where the infimum runs over all the weakly special divisorial valuations $v\in \Val_{X,o}^*$. 
\end{thm}
\begin{proof}
By Theorem \ref{Theorem: weakly special approximations}, there exists a sequence of weakly special valuations $\{v_i\}$ such that $\BH(v_i) \to \BH(X,\D)$ as $ i\to \infty$. Then there exists a $\IQ$-complement $\Gamma_i$ of $(X,\D)$ such that $v_i\in \LC(X,\D+\Gamma_i)$. By taking a log resolution of $(X,\D+\Gamma_i)$, we see that $v_i\in \sigma_i$ for some quasi-monomial simplicial cone $\sigma_i\seq \LC(X,\D+\Gamma_i)$. Then $\BH|_{\sigma_i} = \log \circ\BV|_{\sigma_i}$ is continuous by Remark \ref{Remark: H = log W} and Theorem \ref{Theorem. H is continuous on quasi-monomial cone}. We can choose $w_i \in \sigma_{i,\IQ}$ such that $\BH(w_i)\le \BH(v_i) + i^{-1}$. Hence $w_i$ is a weakly special divisorial valuation and $\BH(w_i)\to \BH(X,\D)$.  
\end{proof}

\subsection{Properness of valuation spaces}
In this subsection, we prove a properness estimate in terms of $\BH$. First, we aim to control the log discrepancy of a critical point for $\BV$. 
For any valuation $v\in\Val_{X,o}^*$, we know that the weighted volume function $f_v(x)\coloneqq \BV(x\cdot v)$ is strictly convex along the ray $\{x\cdot v\mid x>0\}$, since
\begin{eqnarray*}
f_v(x) &=& \int_0^\infty e^{x\cdot(A_{X,\D}(v) - t)} \DH_{v}(\dif t),\\
f_v'(x) &=& \int_0^\infty (A_{X,\D}(v) - t)\cdot e^{x\cdot(A_{X,\D}(v) - t)} \DH_{v}(\dif t),~ \text{and}\\
f_v''(x) &=& \int_0^\infty (A_{X,\D}(v) - t)^2\cdot e^{x\cdot(A_{X,\D}(v) - t)} \DH_{v}(\dif t) > 0.
\end{eqnarray*}
In the following, we fix the rescaling of $v$ such that $f_v'(1)=0$. Then 
\begin{eqnarray*}
A_{X,\D}(v) = \frac{
\int_0^\infty e^{ - t}t \cdot \DH_{v}(\dif t)
}{
\int_0^\infty e^{ - t} \cdot \DH_{v}(\dif t)
}. 
\end{eqnarray*}
Recall that $\DH_v(\dif t) = \dif\, \vol(v;t)$ is of polynomial growth by Theorem \ref{Theorem. Weak convergence of DH}. Using integration by part, we have $\int_0^\infty e^{ - t} \cdot \DH_{v}(\dif t) = \int_0^\infty e^{ - t} \cdot \vol(v;t)\dif t$. 
Hence 
\begin{eqnarray}\label{eqn:A-change of variable}
A_{X,\D}(v) = \frac{
\int_0^\infty e^{ - t} \cdot t\, \dif \vol(v;t)
}{
\int_0^\infty e^{ - t} \cdot \vol(v;t)\dif t
}. 
\end{eqnarray}
If $X=Z$, then $\vol(v;t)= \vol(v)\cdot t^{n}$ and $A_{X,\D}(v)=n=\dim X$. In general, we have the following upper bound for the log discrepancy.

\begin{prop}
\label{Theorem. bounded log discrepancies}
{}If $v=a\cdot \ord_E$ for some weakly special divisor $E$ satisfying $f_v'(1)=0$, then $$A_{X,\D}(v)\le n. $$
\end{prop}
\begin{proof}
Since
\begin{eqnarray*}
\dif \frac{\vol(v;t)}{t^n} = \frac{1}{t^{n+1}} \big( t\,\dif \vol(v;t) - n\cdot\vol(v;t)\dif t\big), 
\end{eqnarray*}
we have
\begin{eqnarray*}
(A_{X,\D}(v) - n) \cdot \int_0^\infty e^{ - t} \cdot \vol(v;t)\dif t 
&=& \int_0^\infty e^{ - t} \cdot \big(t\, \dif \vol(v;t) - n\cdot\vol(v;t)\dif t\big) \\
&=& \int_0^\infty e^{ - t} \cdot t^{n+1} \dif \frac{\vol(v;t)}{t^n}\le 0, 
\end{eqnarray*}
where the first equality follows from \eqref{eqn:A-change of variable} and the inequality follows from Theorem \ref{Theorem. monotonicity of vol/t^k}. 
\end{proof}

We have the following uniform version of the Izumi-type inequality generalizing \cite[Theorem 3.1]{Li18}, which follows from the family version of the Izumi-type inequality \cite[Theorem 20]{BL18a}. 

\begin{lem}
\label{Lemma. Global Izumi}
Let $(X,\D)$ be a klt pair. Then there exists a constant $c_1>0$ depending only on $(X,\D)$ such that for any $v\in\Val_{X,o}^*$,  closed point $x\in C_X(v)$ and $g\in \CO_{X,x}$, we have
\begin{eqnarray}
v(g) \le c_1 \cdot A_{X,\D}(v)\cdot \ord_x(g).  
\end{eqnarray}
\end{lem}
\begin{proof}
By taking a finite affine open cover, it suffices to prove the lemma when $X$ is affine. Let $(\CX,\D_\CX) \coloneqq (X,\D) \times \D$ and $B = X$. Then $\pr_1: (\CX,\D_\CX) \to B$ with the diagonal section $\sigma: B\to \CX, x\mapsto (x,x)$ is a $\IQ$-Gorenstein family of klt singularities. By \cite[Theorem 20]{BL18a}, we get the desired constant $c_1>0$. 
\end{proof}

\begin{cor}
\label{Cor. Izumi}
Let $(X,\D)\to Z\ni o$ be a klt fibration germ. 
There exists a constant $c_1>0$ depending only on $(X,\D)\to Z\ni o$ such that, for any $v\in\Val_{X,o}^*$, closed point $x\in C_X(v)$ and $g\in \CO_{X,x}$, we have 
\begin{eqnarray*}
v(\fm_o) \cdot \ord_o(g) \le v(g) \le c_1 A_{X,\D}(v)\cdot \ord_x(g). 
\end{eqnarray*}
\end{cor}

The following version of proper estimate is one of the main ingredients in the proof of the existence of $\BH$-minimizer. 
It is a generalization of \cite[Theorem 4.1]{Li18}. 

\begin{thm}[Properness]
\label{Theorem. Properness}
Let $(X,\D)\to Z\ni o$ be a log Fano fibration germ. 
For any constants $A,C\in \IR_{>0}$, there exists $\vep=\vep(A,C) >0$ such that for any $v\in\Val_{X,o}^*$ satisfying $A_{X,\D}(v)\le A$ and $\BH(v)\le C$, we have $v(\fm_o) \ge \vep$. 
\end{thm}
\begin{proof}
We may assume that $q\coloneqq v(\fm_o)^{-1} \ge 2$ and that $n=\dim X\ge 2$. Fix $l_0\in\bZ_{>0}$ such that $-l_0(K_X+\Delta)$ is very ample over $Z$. 
Using the estimate $e^{-x}\ge 1-x \ge \frac{1}{2}$ for $0\le x\le \frac{1}{2}$, we have 
\begin{eqnarray}
\label{Eqnarray. Properness. first estimate by volume}
e^C 
\ge e^{\mu_{X,\D}(v)} \cdot \int_{0}^\infty e^{-x} \cdot \DH_v(dx) 
\ge 1\cdot \int_{0}^{\frac{1}{2}} \frac{1}{2} \cdot \DH_v(dx)
= \frac{1}{2} \vol(v;\frac{1}{2}). 
\end{eqnarray}
Next, we give a lower bound of $\vol(v; \frac{1}{2})$ independent of $v\in\Val^*_{X,o}$, following the argument of \cite[Theorem 4.1]{Li18}. 
Let $v'\coloneqq q\cdot v$. Then we can re-write the Izumi-type inequality of Corollary \ref{Cor. Izumi} as
\begin{eqnarray}
\label{Eqnarray. Properness. Izumi}
\ord_o(u) \le v'(u) \le s \cdot \ord_x(u) 
\end{eqnarray} 
for any closed point $x\in C_X(v)$ and any $u\in R_m$, where we may assume that $s\coloneqq \lceil qc_1A\rceil \ge 2$. 
For any $p\in\bZ_{\ge 0}$, let $d_p \coloneqq \ell(\Gr_x^pR_m)$, and choose
\begin{eqnarray*}
u^{(p)}_1, \cdots, u^{(p)}_{d_p} \in \CF^{p}_x R_m
\end{eqnarray*}
such that their images in $\Gr_x^p R_m=\CF^{p}_xR_m/\CF^{p+1}_xR_m$  (as $\cF_x$ is a $\bZ$-filtration) form a $\Ik$-basis. 
Since $v'(\fm_o)=1$, there exists $g\in \fm_o \setminus \fm_o^2$ such that $v'(g)=1$. 
Fix $r\in\bZ_{>0}$. For any $j\in\bZ_{\ge 0}$ with $0\le j\le s-1$ and $0\le p\le \lfloor r/s\rfloor-1$, let $h^{(p,j)}_i\coloneqq g^{r-sp-j}\cdot u_i^{(p)}$ for $1\le i\le d_p$. By (\ref{Eqnarray. Properness. Izumi}), we know that
\[
    v'(h_i^{(p,j)}) = r-sp-j + v'(u_i^{(p)}) \le r-sp-j+ sp \le r.
\]

We claim that the subset
\[
    \{[h_i^{(p,j)}]\mid 0\le p\le \lfloor r/s\rfloor-1,~ 0\le j\le s-2,~ 1\le i\le d_p\} \seq R_m/\CF_{v'}^{r+1}R_m
\]
is linearly independent. 
Note that we remove a term $g^{r-sp-s+1}u_i^{(p)}$ for each $p$ and $i$. 

Grant the claim for now. For any $t>0$ and $m\gg 0$, let $r\coloneqq \lfloor tm \rfloor -1$, and we can estimate 
\begin{equation*}
\begin{aligned}
    \ell(R_m/\CF_{v'}^{tm} R_m) &\ge \ell(R_m/\CF_{v'}^{r+1} R_m)\\
        &\ge (s-1)\cdot (d_0+d_1+\cdots +d_{\lfloor r/s \rfloor-1})\\
        &\ge s\cdot \ell(R_m / \CF_x^{\lfloor r/s \rfloor} R_m)-r\\
        &\ge s\cdot \ell(R_m / \CF_x^{{\lfloor tm/s \rfloor-2}} R_m)-tm,
\end{aligned}
\end{equation*}
where the last equality follows from an elementary calculation using the assumption $s\ge 2$.
Dividing by $m^n/n!$ and letting $m\to \infty$, we know that 
$\vol(v';t) \ge s \cdot \vol(\CF_x; t/s)$. 
Let $t = q/2$. 
Since $c_1A \le \lceil qc_1A \rceil /q <c_1A+1$ and $q\ge 2$, we know that  
\begin{equation*}
\begin{aligned}
    2e^{C} 
&\ge \vol(v;\frac{1}{2}) = \vol(v';\frac{q}{2}) \\
&\ge \lceil qc_1A \rceil \cdot \vol(\CF_x; \frac{q}{2 \lceil qc_1A \rceil}) \\
&\ge qc_1A \cdot \vol(\CF_x; \frac{1}{2 (c_1A+1)})\\
&=qc_1A\cdot\frac{1}{2^n(c_1A+1)^n}\cdot\mult(\fm_x) 
\end{aligned}
\end{equation*}
for $0<\frac{1}{2(c_1 A+1)}<l_0^{-1}$, where the first inequality follows from \eqref{Eqnarray. Properness. first estimate by volume}, the third follows from the monotonicity of $\vol(\cF_x;t)$, and the last follows from Lemma \ref{Lemma. volume of a closed point}.

By the upper semi-continuity of the function $x\mapsto \mult(\fm_x)$ (see, for example, \cite[Theorem 3.4]{Smi-mult-usc}), we know that $\mult(\fm_x)\ge 1$, and hence
\begin{equation*}
    \begin{aligned}
        v(\fm_o) = q^{-1} \ge\vep(A,C)\coloneqq
        \frac{c_1A}{2e^C} \cdot \min\{\frac{1}{2^n(c_1A+1)^n}, l_0^{-n}\}>0.  
    \end{aligned}
\end{equation*}

It remains to prove the claim. We first fix $0\le p\le \lfloor r/s\rfloor-1$ and $0\le j\le s-2$ and show that the subset $\{[h_i^{(p,j)}]\mid 1\le i\le d_p\} \seq R_m/\CF_{v'}^{r+1}R_m$ is linearly independent. 
For any non-zero vector $\bm a=(a_1,\cdots,a_{d_p}) \in (\Ik^{d_p})^*$, let $u^{(p)} =u^{(p)}(\bm a)\coloneqq \sum_{1\le i\le d_p} a_i u^{(p)}_i\in\cF_x^p R_m$. Then $[u^{(p)}]\ne 0 \in \CF^{p}_xR_m/\CF^{p+1}_xR_m$ by the choice of the $u^{(p)}_i$.
So $\ord_x(u^{(p)}) = p$. Let 
\begin{eqnarray*}
h^{(p,j)}=\sum_{1\le i\le d_p} a_i h^{(p,j)}_i \coloneqq  \sum_{1\le i\le d_p} a_i g^{r-sp-j}u^{(p)}_i = g^{r-sp-j} u^{(p)}. 
\end{eqnarray*}
Then $v'(h^{(p,j)}) = r-sp-j + v'(u^{(p)})\le r$ by (\ref{Eqnarray. Properness. Izumi}), and hence $[h^{(p,j)}]\ne 0 \in R_m/\CF_{v'}^{r+1}R_m$.

Now, let $\{h^{(p,j)}\mid 0\le p\le \lfloor r/s\rfloor-1,~ 0\le j\le s-2\} \subseteq R_m\setminus \CF_{v'}^{r+1}R_m$ be as above. For any non-zero vector $\bm b=(b_{p,j}) \in (\Ik^{s\cdot\lfloor r/s\rfloor})^*$, let $h\coloneqq\sum_{p,j} b_{p,j} h^{(p,j)}$ and let $L\coloneqq \max\{sp+j\mid b_{p,j}\ne 0\}$. 
Then we may write
\[
    h = g^{r-L} \sum_{0\le p\le p_L,~0\le j\le s-2} b_{p,j} g^{L-sp-j} u^{(p)} = g^{r-L} u,
\]
where $p_L\coloneqq\lfloor L/s\rfloor$ and $u\coloneqq \sum_{0\le p\le p_L,~0\le j\le s-2} b_{p,j}g^{L-sp-j} u^{(p)}$. 
For any $0\le p<p_L$ and any $0\le j\le s-2$, we can compute
\begin{equation*}
\begin{aligned}
    \ord_x (g^{L-sp-j}u^{(p)}) &= L-sp-j+p
        \ge L-sp-s+2+p\\
        &=L-(s-1)(p+1)+1
        \ge sp_L-(s-1)p_L+1\\
        &>p_L.
\end{aligned}
\end{equation*}
By the last paragraph, we know that
\begin{eqnarray*}
\ord_x(u)= \ord_x (\sum_{0\le j\le L-s p_L} b_{p_L,j} g^{L-sp_L-j} u^{(p_L)}) 
= \ord_x (u^{(p_L)}) = p_L.  
\end{eqnarray*}
Hence $v'(h) = r-L +v'(u) \le r-L + s\cdot p_L \le r$ by (\ref{Eqnarray. Properness. Izumi}), that is, $h\ne 0\in\cF_x^pR_m/\cF_x^{p+1}R_m$, as desired.
\end{proof}

\subsection{Invariance of volume functions}

In this subsection, we prove the deformation invariance of the volume functions in Definition \ref{Definition. volume functions}, which can be viewed as a relative version of \cite[Theorem 4.2]{HMX13} (global version) and \cite[Theorem 2.18]{Xu19} (local version). 

We first recall some basic terminology for families of varieties. 
Let $B$ be a scheme of finite type. We say that $(X,\D) \to Z$ is a {\it $\IQ$-Gorenstein family of log Fano fibrations over $B$} (\cite[Definition 2.25]{HQZ25}) if 
\begin{enumerate}
\item both $X$ and $Z$ are flat over $B$, and $X\to Z$ is a fibration, 
\item for any closed point $b\in B$, $X_b$ and $Z_b$ are connected and normal, 
\item $\D$ is an effective $\IQ$-divisor whose support does not contain any fiber $X_b$ of $X\to B$,
\item $X$ and $B$ are normal, $-(K_{X/B}+\D)$ is $\IQ$-Cartier and ample over $Z$, and 
\item $(X_b,\D_b)$ is klt for any closed point $b\in B$. 
\end{enumerate}
Hence $(X_b, \D_b) \to Z_b$ is a log Fano fibration for any closed point $b\in B$. Moreover, if $Z\to B$ is affine and admits a section $B\seq Z$, we say that $(X,\D)\to Z\supseteq B$ is a {\it $\IQ$-Gorenstein family of log Fano fibration germs over $B$}. 

A $\IQ$-divisor $\D^+ \ge \D$ is called a {\it relative $\IQ$-complement} of $f:(X,\D)\to Z\supseteq B$ if
\begin{enumerate}
\item the support of $\D^+$ does not contain any fiber $X_b$,
\item $K_{X/B} +\D^+ \sim_{\IQ, Z} 0$, 
\item $(X_b,\D_b^+)$ is log canonical for any closed point $b\in B$ (then $(X,\D^+)$ is log canonical by inversion of adjunction), and 
\item there exists a log canonical center $W\seq X$ of $(X,\D^+)$ dominating $B\seq Z$. 
\end{enumerate}
In particular, $(X,\D^+) \to Z$ is a $\IQ$-Gorenstein family of log Calabi-Yau fibrations over $B$. 

Assume moreover that $B$ is a smooth variety. A birational morphism $\pi: Y\to X$ is called a {\it fiberwise log resolution} of $(X,\D^+)$ if 
\begin{enumerate}
\item $Y$ is smooth over $B$, 
\item $E = \sum_{i\in I} E_i = \Ex(\pi) + \pi_*^{-1}\D^+$ is a simple normal crossing divisor, and 
\item each stratum of $E$ (an irreducible component of $\cap_{i\in Z}E_i$ for some $J\seq I$) is smooth with irreducible fibers over $B$. 
\end{enumerate}

\begin{thm}\label{Theorem. invariance of volume functions}
Let $(X,\D)\to Z\supseteq B$ be a $\IQ$-Gorenstein family of log Fano fibration germs over a smooth base $B$, and $\D^+\ge \D$ be a relative $\IQ$-complement. Assume that there is a fiberwise log resolution $\pi: Y\to (X,\D^+)$ over $B$. 
Let $F$ be a prime toroidal divisor with respect to $(Y, \pi_*^{-1} \D^+ + \Ex(\pi))$ with $A_{X,\D^+}(F) <1$, whose center on $X$ dominates $B\seq Z$. Fix $t\in \IQ_{>0}$. Then $\vol_{X_b}(\ord_{F_b}; t)$ is locally constant for $b\in B$. 
\end{thm}
\begin{proof}
We first follow the proof of \cite[Proposition 4.1]{BLX19}. By shrinking and cutting by hyperplane sections, we can assume that $B$ is a smooth affine curve. 
By taking a toroidal resolution, we can assume $F$ is a divisor on $Y$. 
Let $F_1, F_2 \ge 0$ be $\IQ$-divisors on $Y$ without common components in their support such that 
$$K_Y+F_1 \sim_\IQ \pi^*(K_X +\D^+) +F_2 \sim_\IQ F_2$$
and $\pi_*F_1 = \D^+$, where the second $\sim_\IQ$ follows since $(X,\D^+)$ is log Calabi-Yau and $Z$ is affine. Since $(X,\D^+)$ is log canonical, $F_1$ has coefficients in $[0,1]$. 
After a further toroidal blowup, if necessary, we may assume that $(Y, F_1)$ is terminal (see \cite[Corollary 1.4.3]{BCHM10}). 
Note that 
$$d := \ord_F(F_1) = 1- A_{X,\D^+}(F) >0. $$

Since $-(K_X+\D)$ is ample ($Z,B$ are affine), by Bertini's theorem, we may find an effective $\IQ$-divisor $H \sim_\IQ -\frac{d}{t}(K_X+\D) $ such that 
$$\Phi := F_1-dF +\pi^*H$$ 
has coefficients in $[0,1]$ with relatively SNC support over $B$ (by shrinking $B$), and $(Y, \Phi)$ is still terminal. 
Note that 
\begin{eqnarray}
\label{Eqnarray K_Y+ Phi}
K_Y+\Phi
&\sim_\IQ& \pi^*H -dF + F_2, \\ 
\nonumber
&=&  \frac{d}{t}\big(-(K_X+\D)-tF + \frac{t}{d} F_2\big). 
\end{eqnarray}
Let $\phi = f \circ \pi: Y \to Z$, which is a projective fibration. Hence 
\begin{eqnarray}
\label{Eqnarray. K_Yb + Phi_b}
K_{Y_b}+\Phi_b 
&\sim_\IQ& 
\frac{d}{t}\big(-(K_{X_b}+\D_b)-tF_b + \frac{t}{d} F_{2,b}\big), \\
\nonumber
\phi_{b,*} \CO_{Y_b}( m(K_{Y_b}+\Phi_b) )
& \cong &
\phi_{b,*} \CO_{Y_b}(\frac{dm}{t}(-(K_{X_b}+\D_b)-tF_b)) \\
\nonumber
&=& \CF_{F_b}^{dm} R_{\frac{dm}{t},b} 
\,\,\,=\,\,\, 
\CF_{F_b}^{tm'} R_{m',b}
\end{eqnarray}
for sufficiently divisible $m$ (then $m' = \frac{dm}{t}$ is integral), where ``$\cong$'' follows since $F_{2,b}$ is $\pi_b$-exceptional. Similarly 
\begin{eqnarray*}
\phi_* \CO_{Y}( m(K_{Y}+\Phi) )
& \cong &
\CF_{F}^{tm'} R_{m'}. 
\end{eqnarray*}

\begin{claim} \label{Claim. |L|_s  =  |L_s|}
$\phi_* \CO_{Y}( m(K_{Y}+\Phi) ) \cdot \CO_{Z_b}
 \cong
 \phi_{b,*} \CO_{Y_b}( m(K_{Y_b}+\Phi_b) ).$ 
\end{claim}
Assuming the claim, we have
\begin{eqnarray*}
(R_{m'} / \CF_{F}^{tm'} R_{m'}) \cdot \CO_{Z_b} 
\cong 
R_{m',b} / \CF_{F_b}^{tm'} R_{m',b}
\end{eqnarray*}
for any closed point $b\in B$. Thus the finite rank $\CO_B$-module $R_{m'} / \CF_{F}^{tm'} R_{m'}$ is locally free. Thus 
\begin{eqnarray*}
\vol(\ord_{F_b};t) = \mathop{\lim}_{m'\to \infty} \frac{\ell(R_{m',b} / \CF_{F_b}^{tm'} R_{m',b})}{m'^n/n!} 
=
\mathop{\lim}_{m'\to \infty} \frac{\rank_{\CO_B}(R_{m'} / \CF_{F}^{tm'} R_{m'})}{m'^n/n!}
\end{eqnarray*}
is independent of the choice of $b\in B$. 
\end{proof}

\begin{proof}[Proof of Claim]
We follow the strategy of \cite[Theorem 4.2]{HMX13} or \cite[Theorem 2.18]{Xu19}. Since $K_{Y_b}+\Phi_b$ is big over $Z_b$ by (\ref{Eqnarray. K_Yb + Phi_b}), its negative part $N_\sigma(Y_b/Z_b, K_{Y_b}+\Phi_b)$ (see \cite[Definition-Lemma 3.3.1]{BCHM10}) is a $\IQ$-divisor by \cite[Theorem 1.2]{BCHM10}. Let 
\begin{eqnarray*}
\Gamma_b := \Phi_b - \Phi_b \wedge N_\sigma(Y_b/Z_b, K_{Y_b}+\Phi_b). 
\end{eqnarray*}
Then there exists a $\IQ$-divisor $0\le \Gamma\le \Phi$ such that $\Gamma|_{Y_b} = \Gamma_b$. Note that $(Y_b,\Gamma_b)$ is terminal since $(Y,\Phi)$ is. 

We may run a relative MMP with scaling for $K_Y+\Gamma$ over $Z$, and it terminates at a relative minimal model $g: (Y, \Gamma) \dashrightarrow (Y', \Gamma' = g_*\Gamma)$ and $\phi': Y'\to Z$ by \cite{BCHM10}. Denote by $g^k: Y^k \dashrightarrow Y^{k+1}$ the $k$-th MMP step, and denote by $\Gamma^k$ the pushforward of $\Gamma$ to $Y^k$. By the same argument as in \cite[Proof of Theorem 4.1]{HMX13} or \cite[Proof of Theorem 2.18]{Xu19}, we may prove (by induction on $k$) that 
\begin{enumerate}[{\rm \quad (a)}]
\item $g^k: Y^k \dashrightarrow Y^{k+1}$ is an isomorphism at the generic point of every component of $\Gamma_b^k$; 
\item $g^k_b: Y^k_b \dashrightarrow Y^{k+1}_b$ is a birational contraction. 
\end{enumerate}
Suppose (a)$_{\le k-1}$ and (b)$_{\le k-1}$ hold. Then $Y_b\dashrightarrow Y_b^k$ is a birational contraction that does not contract any component of $\Gamma_b$. Hence $(Y_b^k, \Gamma_b^k)$ is still terminal. 

By definition, no component of $\Gamma_b$ is a component of the stable base locus of $K_{Y_b} + \Gamma_b$. The condition (b)$_{\le k-1}$ implies that no component of $\Gamma_b^k$ is a component of the stable base locus of $K_{Y_b^k} + \Gamma_b^k$. Let $D\seq Y_b^k$ be a prime divisor. If $g^k: Y^k\dashrightarrow Y^{k+1}$ is not an isomorphism at the generic point of $D$, then $D$ is covered by curves $C$ over $Z_b$ such that  
$$0 > C\cdot (K_{Y^k}+\Gamma^k) = C\cdot (K_{Y_b^k}+\Gamma_b^k). $$
It follows that $D$ is contained in the stable base locus of $K_{Y_b^k}+\Gamma_b^k$. Hence $D$ is not a component of $\Gamma_b^k$. Thus (a)$_{k}$ holds. 

Next we prove (b)$_k$. Since the MMP step $g^k$ is $(K_{Y^k}+\Gamma^k)$-negative, its restriction $g_b^k$ is also $(K_{Y_b^k}+\Gamma_b^k)$-negative. If $g^k_b: Y^k_b \dashrightarrow Y^{k+1}_b$ is not a birational contraction, then there is a divisorial component $D'$ of $\Ex((g_b^k)^{-1})$ with non-negative coefficient with respect to $(Y_b^{k+1}, \Gamma_b^{k+1})$, hence has positive coefficient with respect to $(Y_b^k, \Gamma_b^k)$ since $g_b^k$ is $(K_{Y_b^k}+\Gamma_b^k)$-negative. But as $(Y_b^k, \Gamma_b^k)$ is terminal, we know that $D'$ is a component of $\Supp(\Gamma_b^k)$, which contradicts (a)$_k$. 

Then $(g_*\Gamma)|_{Y'_b} = g_{b,*}\Gamma_b$. 
Hence
\begin{eqnarray*}
\phi_{b,*} \CO_{Y_b}( m(K_{Y_b}+\Phi_b) )
&=&
\phi_{b,*} \CO_{Y_b}( m(K_{Y_b}+\Gamma_b) ) \\
&=& 
\phi'_{b,*} \CO_{Y'_b}( m(K_{Y'_b}+\Gamma'_b) )   \\
&\cong& 
\phi'_{*} \CO_{Y'}( m(K_{Y'}+\Gamma') ) \cdot \CO_{Z_b} \\
&\seq& 
\phi'_{*} \CO_{Y'}( m(K_{Y'}+g_* \Phi) ) \cdot \CO_{Z_b}  \\
&=& 
\phi_{*} \CO_{Y}( m(K_{Y}+\Phi) ) \cdot \CO_{Z_b}, 
\end{eqnarray*}
where the first ``$=$'' follows from the definition of $\Gamma_b$, the second ``$=$'' follows since $g_b: Y_b\dashrightarrow Y'_b$ is a birational contraction and is $(K_{Y_b}+\Gamma_b)$-negative, ``$\cong$'' follows from the fact that $\phi': Y'\to Z$ is a relative minimal model of $(Y', \Gamma')$ and $(g_*\Gamma)|_{Y'_b} = g_{b,*}\Gamma_b$. The reverse inclusion is automatic, so equality holds. 
\end{proof}

\begin{cor}
\label{Corollary. invariance of DH}
With the same assumptions as Theorem \ref{Theorem. invariance of volume functions}, for any log canonical place $v\in \QM(Y, \pi_*^{-1} \D^+ + \Ex(\pi))$ of $(X,\D^+)$, whose center on $X$ dominates $B\seq Z$, the function $b\mapsto \DH_{v_b}$ is locally constant for $b\in B$. 
\end{cor}

\begin{proof}
It follows directly from Theorem \ref{Theorem. v to vol(v;1) is continuous} and \ref{Theorem. invariance of volume functions}. 
\end{proof}

\subsection{Existence of \texorpdfstring{$\BH$}{}-minimizers}

Now we are ready to show that the infimum in the definition $\BH(X,\D) = \inf_{v\in \Val_{X,o}^*} \BH(v)$ is a minimum. 

\begin{thm}
\label{Theorem:  Existence} 
There exists a weakly special valuation $v_0\in\Val_{X,o}^*$ such that 
$$\BH(X,\D)
\,\,\,=\,\,\, 
\BH(v_0) 
\,\,\,=\,\,\, 
\log\circ \BV(v_0). $$
\end{thm}

\begin{proof}


By Theorem \ref{Theorem. weakly special divisor approximation}, there exists a sequence of weakly special divisorial valuations $\{v_i\}_{i=1}^{\infty}$ such that 
$\lim_{i\to \infty} \BH(v_i) = \BH(X,\D). $
After rescaling, we may assume that the function $x\mapsto \BH(x\cdot v_i)$ is minimized at $x=1$. Then by Proposition \ref{Theorem. bounded log discrepancies},
$$A_i := A_{X,\D}(v_i) \le n.$$ 
By taking a sub-sequence, we may assume that $\BH(v_i)$ is non-increasing. Then by Theorem \ref{Theorem. Properness}, there exists $\vep = \vep(\dim X, \BH(v_1))>0$ such that $v_i(\fm_o) > \vep$ for all $i\ge 1$. 

Recall $L=-(K_X+\D)$ and $R_m = H^0(X,mL)$. 
By \cite[Theorem 1.8]{Bir19}, there exists an integer $N>0$ depending only on $\dim X$ and coefficients of $\D$ such that $N(K_X+\D)$ is Cartier, and there exist $N$-complements $\Psi_i$ of $(X,\D)$ such that $v_i$ is a log canonical place of $(X,\D+\Psi_i)$. We may write $N\Psi_i = \div(\psi_i)$ for some $\psi_i \in R_N$. 
On the other hand, by \cite{BCHM10}, there exists a projective birational morphism $\pi_i: Y_i \to (X,\D)$ extracting precisely one prime divisor $F_i$ such that $(Y_i,\pi_{i,*}^{-1}\D + F_i)$ is log canonical and $v_i=c_i \cdot \ord_{F_i}$ for some $c_i>0$. Denote by 
\begin{eqnarray*}
K_{F_i}+\D_{F_i}^+= \pi_i^*(K_X+\D+\Psi_i)|_{F_i}. 
\end{eqnarray*}
Then $(F_i,\D_{F_i}^+)$ is log canonical by adjunction. 

Fix $M\in \IN$ such that $M\cdot \vep >n\cdot N$. Choose $g_1,\cdots, g_K \in R_N$ such that their images $\bg_1,\cdots, \bg_K \in R_N/\fm_o^M R_N$ form a $\Ik$-basis of $R_N/\fm_o^M R_N$. Then there exists $a_{ij}\in\Ik$ such that $\bpsi_i=\sum_{j} a_{ij}\bg_j$. Let $\phi_i=\sum_{j} a_{ij}g_j \in R$. Then $\bphi_i=\bpsi_i$ in $R_N/\fm_o^M R_N$ for any $i$. Denote by $\Phi_i = \frac{1}{N}\div(\phi_i)$. 

Next, we show that $(X,\D+\Phi_i)$ is log canonical with $v_i$ as a log canonical place of it. 
In the following $i$ is fixed. 
Let $s_i=\phi_i-\psi_i \in \fm_o^M R_N$. Then $v_i(s_i)\ge M\cdot v_i(\fm_o) > M\cdot\vep> n\cdot N$. Hence $n\cdot N \ge A_i \cdot N=v_i(\psi_i) = v_i(\phi_i)$ and $A_{X,\D+\Phi_i}(v_i)=0$. 
We have
\begin{eqnarray*}
(K_{Y_i}+\pi_{i,*}^{-1}(\D+\Phi_i) + F_i)|_{F_i} 
&=& \pi_i^*(K_X+\D+\Phi_i)|_{F_i} \\
&=& \pi_i^*(K_X+\D+\Psi_i)|_{F_i} 
= K_{F_i}+\D_{F_i}^+. 
\end{eqnarray*}
Hence $(Y, \pi_{i,*}^{-1}(\D+\Phi_i) +F_i)$ is log canonical in a neighborhood of $F_i\seq Y_i$ by inversion of adjunction. We conclude that $(X,\D+ \Phi_i)$ is log canonical. 
We get the following parameter space of $N$-complements
\begin{eqnarray*}
V=\{ [x_1,\cdots,x_K] \in \IP^{K-1} \mid \lct(X,\D;\frac{1}{N}\div(\phi)) \ge 1,\,  \phi=\sum x_j g_j \in R_N \}, 
\end{eqnarray*}
which is a locally closed subset of $\IP^{K-1}$ by 
the lower semi-continuity of log canonical thresholds. Let $D_V\seq X\times V$ 
be the universal $\IQ$-divisor parameterizing divisors in $\frac{1}{N}|NL|$.  For any $z\in V$, let 
\begin{eqnarray}
\label{Eqnarray: h_z}
h_z := \mathop{\inf}_{v\in \LC(X, \D+D_z)} \BH(v). 
\end{eqnarray}
Choose a log resolution $(Y_z,E_z) \to (X,\D+D_z)$. Then $\LC(X, \D+D_z)\seq \QM(Y,E)$. Hence, the infimum in (\ref{Eqnarray: h_z}) is a minimum by Theorem \ref{Theorem. H is continuous on quasi-monomial cone}, that is, $h_z = \BH(v_z)$ for some $v_z\in \LC(X,\D+D_z)$. 

Since $(X_V, \D_V+D_V) := (X\times V, \D\times V+ D_V) \to V$ is a $\IQ$-Gorenstein family of pairs, we can divide $V$ into a disjoint union of finitely many locally closed subsets $V= \sqcup_j V_j$ such that, for each $j$, $V_j$ is smooth, and there exists an \'etale cover $V_j'\to V_j$ such that the base change $(X_{V_j'}, \D_{V_j'}+D_{V_j'})$ admits a fiberwise log resolution $(Y_{V_j'},E_{V_j'})$ over $V_j'$. By Corollary \ref{Corollary. invariance of DH}, for any $v \in \QM(Y_{V_j'},E_{V_j'}) \cap \LC(X_{V_j'}, \D_{V_j'}+D_{V_j'})$, the DH measure $\DH_{v_z}$ is constant for $z\in V_j'$. On the other hand, $A_{X,\D}(v_z)$ is constant for $z\in V_j'$ since $(Y_{V_j'},E_{V_j'})$ is simple normal crossing over $V_j'$. We conclude that $h_z$ is constant for $z\in V_j'$, denoted by $h_j$. 
Finally, by our choice of $N$ and $V$, we have $\BH(X,\D) = \inf_{z\in V} h_z = \min_{j} h_j$. Let $j_0$ be a minimizer. Then for any $z\in V_{j_0}'$, the minimizer $v_z$ of $h_z$ in (\ref{Eqnarray: h_z}) is the desired quasi-monomial valuation minimizing $\BH(X,\D)$.  
\end{proof}

\begin{cor}
\label{Corollary: Existence and Uniqueness} 
There exists a unique valuation $v_0\in\Val_{X,o}^*$ such that 
$$\BH(X,\D)
\,\,\,=\,\,\, 
\BH(v_0). $$
\end{cor}
\begin{proof}
It follows directly from Corollary \ref{Corollary: uniqueness of minimizer} and Theorem \ref{Theorem:  Existence}. 
\end{proof}

\begin{cor}
\label{Corollary: G-invariance of H-minimizer} 
If $(X,\D)\to Z\ni o$ admits the action by an algebraic group $G$. Then the $\BH$-minimizer $v_0$ is $G$-invariant. 
\end{cor}
\begin{proof}
For any $g\in G$, its action on $v_0$ is defined by the valuation $(g\cdot v_0)(s) = v_0(g^{-1}\cdot s)$ for any $s\in \CO_X$. Then 
$\CF_{g\cdot v_0}^\lam R_m = g \cdot \CF_{v_0}^\lam R_m$ for any $m\in l_0\IN, \lam\in\IR$. In particular, $I_{m,mt}(\CF_{g\cdot v_0}) = g\cdot I_{m,mt}(\CF_{v_0})$ and $\DH_{g\cdot v_0, m} = \DH_{v_0, m}$. Hence $\mu(\CF_{g\cdot v_0}) = \mu(\CF_{v_0})$ and $\tS({g\cdot v_0}) = \tS(v_0)$. 
So $\BH({g\cdot v_0}) = \BH(v_0) = \BH(X,\D)$. We conclude that $g\cdot v_0 = v_0$ by Corollary \ref{Corollary: uniqueness of minimizer}. 
\end{proof}

\begin{cor}
\label{Corollary. constructbility of H-invariants}
Let $(X,\D)\to Z\supseteq B$ be a $\IQ$-Gorenstein family of log Fano fibration germs over a scheme $B$ of finite type. Then the function $\BH: b\mapsto \BH(X_b,\D_b)$ is constructible and lower semicontinuous. 
\end{cor}
\begin{proof}
We follow the proof of \cite[Theorem 6.4]{BLXZ23}. By passing to a resolution of $B_\red$, we may assume $B$ is smooth. 
By the proof of Theorem \ref{Theorem:  Existence}, we know that the minimizers of $\BH(X_b, \D_b)$ are log canonical places of $\IQ$-complements. Hence there exists a constant $N\in \IN$ depending only on $\dim X$ and $\Coeff(\D)$, and a constant $\vep> 0$ depending only on the dimension of the generic fiber of $X\to B$, such that 
$$\BH(X_b,\D_b) = \min_v\{\BH_{X_b,\D_b}(v)\mid v \text{ is a } N \text{-complement and } v(\fm_b) > \vep \}, $$
where $\fm\seq \CO_Z$ is the ideal of $B\seq Z$. 
Then by the same argument as in the proof of Theorem \ref{Theorem:  Existence}, there exists a finite type variety $\phi: V\to B$ with a relative $N$-complement $D\seq X_V = X\times_B V$ over $V$ such that, for any closed point $b\in B$, any $N$-complement $\Gamma_b$ of $(X_b,\D_b)$, and any log canonical place $v$ of $(X_b,\D_b+\Gamma_b)$, there exists a closed point $z\in V $ mapping to $b\in B$ for which $v$ is a log canonical place of $(X_z,\D_z+D_z)$, where $(X_z,\D_z)=(X_b,\D_b)$. 

By resolving, stratifying and passing to finite base change, we may assume that $V$ is a union of finitely many smooth connected components $V_j$ such that $(X_{V_j}, \D_{V_j}+D_{V_j})$ admits a fiberwise log resolution. Then by the argument in the last paragraph of the proof of Theorem \ref{Theorem:  Existence}, we see that $h_z = \inf_{v\in\LC(X_z,\D_z+D_z)} \BH_{X_z,\D_z}(v)$ is constant for $z\in V_j$, and we denote this common value by $h_j$. Hence $B \ni b \mapsto \BH(X_b,\D_b) = \min\{h_j\mid b\in \phi(V_j)\}$ is constructible. 

In order to prove the lower semi-continuity of $b\mapsto \BH(X_b,\D_b)$, it suffices to consider the case when $B=\Spec S$ for some DVR $S$ essentially of finite type over $\Ik$. Let $K$ and $\kappa$ be the fractional field and residue field of $S$ respectively. 
Let $\CF_K$ be a filtration on $R_{m,K}$, then it extends to a filtration $\CF$ on $R_m$ and restricts to $\CF_\kappa$ on $R_{m,\kappa}$ by  
\begin{eqnarray}
\label{Eqnarray. filtration extension and restriction}
\CF^\lam R_m &:=& \CF_{K}^\lam R_{m,K} \cap R_m, \\
\CF_\kappa^\lam R_{m,\kappa} &:=& (\CF^\lam R_m )\otimes_S \kappa.  
\end{eqnarray}
Hence $\DH_{\CF_K} =\DH_{\CF_\kappa}$ and $\mu(\CF_K) \ge \mu(\CF_\kappa)$ by the lower semi-continuity of log canonical thresholds. We conclude that $\BH(X_K,\D_K) \ge \BH(X_\kappa, \D_\kappa)$. 
\end{proof}

\section{K-stability and Stable Degenerations}
\label{Section. K-stability and Stable degenerations}

Let $f: (X,\D)\to Z\ni o$ be a log Fano fibration germ, and $L=-(K_X+\D)$. 
In this section, we will show that the $\BH$-minimizer $v_0$ obtained in the previous section is a special valuation. We also show that the special degeneration $(X_0,\D_0,\xi_0)\to Z_0\ni o$ of $(X,\D)\to Z\ni o$ induced by $v_0$ is K-semistable.

\subsection{Minimizers of weighted delta invariants}

\begin{lem}
\label{Lemma. lower bound of S(L;L)}
Fix $v_0\in \Val_{X,o}$. 
For any $r \in \IZ_{\ge 1}$ and $s\in \IR_{\ge0}$, there exists a constant $\alpha(r ,s)>0$ such that, for any effective $\IQ$-divisor $D\sim_{\IQ} -(K_X+\D)$ satisfying $r D\sim -r (K_X+\D)$ and $v_0(r D)=s$, we have 
$S(v_0; \CF_D)= 2\alpha(r ,s)$. 
\end{lem}

\begin{rmk}\rm
In the log Fano case, $S(-(K_X+\D);D) = \frac{1}{n+1}$. 
\end{rmk}

\begin{proof}
It suffices to show that $S(v_0; \CF_D)$ depends only on $r$ and $s$. Choose $f\in R_r$ such that $rD=\div(f)$. For sufficiently divisible $m$, consider the surjective map: 
\begin{eqnarray*}
\CF_{v_0}^{\lam - sj}R_{m-rj} \to \CF_{f}^j \CF_{v_0}^\lam R_m, \quad 
g\mapsto g\cdot f^j. 
\end{eqnarray*}
Note that $g\cdot f^j \in \CF_{v_0}^{>\lam}R_m$ if and only if $v_0(g) > \lam-sj$. We get an isomorphism 
\begin{eqnarray}
\Gr_{v_0}^{\lam - sj}R_{m-rj} \cong 
\CF_{f}^j (\Gr_{v_0}^\lam R_m).
\end{eqnarray}
We simply denote by $V_{m,\lam} = \Gr_{v_0}^\lam R_m$. Then 
$V_{m-rj,\lam-sj}\cong 
\CF_{f}^j V_{m,\lam}.$ Recall that 
\begin{eqnarray*}
\BV_m(v_0) 
&=& \frac{n!}{m^n} \mathop{\sum}_{\lam\in\IR_{\ge0}} e^{\mu_0 - \frac{\lam}{m}} \cdot \ell(V_{m,\lam}), \\
S_m(v_0;\CF_f) 
&=& \frac{n!}{m^n\cdot \BV_m(v_0)} \mathop{\sum}_{\lam\in\IR_{\ge0}} e^{\mu_0 - \frac{\lam}{m}} \cdot 
\mathop{\sum}_{j=1}^\infty \frac{j}{m} \ell(\Gr_f^j V_{m,\lam}) \\ 
&=& \frac{n!}{m^n\cdot \BV_m(v_0)} \mathop{\sum}_{\lam\in\IR_{\ge0}} e^{\mu_0 - \frac{\lam}{m}} \cdot 
\frac{1}{m} \mathop{\sum}_{j=1}^\infty \ell(\CF_f^j V_{m,\lam}) \\ 
&=& \frac{n!}{m^n\cdot \BV_m(v_0)} \mathop{\sum}_{\lam\in\IR_{\ge0}} e^{\mu_0 - \frac{\lam}{m}} \cdot 
\frac{1}{m} \mathop{\sum}_{j=1}^\infty \ell(V_{m-rj,\lam-sj}). 
\end{eqnarray*}
Hence we see that $S_m(v_0;\CF_f)$ depends only on $r$ and $s$, and so is the limit $S(v_0;\CF_f)>0$. We are done since $S(v_0;\CF_f)=S(v_0;\CF_{rD})=\frac{1}{r}S(v_0;\CF_D)$. 
\end{proof}

\begin{thm}
\label{Theorem:  H-minimizer v_0 and delta^(g,v_0)}
Let $v_0\in\Val_{X,o}^*$ be a weakly special valuation with $\mu_0=\mu(v_0)=A_{X,\D}(v_0)$. 
Then $v_0$ minimizes $\BH$ if and only if $\delta(X,\D,v_0)=\frac{A_{X,\D}(v_0)}{S(v_0;v_0)}=1$. 
\end{thm}

\begin{proof}
The proof follows from \cite[Theorem 5.1]{BLXZ23}. We first prove the ``if'' part. By Theorem \ref{Theorem: weakly special approximations}, it suffices to show $\log\circ \BV(v)\ge \BH(v_0)$ for any valuation $v$ over $X$. 
Let $\CF_t$ be the geodesic connecting $\CF_0=\CF_{v_0}$ and $\CF_1:= \CF_{v}$. We define the following analog of $\BH(\CF_t)$:
\begin{eqnarray*} 
f(t) 
&:=& \log\Big(
\int_{\IR^2} 
e^{(1-t)(\mu_0-x)+t(A_{X,\D}(v)-y)} 
\DH_{\CF_0,\CF_1}(\dif x \dif y)
\Big). 
\end{eqnarray*} 
Then a similar argument to Theorem \ref{Theorem:  Convexity} shows that $f$ is convex. 
We have 
\begin{eqnarray*} 
f'(0) 
&=& \BV(v_0)^{-1}\cdot 
\int_{\IR^2} \big((A_{X,\D}(v)-y)-(\mu_0-x)\big) e^{\mu_0-x} \DH_{\CF_0,\CF_1}(\dif x \dif y)
 \\
&=& 
(A_{X,\D}(v)   - S(v_0;v))   
    -(\mu_0 - S(v_0;v_0)) \\
&=& 
A_{X,\D}(v)   - S(v_0;v)
\,\,\,\ge\,\,\, 0. 
\end{eqnarray*} 
Hence 
\begin{eqnarray*}  
\BH(v_0)=f(0)\le f(1)=\log\circ \BV(v). 
\end{eqnarray*}

Next, we prove the ``only if'' part. 
For any valuation $v$ over $X$, let $\CF_t$ and $f$ be the same as above. 
Since $\mu(\CF_v) \le A_{X,\D}(v)$, we have 
\begin{eqnarray*} 
\mu(\CF_t) 
\le (1-t)\mu(\CF_0) + t\mu(\CF_1) 
\le (1-t)\mu_0 + tA_{X,\D}(v). 
\end{eqnarray*} 
Hence $f(0)=\BH(v_0)\le\BH(\CF_t) \le f(t)$ for any $0\le t\le 1$. We conclude that $f'(0)\ge 0$ since $f$ is convex, that is, 
\begin{eqnarray*}  
A_{X,\D}(v) - S(v_0;v)
\ge \mu_0 - S(v_0;v_0) = A_{X,\D}(v_0) - S(v_0;v_0) 
\end{eqnarray*}
by the assumption. 
If $v=\lam v_0$, we see that
$$(\lam-1)(A_{X,\D}(v_0) - S(v_0;v_0))\ge 0 $$
for any $\lam > 0$. Hence $A_{X,\D}(v_0) - S(v_0;v_0) = 0$. The proof of Lemma \ref{Theorem:  H-minimizer v_0 and delta^(g,v_0)} is finished. 
\end{proof}


\begin{thm}
\label{Theorem: delta-minimizer is very special} 
Assume that $(X,\D)\to Z\ni o$ admits a good $\IT=\IG_m^l$-action ($0\le l\le n$). 
Let $v, v_0\in \Val_{X,o}^*$ be $\IT$-invariant weakly special valuations such that $\delta_\IT(X,\D,v_0)\le1$ is minimized by $v$. Assume that 
\begin{enumerate}
\item $v=v_0$; or 
\item $v_0=\wt_{\xi_0}$ for some $\xi_0\in \reeb$. 
\end{enumerate}
Then for any $\IT$-invariant effective $\IQ$-divisor $D\sim_{\IQ} -(K_X+\D)$, there exists $\alpha >0$ and a $\IQ$-complement $\Gamma$ of $(X,\D)$ such that $\Gamma\ge \alpha D$ and $v$ is a log canonical place of $(X,\D+\Gamma)$. 
\end{thm}

\begin{proof}
We follow the strategy of \cite[Lemma 3.1]{LXZ22}, see also \cite[Theorem 5.4]{BLXZ23} and \cite[Theorem 2.11]{Wang24b}. Choose $l_0>0 $ such that $l_0D\sim -l_0(K_X+\D)$ is Cartier. 
Let $\alpha=\alpha(l_0, v_0(l_0D))$ be the constant given by Lemma \ref{Lemma. lower bound of S(L;L)}. 

Rescale $v$ such that $A_{X,\D}(v)=1$ and let $\fa_\bu=\fa_\bu(v)$. Then $A_{X,\D+\fa_\bu}(v)=0$. By our assumption, the valuation $v$ is quasi-monomial, and there exists a quasi-monomial simplicial cone $\sigma\seq \Val_X$ containing $v$. 
The functions $w\mapsto A_{X,\D}(w)$ and $w\mapsto w(\fa_\bu)$ are linear and concave on $\sigma$ respectively. Hence the function $A_{X,\D+\fa_\bu}(-): \sigma \to \IR, w\mapsto A_{X,\D+\fa_\bu}(w) = A_{X,\D}(w) - w(\fa_\bu)$ is convex on $\sigma$. In particular, it is locally Lipschitz at $v \in \sigma$. There exists constants $K, C>0$ such that 
\begin{eqnarray*}
0\le A_{X,\D+\fa_\bu}(w) = 
|A_{X,\D+\fa_\bu}(w) -
   A_{X,\D+\fa_\bu}(v) |
\le C|w-v| 
\end{eqnarray*}
for any $w\in \sigma$ satisfying $|w-v|<K$. 

By \cite{HMX14}, there exists $0<\vep<CK$ depending only on $\dim X$ and coefficients of $\D' = \D+ \alpha D$ such that, for any birational morphism $\pi:Y\dashrightarrow X$ and any reduced divisor $E$ on $Y$, the pair $(Y,\pi_*^{-1}\D+(1-\vep)E)$ is log canonical if and only if $(Y,\pi_*^{-1}\D+E)$ is.

Assume that $v$ is of rational rank $r$. 
By Diophantine approximation \cite[Lemma 2.7]{LX18}, there exist divisorial valuations $\{v_1,\cdots, v_r\}\seq \sigma$ and positive integers $q_1,\cdots,q_r,c_1,\cdots,c_r$ such that 
\begin{itemize}
\item $\{v_1,\cdots,v_r\}$ spans a simplicial cone in $\sigma$ containing $v$; 
\item for any $1\le i\le r$, there exists a prime divisor $E_i$ over $X$ such that $q_iv_i=c_i\ord_{E_i}$; 
\item $|v_i-v|< \frac{\vep}{2Cq_i}$ for any $1\le i\le r$. 
\end{itemize}
In particular 
\begin{eqnarray*}
A_{X,\D+\fa_\bu}(E_i) 
= \frac{q_i}{c_i}\cdot A_{X,\D+\fa_\bu}(v_i) 
\le \frac{q_i}{c_i}\cdot C|v_i-v|
< \frac{q_i}{c_i}\cdot C \cdot \frac{\vep}{2C q_i} 
\le \frac{\vep}{2}. 
\end{eqnarray*}
Choose $\vep_0 >0$ small enough such that $\vep_0\ord_{E_i}(\fa_\bu) < \frac{\vep}{2}$ for any $1\le i\le r$, then 
\begin{eqnarray*}
A_{X,\D+\fa_\bu^{1-\vep_0}}(E_i) 
= A_{X,\D+\fa_\bu}(E_i) + \vep_0\ord_{E_i}(\fa_\bu)
<  \vep. 
\end{eqnarray*}

By the convergence delta invariants (Corollary \ref{Corollary. convergence of delta_m and delta^(t)}), we can choose a sequence of real numbers $\{t_m\}_{m\in l_0\IN}$ such that 
\begin{eqnarray*}
\mathop{\lim}_{m\to \infty} 
\delta_{m,mt_m}(X,\D,v_0) &=& \delta(X,\D,v_0), \\
\mathop{\lim}_{m\to \infty}
S_{m,mt_m}(v_0;v) &=& S(v_0;v). 
\end{eqnarray*}
Let $D_m$ be a $\IT$-invariant $v_0$-weighted $(m,mt_m)$-basis type divisor of $R$ which is compatible with both $v$ and $\CF_D$ (exists since $v=v_0$ or $v_0 =\wt_{\xi_0}$). Then 
$$D_m\ge S_{m,mt_m}(v_0;D)\cdot D\ge 2\alpha D$$ 
by Lemma \ref{Lemma. lower bound of S(L;L)}, and $v(D_m) = S_{m,mt_m}(v_0;v)$. Choose 
$$0<\delta_m<\min\{1,\delta_{m,mt_m,\IT}(X,\D,v_0)\}$$ 
satisfying $\lim_{m\to\infty} \delta_m = \delta_\IT(X,\D,v_0)=\delta\le1.$
Then $\delta_mD_m \ge \alpha D$ for $m\gg0$. Since $\delta_m<\delta_{m,mt_m,\IT}(X,\D,v_0)\le\lct(X,\D;D_m) $, we see that $(X,\D+\delta_mD_m)$ is log canonical. 

Since $v$ minimizes $\delta_\IT(X,\D,v_0)=\delta \le 1$, we have 
\begin{eqnarray*}
\mathop{\lim}_{m\to \infty} v(\delta_m D_m) = \delta\cdot \mathop{\lim}_{m\to \infty} S_{m,mt_m}(v_0;v) =\delta\cdot S(v_0;v)= A_{X,\D}(v)=1. 
\end{eqnarray*}
Hence $v(\delta_mD_m) \ge 1-\vep_0 = v(\fa_\bu^{1-\vep_0})$ for $m\gg0$. Then 
\begin{eqnarray*}
0\le a_i := A_{X,\D+\delta_mD_m}(E_i) \le A_{X,\D+\fa_\bu^{1-\vep_0}}(E_i) \le \vep. 
\end{eqnarray*}

By \cite[Corollary 1.4.3]{BCHM10}, there exists a $\IT$-equivariant $\IQ$-factorial birational model $\pi:Y\to X$ extracting precisely $E_1,\cdots, E_r$. Then 
\begin{eqnarray} 
\label{Eqnarray: crepant pullback 2}
&& 
\pi^*(K_X+\D+\delta_mD_m) \\
\nonumber
&\sim_\IQ& 
K_Y+\pi_*^{-1}(\D+\delta_mD_m)+\sum_{i=1}^r (1-a_i)E_i. 
\end{eqnarray}
In particular, $\pi^*(K_X+\D+\delta_mD_m)\ge K_Y+\pi_*^{-1}\D'+(1-\vep)E$, where $\D'=\D+\alpha D$. Since $\lct(X,\D;\delta_mD_m) >1$, the pair $(Y,\pi_*^{-1}\D'+(1-\vep)E)$ is log canonical. 
Hence $(Y,\pi_*^{-1}\D'+E)$ is also log canonical by our choice of $\vep$. 

Since $-(K_X+\D+\delta_mD_m)\sim_{\IR}-(1-\delta)(K_X+\D)$ is ample over $Z$, $(X,\D+\delta_mD_m)$ is a log canonical Fano fibration over $Z$. Hence $Y$ is of Fano type over $Z$ by (\ref{Eqnarray: crepant pullback 2}). We may run a $\IT$-equivariant $-(K_Y+\pi_*^{-1}\D'+E)$-MMP over $Z$ and get a good minimal model $\phi: Y\dashrightarrow \oY$ with the induced birational map $\opi: \oY\dashrightarrow X$. Then $-(K_{\oY}+\opi_*^{-1}\D'+\oE)$ is nef over $Z$. So it is semiample since $\oY$ is of Fano type over $Z$ which is affine, where $\oE=\phi_*E$. 
With the same argument as in the proof of Theorem \ref{Theorem: weakly special valuations}, we see that $(\oY,\opi_*^{-1}\D'+\oE)$ is log canonical and $\phi$ is an isomorphism at any stratum of $E$. 
Hence $K_Y+\pi_*^{-1}\D'+E \le \phi^*(K_\oY+\opi_*^{-1}\D'+\oE)$. Since $-(K_\oY+\opi_*^{-1}\D'+\oE)$ is semiample, $(\oY, \opi_*^{-1}\D'+\oE)$ admits a $\IQ$-complement by Bertini's theorem. In particular, $(Y,\pi_*^{-1}\D'+E)$ admits a $\IQ$-complement $\Gamma_Y$. Then $\Gamma=\pi_*\Gamma_Y+\alpha D$ is the desired $\IQ$-complement of $(X,\D)$. 
\end{proof}

\begin{thm}
\label{Theorem. H-minimizer has f.g. graded ring}
Let $v_0\in \Val_{X,o}^*$ be a minimizer of the $\BH$-invariant $\BH(X,\D)$. Then $v_0$ is special (Definition \ref{Definition. special degeneration induced by v}).
\end{thm}
\begin{proof}
We proceed using cone construction. Choose $l_0\in \IZ_{>1}$ such that $L=-l_0(K_X+\D)$ is Cartier and let $C=C(X,L) = \Spec \oplus_{k\in\IN} H^0(X,kL) \to Z$ be the relative cone of $X \to Z$. Then we have the section $o\in Z\seq C$ of vertices. Denote by $\xi$ the co-weight vector on $C$ reading the $k$-grading of $R= \oplus_{k\in\IN} H^0(X,kL)$ and consider the $\IG_m$-action on $C$ generated by $\xi$. Pulling back $v_0$ to $C$ and twisting by $\xi$, we get a quasi-monomial valuation $w_0 = v_{0,\xi} \in \Val_{C,o}^{\IG_m,*}$ over $C$. Let $\D_C$ be the $\IQ$-divisor on $C$ defined by (\ref{Eqnarray. D_C}), where $\lam=l_0^{-1} <1$. Then $(C,\D_C)$ is klt by Lemma \ref{lem:cone singularity}. 

There exists a $\IG_m$-equivariant log resolution $\chi: (Y,E)\to (C,\D_C)$ whose exceptional locus supports an effective $\chi$-antiample $\IQ$-divisor $F$ and $w_0\in \QM(Y,E)$. For sufficiently divisible $m>0$, there exists $\IG_m$-invariant $s\in \chi_*\CO_Y(-mF) \cap H^0(X, kL)$ for some $k>0$ such that $\chi_*^{-1}\div_C(s) \sim_{\IQ, \chi} -m F$ does not contain $C_Y(w_0)$, that is, $w_0(\chi_*^{-1}\div_C(s)) = 0$. Since $C$ is affine, the $\chi$-ample divisor $\chi_*^{-1}\div_C(s)$ is indeed ample. Let $D_0= (kl_0)^{-1}\cdot \div_C(s)$ and $D= (kl_0)^{-1}\cdot \div_X(s)$, then $D_0 = \varphi_*\pi^*D$ (notion as in Section \ref{ssec: Cone construction}). We have that $kl_0 D \sim -kl_0(K_X+\D)$. Let $\alpha = \alpha(kl_0, v_0(kl_0 D))$ be the constant given by Lemma \ref{Lemma. lower bound of S(L;L)}. By Theorem \ref{Theorem: delta-minimizer is very special}, there exists a $\IQ$-complement $\Gamma\ge \alpha D$ of $(X,\D)$ such that $v_0\in \LC(X,\D+\Gamma)$. Hence $\Gamma_0 = \varphi_*\pi^*\Gamma$ is a $\IQ$-complement of $(C,\D_C)$ such that $w_0 \in \LC(C,\D_C+\Gamma_0)$ by Lemma \ref{lem:cone singularity}. By construction, the $\IQ$-complement $\Gamma_0$ is special with respect to $w_0$ (Definition \ref{Definition. special complements with respect to w}).
By Corollary \ref{Corollary. XZ special valuations_slightly generalization}, we see that $\Gr_{v_0}R \cong \Gr_{w_0} R$ is finitely generated and the corresponding degeneration $(C_0= \Spec (\Gr_{v_0} R), \D_{C_0})$ of $(C,\D_C)$ is klt. Hence so is 
$$(X_0=\Proj_{\Gr_{v_0}R_0}(\Gr_{v_0}R),\D_{X_0}) \to \Spec (\Gr_{v_0} R_0) = Z_0 \ni o $$
by Lemma \ref{lem:cone singularity}. 
\end{proof}

\begin{cor}
\label{Corollary. special divisor approximations of H}
We have $\BH(X,\D) = \inf_{v} \, \BH(v), $
where the infimum runs over all the special divisorial valuations $v\in \Val_{X,o}^*$. If $(X,\D)\to Z\ni o$ admits the action by an algebraic group $G$, then we may assume that $v\in \Val_{X,o}^{G,*}$. 
\end{cor}
\begin{proof}
Notation as in the proof of Theorem \ref{Theorem. H-minimizer has f.g. graded ring}. Let $\sigma_0 \seq \LC(C,\D_C+\Gamma_0)$ be a minimal simplicial cone containing $w_0$. Then $\sigma_0 = \sigma \times (\IR_{\ge 0} \cdot \xi) $ for some simplicial cone $\sigma \seq \LC(X,\D+\Gamma)$. By \cite[Theorem 4.1]{XZ-sdsing}, any $w\in \sigma_{0,\IQ}$ is a Koll\'ar component (i.e. special divisorial valuation) over $o\in (C,\D_C)$. Hence, any $v\in \sigma_\IQ$ is a special divisorial valuation over $(X,\D)$ using relative cone construction as in the proof of Theorem \ref{Theorem. H-minimizer has f.g. graded ring}. Since $\BH$ is continuous on $\sigma$, we conclude that $\BH(X,\D) = \BH(v_0) = \inf_{v\in \sigma_\IQ} \, \BH(v)$. 

By Corollary \ref{Corollary: G-invariance of H-minimizer}, we know that $v_0$ is $G$-invariant. Hence, the above cone $\sigma$ can be chosen to be $G$-invariant. Hence, the above $v\in\sigma_\IQ$ are $G$-invariant. 
\end{proof}

\subsection{K-stability of polarized log Fano fibration germs}

Let $(X,\D)\to Z\ni o$ be a log Fano fibration germ with a good $\IT=\IG_m^r$-action. Let $\reeb\seq \IR^{r}$ be the Reeb cone of the $\IT$-action on $(X,\D)\to Z\ni o$ introduced in Section \ref{Subsection. torus action and xi-twist}. 

\begin{defi}\rm
\label{Definition: weighted Ding stability}
For any $\IT$-invariant filtration $\CF$ on $R$ and $\xi_0\in \reeb$, we define the {\it $\xi_0$-weighted Ding invariant} of $\CF$ by 
\begin{eqnarray*}
\BD(\CF) 
\,\,\,=\,\,\, \BD_{\xi_0}(\CF)
\,\,\,=\,\,\, \BD_{X,\D,\xi_0}(\CF) 
\,\,\,:=\,\,\,
\mu_{X,\D}(\CF) - S(\xi_0;\CF). 
\end{eqnarray*}
\end{defi}

\begin{rmk}\rm 
Since $\mu(\CF)$ and $\DH_\CF$ are affine with respect to shifting, for any $b\in \IR$, we have 
$\BD(\CF(b))=\BD(\CF). $
Let $(\CX,\D_\CX,\xi_0;\CL)$ be a normal test configuration of $(X,\D,\xi_0)\to Z\ni o$. We simply denote by 
\begin{eqnarray*}
\BD(\CX,\D_\CX,\xi_0;\CL) 
\,\,\,=\,\,\, \BD_{X,\D,\xi_0}(\CF_{(\CX,\D_\CX;\CL)}). 
\end{eqnarray*}
For any $v\in \Val_{X,o}^{\IT,*}$ and $\xi\in\reeb$, we simply denote by 
\begin{eqnarray*}
\BD(v) = \BD(\CF_v), \quad \Fut(\xi) = \BD(\CF_{\triv,\xi}) = \BD(\wt_\xi), 
\end{eqnarray*}
where the last equality follows from (\ref{Eqnarray: valuation of product TC}). 
\end{rmk}

\begin{defi}\rm
A polarized log Fano fibration germ $(X,\D,\xi_0)\to Z\ni o$ is called {\it $\IT$-equivariantly Ding-semistable} if $\BD(\CF)\ge 0$ for any $\IT$-invariant filtration $\CF$ on $R$. 
It is called {\it $\IT$-equivariantly K-semistable} if $\BD(\CX,\D_\CX,\xi_0;\CL)\ge 0$ for any $\IT$-equivariant special test configuration $(\CX,\D_\CX;\CL)$. 
If, moreover, $\BD(\CX,\D_\CX,\xi_0;\CL)= 0$ implies that $(\CX,\D_\CX,\xi_0;\CL)$ is a product test configuration, then $(X,\D)$ is called {\it $\IT$-equivariantly K-polystable}. 

If $(X,\D)\to Z\ni o$ admits the action by an algebraic group $G$, such that $\IT$ is contained in the center of $G$, then we may define the corresponding $G$-equivariant stability notion as above by replacing the group $\IT$ with $G$. 
\end{defi}

\begin{ex}\rm 
\label{Example. soliton candidate}
Let $(X,\D)\to Z\ni o$ be a log Fano fibration germ admitting a good $\IT=\IG_m^r$-action. Let $\BP\seq M_\IR$ be the moment polyhedral and ${\reeb} = \reeb(X,\IT)\seq N_\IR$ be the Reeb cone. The restriction of the $\BH$-functional on ${\reeb}$ is 
$$\BH(\xi)= \BH(\wt_\xi) = \BH(\CF_{\triv,\xi}) = \int_{\BP} e^{-\la\alpha,\xi\ra} \DH_\BP(\dif \alpha).$$ 
By Corollary \ref{Corollary. 0 in int(P)}, we have $0\in\interior(\BP)$. Hence $\BH|_{\reeb} :{\reeb} \to \IR_{>0}$ is strictly convex and proper. There exists a unique vector $\xi_0\in{\reeb}$ minimizing $\BH|_{\reeb}$. The first order derivative of $\BH|_{\reeb}$ at $\xi_0$ is zero: 
$$\Fut(\xi) = \frac{\dif}{\dif t}|_{t=0} \BH(\xi_0+t\xi) = -\int_{\BP} \la\alpha,\xi\ra e^{-\la\alpha,\xi_0\ra} \DH_\BP(\dif \alpha) =0. $$
This is a necessary condition of $(X,\D,\xi_0)\to Z\ni o$ being K-semistable. We say that the $\BH$-minimizer $\xi_0$ on ${\reeb}$ is the {\it soliton candidate} of $(X,\D)\to Z\ni o$ with respect to the $\IT$-action. 
\end{ex}

\begin{ex}\rm
Let $(X,\D)\to Z\ni o$ be a toric log Fano fibration germ (where $\dim X=n$), that is, it admits a good $\IT=\IG_m^n$-action. Let $\xi_0$ be the soliton candidate defined in Example \ref{Example. soliton candidate}. Then $(X,\D,\xi_0)\to Z\ni o$ is $\IT$-equivariantly K-semistable since any $\IT$-invariant valuation $v\in \Val_{X,o}^*$ is toric, that is, $v=\wt_\xi$ for some $\xi\in \reeb$ by \cite[Lemma 4.2]{BHJ17}. 
\end{ex}

\begin{lem}
\label{Lemma. Ding-ss if and only if delta ge 1}
Let $(X,\D,\xi_0)\to Z\ni o$ be a polarized log Fano fibration germ with a good $\IT$-action. Then it is $\IT$-equivariantly Ding-semistable if and only if 
\begin{eqnarray*}
\delta_\IT(X,\D,\xi_0) \ge 1. 
\end{eqnarray*}
\end{lem}
\begin{proof}
It follows Lemma \ref{Lemma. canonical shift}. 
Indeed, the ``only if'' part is clear since 
$$A_{X,\D}(v) - S(\xi_0;v) \ge \mu_{X,\D}(\CF_v)- S(\xi_0;v) = \BD(\CF_v)\ge 0$$ 
for any $v\in\Val_{X,o}^{\IT,*}$ by Lemma \ref{Lemma: mu<A}. Conversely, for any $\IT$-invariant filtration $\CF$, by Lemma \ref{Lemma. canonical shift}, let $v$ be a $\IT$-invariant weakly special valuation minimizing $\lct(X,\D;I^{(\mu(\CF))}_\bu)$, then  
\begin{eqnarray*}
\CF'=\CF(A_{X,\D}(v)-\mu(\CF)) \seq \CF_v, \quad 
\mu(\CF') = A_{X,\D}(v). 
\end{eqnarray*}
Hence $S(\xi_0;\CF') \le S(\xi_0;v)$ and 
\begin{eqnarray*}
\BD(\CF) = \BD(\CF') = \mu(\CF')-S(\xi_0;\CF') \ge A_{X,\D}(v)- S(\xi_0;v) = \BD(v) \ge 0, 
\end{eqnarray*}
where the last inequality follows from $\delta_\IT(X,\D,\xi_0)\ge1$. Hence $(X,\D,\xi_0)$ is $\IT$-equivariantly Ding-semistable. 
\end{proof}

\begin{thm}
\label{Theorem. equivalence of stability notions}
Let $(X,\D,\xi_0)\to Z\ni o$ be a polarized log Fano fibration germ with a good $\IT$-action. The following statements are equivalent. 
\begin{enumerate}[{\rm \quad (a)}]
\item It is $\IT$-equivariantly Ding-semistable. 
\item It is $\IT$-equivariantly K-semistable. 
\item The $\BH$-invariant $\BH(X,\D)$ is minimized by $\wt_{\xi_0}$. 
\item The delta invariant $\delta_\IT(X,\D,\xi_0) =1$ is minimized by $\wt_{\xi_0}$. 
\end{enumerate}
\end{thm}
\begin{proof}
By definition, we have (a) $\Rightarrow$ (b). The equivalence (c) $\Leftrightarrow$ (d) follows from Theorem \ref{Theorem:  H-minimizer v_0 and delta^(g,v_0)}, and (d) $\Rightarrow$ (a) follows from Lemma \ref{Lemma. Ding-ss if and only if delta ge 1}. 

Now we prove the key step (b) $\Rightarrow$ (c). Since $\BH(X,\D)$ admits a $\IT$-invariant special divisorial approximation by Corollary \ref{Corollary. special divisor approximations of H}. It suffices to prove that $\BH(v)\ge \BH(\wt_{\xi_0}) = \BH(\CF_{\triv,\xi_0})$ for any $\IT$-invariant special divisorial valuation $v\in\Val_{X,o}^{\IT,*}$. 
The K-semistability of $(X,\D,\xi_0)$ implies that 
$$\Fut(\xi) = -\int_{\BP} \la\alpha,\xi\ra e^{-\la\alpha,\xi_0\ra} \DH_\BP(\dif \alpha) =0$$ 
for any $\xi\in \reeb$, hence for any $\xi\in N_\IR$ by the linearity of $\Fut|_{N_\IR}$. 
Since $\BH$ is strictly convex along geodesics, it suffices to show that the derivative of $\BH_{X,\D}(\CF_t)$ at $t=0$ is non-negative, where $\CF_t$ is the geodesic connecting $\CF_0= \CF_{\triv, \xi_0}$ and $\CF_1=\CF_v$. Note that
\begin{eqnarray*}
\CF^\lam_t R_m 
&=& \sum_{(1-t)\mu + t\nu\ge \lam}
\CF_0^\mu R_m \cap \CF_1^\nu R_m \\
&=& \bigoplus_{\alpha\in M} \Big\{
s\in R_{m,\alpha}\mid (1-t)\la\alpha,\xi_0\ra + t\, v(s) \ge \lam
\Big\} 
\,\,\,=\,\,\,
\CF_{tv,(1-t)\xi_0}^\lam R_m. 
\end{eqnarray*}
Hence $\CF_t = \CF_{tv,(1-t)\xi_0}$. By Lemma \ref{Lemma. log canonical slope is linear w.r.t rescaling and shifting} and Proposition \ref{Lemma. log canonical slope is invariant under xi-twist}, we have $\mu(\CF_t) = t\mu(v)$. On the other hand, by Lemma \ref{Lemma. DH-measure rescaling and shifting} and Lemma \ref{Lemma. equivariant DH-measure twisting}, we have 
$$\int_{\BP\times \IR} f(\alpha, x) \DH_{\BP,\CF_t} (\dif \alpha \dif x)
= \int_{\BP\times \IR} f(\alpha, tx+(1-t)\la\alpha,\xi_0\ra ) \DH_{\BP,v}(\dif \alpha \dif x) $$
for any measurable function $f$ on $\BP\times \IR$. 
Hence 
\begin{eqnarray*}
\BH(\CF_t) 
&=& \log\Big( 
\int_{\BP\times \IR} 
e^{\mu(\CF_t) - x} \DH_{\BP, \CF_t}(\dif \alpha \dif x)
\Big) \\
&=& \log\Big(
\int_{\BP\times \IR} 
e^{t\mu(v) - tx - (1-t) \la \alpha,\xi_0 \ra } 
\DH_{\BP,v}(\dif \alpha \dif x)
\Big) \\
&=& \log\Big(
\int_{\BP\times \IR} 
e^{-\la \alpha,\xi_0 \ra + t( \mu(v) + \la \alpha,\xi_0 \ra -x)} 
\DH_{\BP,v}(\dif \alpha \dif x) 
\Big).
\end{eqnarray*}
We conclude that 
\begin{eqnarray*}
\frac{\dif}{\dif t}|_{t=0}\, \BH(\CF_t) 
&=& 
\frac{\int_{\BP\times \IR} 
\big( \mu(v) + \la \alpha,\xi_0 \ra -x\big) \cdot 
e^{-\la \alpha, \xi_0\ra} 
\DH_{\BP,v}(\dif \alpha \dif x)}{\int_{\BP} e^{-\la \alpha, \xi_0\ra} 
\DH_{\BP}(\dif \alpha) } \\
&=& 
\frac{\int_{\BP\times \IR} 
\big( \mu(v) -x\big) \cdot 
e^{-\la \alpha, \xi_0\ra} 
\DH_{\BP,v}(\dif \alpha \dif x)}{\int_{\BP} e^{-\la \alpha, \xi_0\ra} 
\DH_{\BP}(\dif \alpha) } \\
&=& 
\BD_{X,\D,\xi_0}(v)
\,\,\, \ge \,\,\, 0.  
\end{eqnarray*}
\end{proof}

In particular, $\BH|_{\reeb}$ is minimized by $\xi_0$. By Example \ref{Example. soliton candidate}, we have: 

\begin{cor}
\label{Corollary. K-ss to Fut|_N = 0}
If $(X,\D,\xi_0) \to Z\ni o$ is $\IT$-equivariantly K-semistable, then $\BD_{X,\D,\xi_0}(\xi) = 0$ for any $\xi\in\reeb(X,\IT)$. 
\end{cor}

Note that (c) in Theorem \ref{Theorem. equivalence of stability notions} is independent of $\IT$. 

\begin{cor}\rm
\label{Corollary. non-equivariant semistability}
Assume that $(X,\D)\to Z\ni o$ admits the action by an algebraic group $G$, $\IT$ is a sub-torus of the center of $G$ and $\xi_0\in \reeb(X,\IT)$. Then $(X,\D,\xi_0)\to Z\ni o$ is $\IT$-equivariantly K-semistable if and only if it is $G$-equivariantly K-semistable. 
\end{cor}
\begin{proof}
By definition, the ``only if'' part is clear. Conversely, if it is $G$-equivariantly K-semistable, then $\BH(X,\D)$ is minimized by $\wt_{\xi_0}$ (which is $G$-invariant) by replacing $\IT$ with $G$ in the proof of (b) $\Rightarrow$ (c) in Theorem \ref{Theorem. equivalence of stability notions}. Hence $(X,\D)\to Z\ni o$ is $\IT$-equivariantly K-semistable. 
\end{proof}

\subsection{K-polystable degenerations}

\begin{lem}
\label{Lemma. special degeneration with vanishing Ding is still K-semistable}
Let $(X,\D,\xi_0)\to Z\ni o$ be a polarized log Fano fibration germ admitting a good $\IT$-action. Assume that $(X,\D,\xi_0)\to Z_v\ni o$ is $\IT$-equivariantly K-semistable. Let $v\in \Val_{X,o}^*$ be a special valuation. Then $(X_v,\D_v,\xi_0) \to Z\ni o$ is $\IT\times \IG_m$-equivariantly K-semistable if and only if $\BD_{X,\D,\xi_0}(v)=0$. 
\end{lem}

\begin{proof}

By the K-semistability of $(X,\D,\xi_0)\to Z\ni o$, we have $\BD_{X,\D,\xi_0}(v) \ge 0$. If $(X_v,\D_v,\xi_0) \to Z\ni o$ is $\IT\times \IG_m$-equivariantly K-semistable, then by Corollary \ref{Corollary. K-ss to Fut|_N = 0}, we have
\begin{eqnarray*} 
\BD_{X,\D,\xi_0}(v)
= \BD_{X_{v},\D_{v},\xi_0}(\xi_v) = 0. 
\end{eqnarray*} 

Conversely, assume that $(X_v,\D_v,\xi_0) \to Z\ni o$ is not $\IT\times \IG_m$-equivariantly K-semistable. By definition, there exists a special divisorial valuation $w\in \Val_{X_v, o}^*$ of $(X_v,\D_v,\xi_0) \to Z\ni o$ such that $\BD_{X_v,\D_v,\xi_0}(w) < 0$. Then by Theorem \ref{Theorem. Replacing two-step dege. to one-step dege.}, there exists a sequence of special divisorial valuations $\{v_\vep\}_{\vep\in \IQ_{>0}} \seq \Val_{X,o}^*$ of $(X,\D)\to Z\ni o$ such that 
\begin{eqnarray*} 
(X_{v_\vep}, \D_{v_\vep}, \xi_{v_\vep}) = 
(X_{v,w},\D_{v,w}, \xi_{v} + \vep \xi_w). 
\end{eqnarray*}
Hence 
\begin{eqnarray*} 
\BD_{X,\D,\xi_0}(v_\vep)  
&=& \BD_{X_{v,w},\D_{v,w},\xi_0}(\xi_{v} + \vep \xi_w) \\
&=& \BD_{X_{v,w},\D_{v,w},\xi_0}(\xi_{v}) +
  \vep \cdot \BD_{X_{v,w},\D_{v,w},\xi_0}(\xi_w) \\
&=& \BD_{X,\D,\xi_0}(v) + 
  \vep \cdot \BD_{X_{v},\D_{v},\xi_0}(w).  
\end{eqnarray*}
If $\BD_{X,\D,\xi_0}(v) = 0$, then $\BD_{X,\D,\xi_0}(v_\vep)<0$, which contradicts the K-semistability of $(X,\D,\xi_0)\to Z\ni o$. Hence $\BD_{X,\D,\xi_0}(v) > 0$. 
\end{proof}

Let $S$ be a DVR with fractional field $K$ and residue field $\kappa$.  Let $(X_S,D_S,\xi_0)\to Z_S \to \Spec S$ be a $\IT=\IG_m^l$-equivariant $\IQ$-Gorenstein family of K-semistable polarized log Fano fibration germs (admitting a section $o: \Spec S \to Z_S$) with respect to the trivial $\IG_m^l$-action on $\Spec S$. Let $(\CX_K, \CD_K, \xi_0) \to \CZ_K \supseteq \IA^1_K$ be a $\IT$-equivariant special test configuration of $(X_K,D_K,\xi_0)\to Z_K \ni o_K$ with vanishing $\BD$ induced by a $\IT$-invariant special divisor $E_K$ over $(X_K,D_K)$. Then it extends trivially to the family 
\begin{eqnarray}
\label{Eqnarray. punctured family over A^1 * B}
(\toCX, \toCD, \xi_0) \to \toCZ \to \toIA_S = (\IA^1 \times \Spec S) \setminus (0,\Spec \kappa). 
\end{eqnarray}

Denote the anti-canonical ring by $R=R(X_S,D_S) = \oplus_{m\in l_0\IN} R_m$ where 
$$R_m=H^0(X_S,-m(K_{X_S}+D_S)) $$
is a finite $S$-module. 
The filtration induced by $E_K$ on $R_{m,K}$ extends naturally to $R_m$ by 
\begin{eqnarray}
\label{Eqnarray. filtration induced by general fiber}
\CF^\lam R_m := \CF_{E_K}^\lam R_{m,K} \cap R_m, 
\end{eqnarray}
where $R_m\seq R_{m,K}$ via localization. We define the extended Rees algebra of $\CF$ by 
\begin{eqnarray*}
\Rees_\CF R = \bigoplus_{m\in l_0\IN} \bigoplus_{\lam\in \IZ} t^{-\lam} \CF^\lam R_m.   
\end{eqnarray*}

The following theorem is similar to \cite[Proposition 5.3]{ABHLX19}. 

\begin{thm}
\label{Theorem. special degeneration of family}
Let $(\CX_K, \CD_K, \xi_0) \to \CZ_K \supseteq \IA^1_K$ be a $\IT$-equivariant special test configuration of $(X_K,D_K,\xi_0)\to Z_K \ni o_K$ with vanishing $\BD$ corresponding to a $\IT$-invariant special divisor $E_K$ over $(X_K,D_K)$. 
The extended Rees algebra $\Rees_\CF R$ is finitely generated and 
\begin{eqnarray}
\label{Eqnarray. extending two-dim puntured family}
(\CX= \Proj_{\CZ} \Rees_\CF R, \CD, \xi_0) \to \Spec \Rees_\CF R_0 = \CZ \to \IA^1_S
\end{eqnarray}
is an extension of (\ref{Eqnarray. punctured family over A^1 * B}) with K-semistable central fiber.  
\end{thm}

\begin{proof}
We follow the proof of \cite[Theorem 8.18]{Xu-kstabilitybook}. Let
\begin{eqnarray}
\label{Eqnarray. filtration restrict to the central fiber}
\CF_\kappa^\lam R_{m,\kappa} = (\CF^\lam R_m )\otimes_S \kappa
\end{eqnarray}
be the restriction of $\CF$ to the central fiber. By (\ref{Eqnarray. filtration induced by general fiber}), we see that $\Gr_\CF^\lam R_m$ is a torsion-free $S$-module.
Then we have
\begin{eqnarray*}
\dim_K (\Gr_\CF^\lam R_m \otimes_S K )
= \dim_\kappa (\Gr_\CF^\lam R_m \otimes_S \kappa )
= \dim_\kappa (\Gr_{\CF_\kappa}^\lam R_{m,\kappa}). 
\end{eqnarray*}
In particular, 
\begin{eqnarray*}
\DH_{\xi_0,\CF_\kappa,m}(X_\kappa,D_\kappa)
&=&\DH_{\xi_0,\CF_K,m}(X_K,D_K), \\
S_{X_\kappa,D_\kappa}(\xi_0;\CF_\kappa)
&=&S_{X_K,D_K}(\xi_0; \CF_{K}). 
\end{eqnarray*} 
On the other hand, we have $\mu_{X_\kappa,D_\kappa}(\CF_\kappa) \le \mu_{X_K,D_K}(\CF_K)$ by the lower semi-continuity of log canonical thresholds. Hence 
\begin{eqnarray*}
\mu_{X_\kappa,D_\kappa}(\CF_\kappa) - S_{X_\kappa,D_\kappa}(\xi_0;\CF_\kappa) \le \mu_{X_K,D_K}(\CF_K) - S_{X_K,D_K}(\xi_0; \CF_{K}) = 0, 
\end{eqnarray*}
and $\mu_{X_\kappa,D_\kappa}(\CF_\kappa) = \mu_{X_K,D_K}(\CF_K) =: \mu$ by the K-semistability of fibers of $(X_S,D_S,\xi_0) \to \Spec S$. Then 
\begin{eqnarray*}
\label{Eqnarray: constant log canonical threshold}
1
=   \lct(X_\kappa,D_\kappa; I^{(\mu)}_{\bu,\kappa})
&=& \lct(X_S,D_S+X_\kappa; I^{(\mu)}_\bu) \\
&\le& \lct(X_S,D_S; I^{(\mu)}_\bu)
\le \lct(X_K,D_K; I^{(\mu)}_{\bu,K})
=1, 
\end{eqnarray*}
where $I^{(\mu)}_\bu=I^{(\mu)}_\bu(\CF)$ (on $X_S$) and the second equality follows from inversion of adjunction. We see that $\lct(X_S,D_S+X_\kappa; I^{(\mu)}_\bu) = 1$. So for sufficiently small $\vep>0$ and sufficiently large $m$, we have 
\begin{eqnarray*}
\lct(X_S,D_S+X_\kappa; I_{m,m(\mu-\vep)}) \ge 1. 
\end{eqnarray*}
Hence $(X_S,D_S+X_\kappa +D_m)$ is log canonical for general $D_m\in \frac{1}{m}|\CF^{m(\mu-\vep)}R_m| $, and 
$$0\le A_{X_S,D_S+X_\kappa +D_m}(\ord_{E_K}) = \mu - \ord_{E_K}(D_m) \le \vep.$$
By \cite{BCHM10}, there exists a projective birational morphism $Y\to X$ extracting precisely one prime divisor $E$ satisfying $\ord_E=\ord_{E_K}$. Hence $\Rees_\CF R=\Rees_E R$ is finitely generated by \cite{BCHM10} and $\mu(E) =\mu= A_{X_S,D_S}(E)$. 
By Proposition \ref{Proposition. weakly special divisor 2}, $(\CX,\CD,\xi_0)$ in (\ref{Eqnarray. extending two-dim puntured family}) is a weakly special test configuration of $(X_S,D_S,\xi_0)$. In particular, $(\CX,\CD+\CX_0)$ is log canonical. Hence $(\CX_\kappa,\CD_\kappa+\CX_{\kappa,0})$ is log canonical by adjunction. Since 
$$\BD_{X_\kappa, D_\kappa,\xi_0}(\CX_\kappa,\CD_\kappa)=\BD_{X_\kappa, D_\kappa,\xi_0}(\CF_\kappa) = 0, $$ 
we see that $v= \ord_{X_{\kappa,0}}$ minimizes $\delta_\IT(X_\kappa, D_\kappa,\xi_0)= 1$ by the K-semistability of $(X_\kappa, D_\kappa,\xi_0)$. By Theorem \ref{Theorem: delta-minimizer is very special} and Theorem \ref{Proposition: special divisor definition 1}, $v$ is a special divisor over $(X_\kappa, D_\kappa,\xi_0)$, so $(\CX_\kappa,\CD_\kappa+\CX_{\kappa,0})$ is plt. We conclude by Lemma \ref{Lemma. special degeneration with vanishing Ding is still K-semistable} that the central fiber $(\CX_{\kappa,0}, \CD_{\kappa,0},\xi_0)$ is K-semistable. 
\end{proof}

\begin{cor}
\label{Corollary. uniqueness of K-polystable degeneration}
Let $(\CX^{(i)},\CD^{(i)},\xi_0)\to \CZ^{(i)}\supseteq \IA^1$ ($i=1,2$) be $\IT$-equivariant special test configurations of $(X,\D,\xi_0)\to Z\ni o$ with K-semistable central fibers. Then $(\CX^{(i)}_0,\CD^{(i)}_0,\xi_0)$ specially degenerate to a common K-semistable polarized log Fano fibration germ. 
\end{cor}
\begin{proof}
Let $\CX|_{(\IA^1\setminus 0)^2} = \CX^{(1)}|_{\IA^1\setminus 0} \times \CX^{(2)}|_{\IA^1\setminus 0} \cong X\times (\IA^1\setminus 0)^2$. Then it extends naturally to a punctured family $\toCX\to \IA^2\setminus (0,0)$ such that 
$$\toCX|_{\IA^2\setminus \pr_1^{-1}(0)} \cong (\IA^1\setminus 0) \times \CX^{(2)}, \quad 
\toCX|_{\IA^2\setminus \pr_2^{-1}(0)} \cong \CX^{(1)} \times (\IA^1\setminus 0). $$ 
We define $\CD|_{(\IA^1\setminus 0)^2}, \CZ|_{(\IA^1\setminus 0)^2}, \toCD$ and $\toCZ$ similarly. 
Note that 
$$(\CX^{(2)},\CD^{(2)},\xi_0)\to \CZ^{(2)}\to \IA^1$$ 
is a $\IT$-equivariant $\IQ$-Gorenstein family of K-semistable polarized log Fano fibration germs over $S=\IA^1$, and 
$$((\toCX,\toCD,\xi_0)\to\toCZ)|_{\IA^2\setminus \pr_2^{-1}(0)}$$ 
is a $\IT$-equivariant special test configuration of $(\CX^{(2)}_K,\CD^{(2)}_K,\xi_0)\to \CZ^{(2)}_K\supseteq (\IA^1\setminus 0)$ with vanishing $\BD$ (by Lemma \ref{Lemma. special degeneration with vanishing Ding is still K-semistable} since the central fiber of this special test configuration is K-semistable), where $K=\Ik(S) = \Ik(\IA^1)$. By Theorem \ref{Theorem. special degeneration of family}, the punctured family 
$$(\toCX,\toCD,\xi_0)\to \toCZ\to \IA^2\setminus (0,0)$$ 
extends to a family $(\CX,\CD,\xi_0)\to \CZ\to \IA^2$ with K-semistable central fiber, which is the desired common degeneration.  
\end{proof}

\begin{thm}[K-polystable degeneration]
\label{Theorem. polystable degeneration}
Let $(X,\D,\xi_0)\to Z\ni o$ be a K-semistable polarized log Fano fibration germ. Then there exists a unique K-polystable polarized log Fano fibration germ $(X_p,\D_p,\xi_0)\to Z_p\ni o$ which is a special degeneration of $(X,\D,\xi_0)\to Z\ni o$. 
\end{thm}

\begin{proof}

Uniqueness follows from Corollary \ref{Corollary. uniqueness of K-polystable degeneration}. 
For existence, assume that $(X,\D,\xi_0)$ admits a good $\IG_m^l$-action. Suppose that $(X,\D,\xi_0)$ is not K-polystable. Then there exists a non-product type special test configuration $(\CX^{(1)},\CD^{(1)},\xi_0)$ with vanishing $\BD$. Hence $(\CX^{(1)}_0,\CD^{(1)}_0,\xi_0)$ is also K-semistable, admitting a $\IG_m^{l+1}$-action. Running this program, we finally get a K-polystable polarized log Fano fibration germ $(\CX^{(k)}_0,\CD^{(k)}_0,\xi_0)$ since $l+k\le n$. 
\end{proof}

\subsection{K-semistable degenerations}

Let $v_0\in \Val_{X,o}^*$ be a minimizer of $\BH(X,\D)$. Then $\Gr_{v_0}R$ is finitely generated and $(X_0,\D_0,\xi_0) = (X_{v_0},\D_{v_0},\xi_{v_0})$ is a polarized log Fano fibration germ. Moreover, we have:

\begin{thm}[K-semistable degeneration]
\label{Theorem: K-semistable degeneration}
Let $v_0\in \Val_{X,o}^*$ be a weakly special valuation such that $\Gr_{v_0}R$ is finitely generated, and the corresponding degeneration $(X_0,\D_0) \to Z_0 \ni o$ is a log Fano fibration germ. Let $\xi_0$ be the co-weight vector on $X_0$ reading the $\lam$-grading of 
\begin{eqnarray*}
\Gr_{v_0} R= \bigoplus_{m\in l_0\IN} \bigoplus_{\lam \in \IR_{\ge 0}} \Gr_{v_0}^\lam R_m. 
\end{eqnarray*}
Then $v_0$ minimizes $\BH(X,\D)$ if and only if $(X_{0},\D_{0},\xi_{0})\to Z_0\ni o$ is K-semistable. 
\end{thm}

\begin{proof}
We follow the proof of \cite[Theorem 5.3]{HL20}. First assume that $v_0$ minimizes $\BH$. If $(X_{0},\D_{0},\xi_{0})$ is K-unstable, then there exists a special test configuration $(\CW,\D_{\CW},\xi_0,\eta)$ such that 
\begin{eqnarray*}  
\BD_{X_0,\D_0,\xi_0}(\CW,\D_{\CW},\eta) < 0. 
\end{eqnarray*}
Denote by $(Y,\D_Y,\eta)\coloneqq(\CW_0,\D_{\CW,0},\eta)$. Then
\begin{eqnarray*}  
\BD_{Y,\D_Y,\xi_0}(\eta)
= \BD_{X_0,\D_0,\xi_0}(\CW,\D_{\CW},\eta) 
< 0. 
\end{eqnarray*}
By Theorem \ref{Theorem. Replacing two-step dege. to one-step dege.}, there exists a series of special valuations $\{v_\vep \in \Val_{X,o}^* \}_{\vep \ge 0}$  inducing special degenerations of $(X,\D)$ with central fibers $(Y,\D_Y, \xi_0 +\vep\eta)$. 
Then 
\begin{eqnarray*}  
\BH_{X,\D}(v_\vep) 
= \BH_{Y,\D_Y}(\xi_0 +\vep\eta) 
= \log\big(
\int_\BP e^{-\la \alpha, \xi_0 +\vep\eta \ra} \DH_\BP(\dif \alpha) \big), 
\end{eqnarray*}
where the first equality follows from (\ref{Eqnarray. A_X(v) = A_Xv(wt_eta_v)}) and Lemma \ref{Lemma. DH_v = D_wt}. Hence
\begin{eqnarray*}
\frac{\dif}{\dif \vep}|_{\vep = 0} \,\,
\BH_{X,\D}(v_\vep)
=
\frac{-\int_\BP \la \alpha, \eta \ra\cdot e^{-\la \alpha, \xi_0\ra} \DH_\BP(\dif \alpha)}{\int_\BP e^{-\la \alpha, \xi_0\ra} \DH_\BP(\dif \alpha)}
= \BD_{Y,\D_Y,\xi_0}(\eta) 
< 0, 
\end{eqnarray*}
which contradicts that $v_0$ minimizes $\BH_{X,\D}$. 

Conversely, assume that $(X_0,\D_0,\xi_0)$ is K-semistable. Then $\BH(X_0,\D_0)$ is minimized by $\wt_{\xi_0}$ by Theorem \ref{Theorem. equivalence of stability notions}. For any filtration $\CF$ on $R$, recall that the initial term degeneration $\CF'$ of $\CF$ on $R_{v_0,m}=\Gr_{v_0}R_m$ is defined by 
\begin{eqnarray*}
\CF'^{\lam} R_{v_0,m}  
:= \{ \bar{s}\in R_{v_0,m} \mid s \in \CF^\lam R_m \}. 
\end{eqnarray*}
Note that the initial term degeneration of $\CF_{v_0}$ on $R$ is just $\CF_{\wt_{\xi_0}}$ on $R_{v_0,m}=\Gr_{v_0}R$. 
Hence 
$$\Gr_{\CF'}^\lam \Gr_{\wt_{\xi_0}}^\mu R_{v_0,m}\cong \Gr_{\CF}^\lam \Gr_{v_0}^\mu R_{m},\quad \text{and}\quad
S(\xi_0; \CF') = S(v_0;\CF). $$ 
By the lower semi-continuity of log canonical thresholds, we have
$\mu_{X,\D}(\CF) \ge \mu_{X_0,\D_0}(\CF'). $
We conclude that 
\begin{eqnarray}
\label{Inequality: H ineqality}
    \BH_{X,\D}(\CF) 
\ge \BH_{X_0,\D_0}(\CF')
\ge \BH_{X_0,\D_0}(\xi_0)
=   \BH_{X,\D}(v_0). 
\end{eqnarray}
\end{proof}

\begin{cor}
The $\BH$-invariant $\BH(X_0,\D_0)$ is minimized by $\wt_{\xi_0}$. Moreover, we have
\begin{eqnarray*}  
\BH(X_0,\D_0) = \BH(X,\D), \quad 
\delta_\IT(X_0,\D_0,\xi_0) = \delta(X,\D,v_0). 
\end{eqnarray*}
\end{cor}

Finally, we provide the proof of the main theorem of the paper. 

\begin{proof}[Proof of Theorem \ref{Theorem. Intro. stable degeneration of Fano fibration germ}]

By Theorem \ref{Theorem:  Existence}, there exists a weakly special valuation (hence quasi-monomial) $v_0\in\Val_{X,o}^*$ such that $\BH(v_0) = \BH(X,\D)$. Then any valuative $\BH$-minimizer $v_1\in\Val_{X,o}^*$ is equal to $v_0$ by Corollary \ref{Corollary: uniqueness of minimizer}. We see that (1) and (2) hold. By Theorem \ref{Theorem:  H-minimizer v_0 and delta^(g,v_0)}, we see that $v_0$ minimizes $\delta(X,\D;v_0)=1$. Hence $v_0$ is a special valuation, which has a finitely generated associated graded ring $\Gr_{v_0} R$ by Theorem \ref{Theorem: delta-minimizer is very special} and \ref{Theorem. H-minimizer has f.g. graded ring}. Moreover, by Theorem \ref{Theorem: delta-minimizer is very special}, the special degeneration 
$$(X_{0} = \Proj_{\Gr_{v_0}R_0}\Gr_{v_0}R ,\D_{0},\xi_0)\to Z_{0} = \Gr_{v_0}R_0 \ni o$$ 
induced by $v_0$ is a polarized log Fano fibration germ, which is K-semistable by Theorem \ref{Theorem: K-semistable degeneration}. We have (3) and (4). Finally, there exists a unique K-polystable special degeneration $(X_{p},\D_{p},\xi_0)\to Z_{p}\ni o$ of $(X_{0},\D_{0},\xi_0)\to Z_{0}\ni o$ by Theorem \ref{Theorem. polystable degeneration}. Hence (5) follows. The proof is finished. 
\end{proof}

\appendix

\bibliographystyle{alpha}
\bibliography{ref}

\end{document}